\documentclass[12pt, twoside,reqno]{amsart}

\usepackage{times,amssymb,amsmath,rotating,multirow,color,enumerate,url,mathrsfs}
\usepackage{arydshln}
\usepackage{bbm}
\usepackage{esint}
\usepackage{gensymb}
\usepackage{epstopdf}
\usepackage{comment}
\usepackage{leftidx}
\usepackage{amsthm}
\usepackage{mathtools}
\usepackage{enumerate}
\usepackage{subfig}
\usepackage{amsaddr}
\newtheorem{theorem}{Theorem}[section]
\newtheorem{corollary}[theorem]{Corollary}
\newtheorem{lemma}[theorem]{Lemma}
\newtheorem{proposition}[theorem]{Proposition}

\theoremstyle{definition}
\newtheorem{definition}[theorem]{Definition}
\newtheorem{remark}[theorem]{Remark}
\newtheorem{example}[theorem]{Example}

\numberwithin{equation}{section}

\newcommand{\R}{\mathbb{R}}
\newcommand{\Rd}{\R^d}

\newcommand{\MRd}{\mathcal{M}\left( \Rd \right)}
\newcommand{\prange}{[1,\infty)}
\newcommand{\Lpmuz}{L^{\,p}_{\mu}}
\newcommand{\Lpmudz}{\bigl(L^{\,p}_{\mu}\bigr)^d}
\newcommand{\Hpmuz}{H^{1,p}_{\mu}}
\newcommand{\Pmuz}{P_{\mu}}
\newcommand{\Tmuz}{T^p_{\mu}}
\newcommand{\tgradz}{D_{\mu}}
\newcommand{\Iaa}{{\left(a_-,a_+\right)}}
\newcommand{\Lloc}[1]{{L^1_{\mathrm{loc}}(#1)}}
\newcommand{\muw}{{\mu_w}}
\newcommand{\Leb}{{\mathcal{L}}}
\newcommand{\wcrit}{{\overline{w}}}

\newcommand{\Hnukp}[1]{{H^{#1,p}_{\mu_w}}}
\newcommand{\Hnuktwo}{{H^{2,2}_{\mu_w}}}

\newcommand{\Tmuw}{T^p_{\muw}}

\newcommand{\am}{{a_-}}
\newcommand{\ap}{{a_+}}
\newcommand{\Icr}[1]{{I^{#1}_{cr}}}
\newcommand{\cIcr}[1]{{\bar{I}^{\, #1}_{cr}}}
\newcommand{\cIcrl}[1]{{\bar{I}^{\, #1}_{cr,-}}}
\newcommand{\cIcrr}[1]{{\bar{I}^{\, #1}_{cr,+}}}
\newcommand{\tgradnuk}[1]{D^{#1}_{\mu_w}}
\newcommand{\Lpnu}{{L^{\,p}_{\mu_w}}}
\newcommand{\Pmuw}{P_{\mu_w}}
\newcommand{\noncrit}{{\muw\bigl( \Icr{p}(w) \bigr) = 0}}
\newcommand{\ustep}{{\hat{u}}}
\newcommand{\vstep}{{\hat{v}}}
\newcommand{\hphi}{{\hat{\phi}}}
\newcommand{\Hnukpa}[1]{{H^{#1,p}_{\mu_{w_\gamma}}}}
\newcommand{\Lpnua}{{L^{\,p}_{\mu_{w_\gamma}}}}
\newcommand{\tgradnuka}[1]{D^{#1}_{\mu_{w_\gamma}}}
\newcommand{\al}{\alpha}
\newcommand{\U}{{\mathcal{U}}}
\newcommand{\Ue}{{\mathcal{U}_\eps}}
\newcommand{\V}{{\mathcal{V}}}
\newcommand{\bk}{{\bar{k}}}
\newcommand{\Dk}{{\Delta k}}
\newcommand{\der}[2]{D^{#1}#2}
\newcommand{\dr}[2]{#2^{(#1)}}
\newcommand{\Tr}[2]{\mathrm{Tr}^{#1}\!#2\,}
\newcommand{\J}[1]{{J_{\,#1}}}
\newcommand{\Jc}[1]{{J^{*}_{\,#1}}}
\newcommand{\Bl}{{B_-}}
\newcommand{\Br}{{B_+}}
\newcommand{\cnoncrit}{{\muw\bigl( \cIcr{0,p}(w) \bigr) = 0}}

\newcommand{\Pro}{{\mathcal{P}}}
\newcommand{\Ione}{{I_1}}
\newcommand{\ut}{{\tilde{u}}}
\newcommand{\Dv}{{\Delta v}}

\newcommand{\pairing}[1]{{\left \langle #1 \right \rangle}}
\newcommand{\norm}[1]{\Arrowvert #1 \Arrowvert}
\newcommand{\abs}[1]{{\left \lvert #1 \right \rvert}}

\newcommand{\eps}{\varepsilon}
\newcommand{\Rb}{\overline{\mathbb{R}}}
\newcommand{\G}{\mathcal{G}}
\newcommand{\D}{\mathcal{D}}
\newcommand{\C}{\mathcal{C}}
\newcommand{\Ha}{\mathcal{H}}

\newcommand{\mres}{\mathbin{\vrule height 1.6ex depth 0pt width
		0.13ex\vrule height 0.13ex depth 0pt width 1.3ex}}
\newcommand{\Lpmu}{L^{\,p}_{\mu_\G}}

\begin{document}

	\title[Higher order weighted Sobolev spaces for strongly degenerate weights]{Higher order weighted Sobolev spaces on the real line for strongly degenerate weights. Application to variational problems in elasticity of beams}

	
	\author{Karol Bo{\l}botowski}
	
	\address{Department of Structural Mechanics and Computer Aided Engineering, Faculty of Civil Engineering, Warsaw University of Technology, 16 Armii Ludowej Street, 00-637 Warsaw, \linebreak
	College of Inter-Faculty Individual Studies in Mathematics and Natural Sciences, University of Warsaw, 2C Stefana Banacha St., 02-097 Warsaw
	}
	\curraddr{}
	\email{k.bolbotowski@il.pw.edu.pl}
	\thanks{}
	
	
	\date{\today}
	
	\dedicatory{}
	
	\begin{abstract}
		For one-dimensional interval and integrable weight function $w$ we define via completion a weighted Sobolev space $\Hnukp{m}$ of arbitrary integer order $m$. The weights in consideration may suffer strong degeneration so that, in general, functions $u$ from $\Hnukp{m}$ do not have weak derivatives. This contribution is focussed on studying the continuity properties of functions $u$ at a chosen internal point $x_0$ to which we attribute a notion of criticality of order $k$ and with respect to the weight $w$. For non-critical points $x_0$ we formulate a local embedding result that guarantees continuity of functions $u$ or their derivatives. Conversely, we employ duality theory to show that criticality of $x_0$ furnishes a smooth approximation of functions in $\Hnukp{m}$ admitting jump-type discontinuities at $x_0$. The work concludes with demonstration of established results in the context of variational problem in elasticity theory of beams with degenerate width distribution.
	\end{abstract}
	
	\maketitle


	\section{Introduction}
	\label{intro}
	
	A basic design problem in structural mechanics is to optimally construct an elastic beam -- a horizontal, one-dimensional body that by means of bending transfers a given vertical load to the kinematical supports. Our design should occupy an interval $I = (a_-,a_+) \subset  \R$; the depth of the beam $h_0$ and the material characterized by the Young modulus $E_0$ may be assumed constant and fixed, while we vary the non-negative width distribution of the beam $w: I \rightarrow \R_+\cup \{0\}$. Our limitation is the prescribed total volume of the beam given by $\int_I h_0 \, w(x)\, dx \leq V_0$. Assuming the linearly elastic model of the beam, its deflection function $u:I \rightarrow \R$ is formally governed by the 4-th order elliptic equation $ D^2\bigl((\mathcal{E}_0 w ) D^2 u \bigr) = f$ where $\mathcal{E}_0 = E_0 h_0 /12 $ and the distribution $f \in \D'(I)$ describes the loading. The classical problem is to find the optimal width function $\hat{w} \in L^1(I)$ that minimizes the so-called compliance (potential energy of the system) for a single point force applied in the centre (expressed by the dirac Delta measure $f = \hat{f} =F\, \delta_{(\am,\ap)/2}$) in a clamped beam (namely with kinematical supports formulated through homogeneous Dirichlet boundary conditions). This optimization problem falls within the scope of mathematical theory of optimal shape and mass distribution design that was put forward for the $d$-dimensional setting by \cite{bouchitte2001}, \cite{bouchitte2007} and \cite{bouchitte2008}, where the design variable was a Radon measure $\mu \in \MRd$ that represented the mass support of the target structure. The other work of the present author, \cite{bolbotowski2020b}, localizes this theory for one-dimensional second-order problem (and also for a problem on a graph) where we can limit our search to integrable non-negative functions $w \in L^1(I)$ representing width. This fact is long known at the formal level and first papers on the width optimization in beams date back to late '50s, see for instance \cite{heyman1959}, \cite{rozvany1976}, \cite{prager1977}. Therein derived, a "candy-shaped" optimal design for the problem of the clamped beam loaded at the centre is displayed in Fig. \nolinebreak \ref{fig:clamped_beam}(a).
	\begin{figure}[h]
		\centering
		\subfloat[]{\includegraphics*[width=0.45\textwidth]{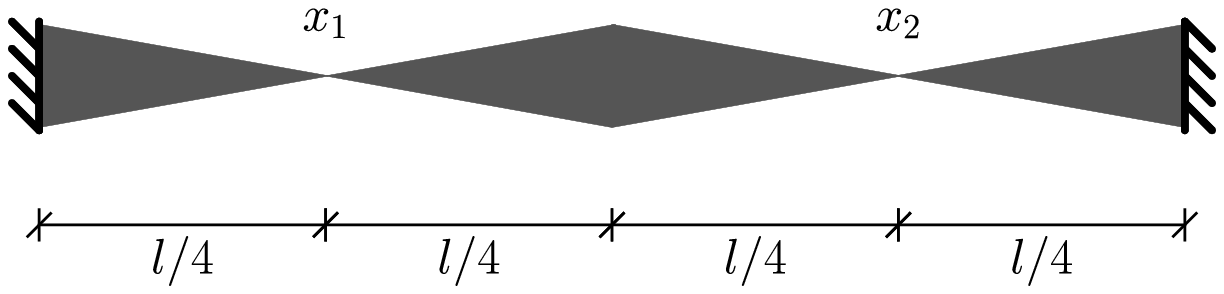}}\hspace{1cm}
		\subfloat[]{\includegraphics*[width=0.45\textwidth]{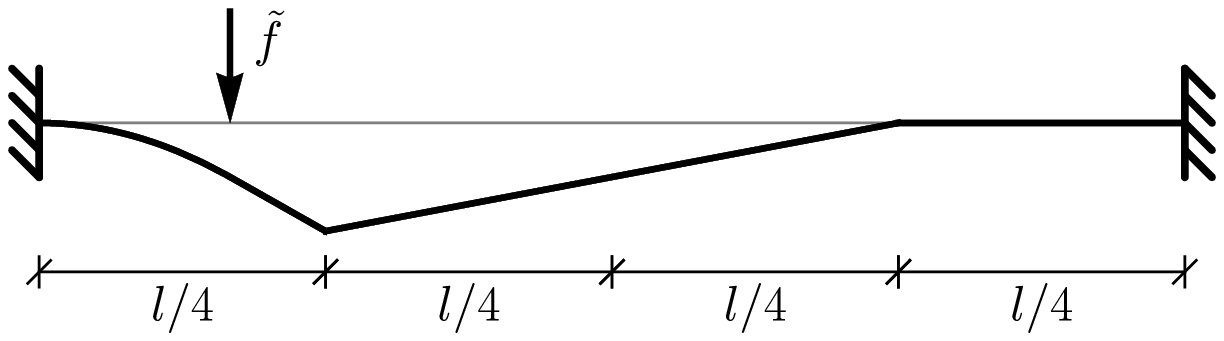}}\\
		\caption{(a) an optimal width distribution $\hat{w}$ in a clamped beam for a point-force applied in the centre (view from the top); (b) deformation of the optimal beam under a shifted load (side view); $l = \abs{\am - \ap}$ denotes the length of the beam.}
		\label{fig:clamped_beam}
	\end{figure}
	
	The beam with the optimal width $\hat{w}$ ought to serve as an elastic structure carrying a load $f = \tilde{f}$ that does not necessarily coincide with the load $\hat{f}$ for which it was designed, see Fig. \nolinebreak \ref{fig:clamped_beam}(b). The function $\hat{w}$ admits singularities at the two points $x_1, x_2$, thus the degenerate differential equation "$D^2\bigl( \hat{w} D^2 u \bigr)= \tilde{f}$" requires special treatment (henceforward we assume that $\mathcal{E}_0 = 1$). The natural approach involves the variational formulation 
	\begin{equation}
	\label{eq:beam_primal}
	\inf \biggl\{ J\bigl(D^m u\bigr) - \bigl \langle u,\tilde{f} \bigr \rangle \ : \ u \in \D(I) \biggr\}, \quad \text{where} \quad J(v) = \frac{1}{p} \int_I w \, \abs{v}^p\, dx,
	\end{equation}
	with $w = \hat{w}$, $m=2$ and $p=2$; $v$ may be any element from $L^p(I)$. A reasonable extension of $\D(I)$ to a Banach space must be proposed along with a lower semi-continuous relaxation of the convex functional $J\bigl(D^m \, \cdot \, \bigr)$ that ought to be coercive in this space. Treating the function $w$ as a weight inevitably we shall find ourselves in a version of weighted Sobolev space. Such space may be variously defined: in the pioneering work of \cite{kufner1984} the definition relies on the notion of weak derivatives, while the Sobolev norm includes norms in weighted Lebesgue space $L^p_w(I)$. The main result of the paper states that such weighted Sobolev space is complete if and only if $1/w^{1/(p-1)} \in \Lloc{I}$. The latter, so-called $B_p$-condition furnishes the essential embedding $L^p_w(I) \hookrightarrow \Lloc{I}$; the reader may also compare \cite{opic1989}. The works, for instance, \cite{kilpelainen1997}, \cite{gol2009}, or \cite{cavalheiro2008weighted} rest upon a stronger $A_p$ (or Muckenhoupt) condition which guarantees that the complete weighted Sobolev space may be indifferently defined via weak derivatives or completion of the space of smooth functions. The optimal width function $\hat{w}$ is easily checked to violate both conditions $A_p$ and $B_p$ in case of $p=2$, which is due to degeneration of $\hat{w}$ around $x_1,x_2$ at linear rate. Weights $w$ that verify $w^{1/(p-1)} \notin \Lloc{I}$ shall be the title \textit{strongly degenerate weights} and for those weights we are forced to define the weighted Sobolev space directly by completion.
	
	The relaxation of energy functionals of the form $J\bigl(D \, \cdot \, \bigr)$ (in the first-order case $m= \nolinebreak1$) was the topic of paper by \cite{bouchitte1997}, except that the integration in $J$ was carried out with respect to arbitrary compactly supported Radon measure $\mu$ on $\R^d$. By means of measures we may capture not only degeneracies of density of a structure, but also the somewhat opposite singularities in the form of lower-dimensional elements such as curves or surfaces. The central idea of the work revolves around definition of the space tangent to measure $\mu$ at a point $\Tmuz(x) \subset \R^d$ and the gradient tangential to measure $D_\mu u$ which for smooth functions is almost everywhere computed as orthogonal projection of classical gradient $D u (x)$ onto $\Tmuz(x)$. The completion of the space of smooth functions with respect to norm $\norm{u}_{\Lpmuz}+ \norm {D_\mu u}_{\Lpmuz}$ receives the name of Sobolev space with respect to measure $\Hpmuz$. In the work \cite{bouchitte1997} the space $\Hpmuz$ is employed to cope with geometrical measures of the form $\mu = \sum_{i=1}^n w_i \, \Ha^{k(i)} \mres S_i$, where $S_i$ is a $k(i)$-dimensional $C^2$-manifold and the weight $w_i$ is piece-wise constant. Simultaneously authors prepare the background for the later paper \cite{bouchitte2001}: they have understood that, since the optimal structure in $\Rd$ ends up being a measure $\mu$, the theory of elasticity of such structures must be first developed. The present work together with a more structural-mechanics-oriented paper \cite{bolbotowski2020a}, play an analogical role for the work on beam and grillage optimization in \cite{bolbotowski2020b}. We shall conveniently utilize the theory of the Sobolev space with respect to measure in one-dimensional interval $I \subset \R$: we choose $\mu = \muw := w \, \Leb^1 \mres I$ and the first-order weighted Sobolev space defined by completion is at our disposal as the space \nolinebreak $\Hnukp{1}$.
	
	The idea of the weighted Sobolev space $\Hnukp{1}$, however, is not original and has been already applied in the paper by \cite{louet2014}. Its main result focuses on characterization of the tangent space $\Tmuw(x)$: it trivially equals $\{0\}$ for a.e. point $x$ in the so called \textit{critical set} $\Icr{p}(w)$ and becomes full space $\R$ whenever the point $x \in I$ lies outside this set. The elements of the closed set $\Icr{p}(w)$ are precisely those points $x_0$ that for every $\eps>0$ yield $\int_{B(x_0,\eps)} 1/w^{1/(p-1)} dx = \infty$, for instance $x_1, x_2 \in \Icr{2}(\hat{w})$ in Fig. \ref{fig:clamped_beam}(a). A clear link with the $B_p$-condition has allowed the author to infer that every function $u \in \Hnukp{1}$ is an element of the classical Sobolev space $W^{1,1}_\mathrm{loc}\bigl( I \backslash \Icr{p}(w) \bigr)$ and that the tangential derivative $\tgradnuk{} u \in \Lpnu$ equals the distributional derivative $D u$ in the open set $I \backslash \Icr{p}(w)$. In particular $u$ is continuous outside $\Icr{p}(w)$, while in each critical point jump discontinuity may occur possibly rendering $D u$ an irregular distribution in the whole interval $I$.
	
	The elasticity problem of the beam with the width distribution $w \in L^1(I)$ requires handling a second-order weighted Sobolev space. The topic of second-order differentiation with respect to measures appeared in \cite{bouchitte2003}, again in a broader setting of $d$-dimensional space where additional issues arise -- we are forced to manipulate an independent Cosserat field that for smooth functions corresponds to $D^\perp_\mu u$, namely the part of the gradient that is orthogonal to $\mu$. On top of that the authors assume a Poincar\'{e}-like inequality condition on the measure $\mu$ that for the strongly degenerate weights considered herein clearly cannot hold, since in particular we allow $w$ to vanish on sets with non-zero Lebesgue measure. We begin our effort in Section \ref{sec:Sobolev_real_line_def} where, upon imposing a very mild assumption on weights (that are met by any $w \in BV(I)$) we inductively define the weighted Sobolev space of any order $\Hnukp{m}$ drawing upon the already developed theory of $\Hnukp{1}$: the elements $u$ in $\Hnukp{m}$ are, roughly speaking, those functions $u \in \Hnukp{m-1}$ for which $\tgradnuk{m-1} u \in \Hnukp{1}$. The lower semi-continuous regularization of $J\bigl( D^m \, \cdot \, \bigr): \D(I) \rightarrow \R$ now reads $J\bigl(\tgradnuk{m}\,\cdot\,\bigr) : \nolinebreak \Hnukp{m} \rightarrow \R$, although the proof of this simple fact was moved to \cite{bolbotowski2020a}. Based on the results of \cite{louet2014}, the inductive definition of the space $\Hnukp{m}$ allows to expect that $\Hnukp{m} \hookrightarrow W^{m,1}_\mathrm{loc}\bigl( I \backslash \Icr{p}(w) \bigr)$, whereas in the set $\Icr{p}(w)$ the functions $u \in \Hnukp{m}$ may suffer discontinuities of the tangential derivative $\tgradnuk{m-1} u$. The big question of this work concerns (dis)continuity of the lower derivatives: how to judge, for instance in case of $m= \nolinebreak 2$, whether at a given point $x_0 \in I$ the function $u \in \Hnukp{2}$ itself has to be continuous or may admit jump-type discontinuity instead?
	
	For the tools that examine the continuity of functions $u \in \Hnukp{m}$ we seek inspiration in structural mechanics. Up to change of the sign, a version of the problem dual to \eqref{eq:beam_primal} reads
	\begin{equation}
	\label{eq:beam_dual}
	\inf \left\{ J^*(M)  :  M \in L^{p'}\!(I), \ D^m M + \tilde{f} =0 \right\} , \quad \text{where} \quad J^*(M) = \frac{1}{p'} \int_I \frac{\abs{M}^{p'}}{w^{\,p'/p}}\,dx
	\end{equation}
	which upon localization for $m=2$ and $p=2$ gives a classical stress-based variational formulation for linear elasticity in beams; the equation $D^m M +\tilde{f} = 0$ must be understood in the distributional sense. In the context of a beam, $M$ is the so-called bending moment function that point-wise describes the stress; the Legendre-Fenchel transform $J^*(M)$ quantifies the complementary elastic energy of the beam. A key remark to make is that the transform was deliberately derived with respect to the duality pairing $\bigl\langle L^p(I),L^{p'}(I) \bigr\rangle$ and not, seemingly natural, pairing $\bigl\langle\Lpnu,L^{p'}_\muw \bigr\rangle$. It is clear that the solution of the problem \eqref{eq:beam_dual} must be sought among those $M$ that furnish finite energy $J^*(M)$. In case of the optimal beam from Fig. \ref{fig:clamped_beam}(a) where for the two singular points $x_i = x_1, x_2$ the integral $\int_{B(x_i,\eps)} 1/\hat{w} \, dx$ equals infinity with arbitrarily small $\eps$, the candidate bending moment $M$ must necessarily tend to zero at those points in case when $p=2$. In mechanics of beam and frame systems it is well-established (the reader is encouraged to look into Chapter I in \cite{lewinski2019} or the book by \cite{rozvany1976}) that enforcing zero bending moment "at a point" can be structurally realized by inserting the so-called hinge -- this may be treated as the very definition of a hinge in its stress (or dual) version. Primally, or kinematically, a hinge allows a jump in rotation that is represented by the derivative of the displacement function $u$: should the hinges be assumed at $x_1$ and $x_2$, the formal methods of structural mechanics deliver the solution of the equation "$D^2\bigl( \hat{w} D^2 u \bigr)= \tilde{f}$" that is displayed in Fig. \ref{fig:clamped_beam}(b). For an arbitrary width/weight $w \in L^1(I)$ this engineering reasoning coincides with the mathematical results given earlier: for a function $u \in H^{2,2}_{\muw}$ its tangential derivative $\tgradnuk{} u$ may admit jumps at critical points $x_0 \in \Icr{2}(w)$ or, by definition, points yielding $\int_{B(x_0,\eps)} 1/w \, dx = \infty \ \ \forall \eps>0$. Encouraged, we dig deeper into the duality links known in mechanics of beams. We wish to verify when at a point $x_0$ a function $u \in H^{2,2}_{\muw}$ itself may admit a jump. In terms of kinematics, a beam with such a deformation $u$ disconnects at $x_0$ entirely and thus, dually, no force interaction can occur. Apart from the bending moment $M$ the beam is also subject to action of the shear force that is defined distributionally through $T := D M$. Therefore, according to mechanics, "at the point $x_0$" where $u$ jumps the derivative $D M$ must vanish, in particular it should be infeasible to have $M(x) = x - x_0$ in a neighbourhood of $x_0$. We recall that the condition forcing zero bending moment at $x_0$ has been above represented in two languages: i) structurally as a hinge; ii) mathematically through the dual variational problem \eqref{eq:beam_dual} as the condition $\int_{B(x_0,\eps)} 1/w \, dx =\infty \ \ \forall \eps>0$. The constraint on the shear force $T$, or equivalently disqualification of the bending moment being locally equal to $M(x) = x - x_0$, was interpreted in terms of mechanics as a "full cut" at $x_0$. This is an approach i), by analogy in the mathematical setting ii) we must introduce a criterion $\int_{B(x_0,\eps)} \frac{\abs{x-x_0}^2}{w(x)} \, dx = \infty \ \ \forall \eps>0$. If we are to trust the mechanics-based reasoning, the latter condition should therefore allow a jump-type discontinuity of a function $u \in H^{2,2}_{\muw}$. For this integral condition to hold at $x_1$ or $x_2$ in the optimal beam from Fig. \ref{fig:clamped_beam}(a), the width $\hat{w}$ would need to degenerate around those points at least at the rate of $\abs{x-x_0}^3$. This work is essentially aimed at rigorous verification of this idea, also for arbitrary exponent $p \in \prange$ and order $m \in \mathbb{N}_+$.
	
	The excursion that we had through duality in mechanics of structures suggests that in order to examine continuity of functions $u$ in higher order weighted Sobolev space $\Hnukp{m}$ we must extend the definition of the critical set to any order $\al \geq 0$. For $p \in (1,\infty)$ it shall read
	\begin{equation*}
	\label{eq:def_Icrap_0}
	\cIcr{\al,p}(w) := \biggl\{ x_0 \in \bar{I} : \forall\, \eps>0  \int\limits_{\bar{I} \cap B(x_0,\eps)} \! \left(\frac{\abs{x-x_0}^\al}{\bigl(w(x)\bigr)^{1/p}}\right)^{p'} \! dx = \infty \biggr\},
	\end{equation*}
	namely the higher order $\al$, the faster degeneracy of $w$ around $x_0$ is required for the point $x_0$ to belong to $\cIcr{\alpha,p}(w)$. The results of the present paper may be loosely summed up as follows: for an interval $I = \Iaa$, a weight $w \in L^1(I)$ satisfying $\muw\bigl( \cIcr{0,p}(w) \bigr)=0$, given an order $m \in \mathbb{N}_+$ and an exponent $p \in \prange$ there hold
	\begin{enumerate}[(i)]
		\item if $x_0 \notin \cIcr{m-1,p}(w)$ then every function $u \in \Hnukp{m}$ has a $\muw$-a.e. equal representative that is continuous (Section \ref{sec:sufficient_conditions});
		\item if \textit{stability} of the weight $w$ is assumed, then $x_0 \in \cIcr{m-1,p}(w)$ implies that a step function $\ustep = \mathbbm{1}_{(x_0,\ap)}$ is an element of $\Hnukp{m}$ (Section \ref{sec:necessary_conditions});
		\item for stable weights the trace operator defined on the space of smooth functions as $\Tr{}\, u\,(a_+) := u(a_+)$ extends continuously to the space $\Hnukp{m}$ if and only if $\ap \notin \nolinebreak \cIcr{m-1,p}(w)$ and the same applies to $\am$ (Section \ref{sec:trace_extension}).
	\end{enumerate}
	
	The point (i) will be obtained by rather standard method: its core is the embedding $\Hnukp{m-1} \hookrightarrow \Lloc{ I \backslash \cIcr{m-1,p}(w)}$ given in Theorem \ref{thm:embedding} which, upon acknowledging the inductive definition of the higher order Sobolev space $\Hnukp{m}$, in turn yields $\Hnukp{m} \hookrightarrow W^{1,1}_\mathrm{loc}\bigl( I \backslash \cIcr{m-1,p}(w) \bigr)$. The trick behind the first embedding is almost the very same as in the proof of $\Lpnu \hookrightarrow \Lloc{ I \backslash \cIcr{0,p}(w)}$ from the work of \cite{kufner1984}, except that additional, quite simple estimate of $\int_{B(x_0,\eps)} \abs{\phi} dx$ by the integral $\int_{B(x_0,\eps)} \abs{D^{m-1} \phi (x)}\abs{x-x_0}^{m-1} dx$ is first needed for smooth functions $\phi$ with compact support in $B(x_0,\eps)$.
	
	The central part of the work revolves around the point (ii): it is here that we draw upon the theory of beam structures and utilize the Legendre-Fenchel transformation between the energy functional $J(v) = \frac{1}{p} \int_I w \, \abs{v}^p\, dx$ and, recalling that the duality pairing is chosen as $\pairing{v,v^*}:= \int_I v \, v^* \, dx$, the functional  $J^*(v^*) = \frac{1}{p'} \int_I \frac{\abs{v^*}^{p'}}{w^{\,p'/p}}\,dx$. Since the space $\Hnukp{m}$ is defined via completion of the space of smooth functions, proving that $\ustep = \mathbbm{1}_{(x_0,\ap)}$ belongs to $\Hnukp{m}$ requires pointing to a sequence $\ustep_h$ of smooth functions that converges to $\ustep$ in the $\Hnukp{m}$-norm. We will make an effort to show that this is possible only if all the tangential derivatives $\tgradnuk{k} \ustep$ for $k\in \{1,\ldots,m\}$ are zero in $\Lpnu$, which is non-intuitive as the first distributional derivative $D \ustep$ in the domain $I$ equals Dirac delta measure at $x_0$. Loosely speaking, if indeed $\ustep \in \Hnukp{m}$, all the distributional derivatives $D \ustep,\ldots,D^m \ustep$ must be killed by the weight $w$ degenerating around $x_0$: the higher the order $m$ the faster the weight must degenerate which, as we shall show, is incorporated in the condition $x_0 \in \cIcr{m-1,p}(w)$. Our technique will be to find a sequence of smooth functions $\hphi_h$ that approximates Dirac delta measure at $x_0$; the sequence that we originally seek may be then defined as $\ustep_h(x) := \int_{-\infty}^x \hphi(y) \, dy$. The full success comes when all the derivatives $D \hphi,\ldots, D^{m-1}\hphi$ converge to zero in $\Lpnu$. The problem of finding such a sequence $\hphi$ will be reformulated as a variational problem that involves the energy functional $J$. At this point, in Theorem \ref{thm:dirac_delta_approx_crit_point}, general duality theory comes into the picture and a dual variational problem emerges where we minimize the conjugate functional $J^*(v^*)$ over functions $v^* \in L^{p'}(I)$ satisfying the distributional constraint $D^{m-1} v^* \geq 1$. We spot that this constraint together with the formula for $J^*$ lie closely to definition of the set $\cIcr{m-1,p}(w)$ and ultimately we infer that the dual infimum must be non other than infinity for $x_0$ in this set. Through a chain of equivalences we infer existence of the sequence $\hphi$ that approximates Dirac delta at $x_0$ and verifies $\norm{D^{m-1}\hphi}_\Lpnu \rightarrow 0$, yet only for derivative of the highest order considered. To obtain convergence of $D^k \hphi$ to zero for lower $k \geq 0$ a Poincar\'{e}-like inequality must be recovered in some neighbourhood of $x_0$. For that purpose an additional assumption on the weight $w$ is needed and we decide to propose a condition that we call \textit{stability}: for every critical point $x_0$ the degeneration to zero is enforced to be monotonic in some neighbourhood of $x_0$. We have, in fact, sketched the proof of the point (ii). The last part of Section \ref{sec:necessary_conditions} is devoted to provide some additional insight into relations between: 1) \nolinebreak criticality of a point $x_0$; 2)  \nolinebreak occurrence of the step function $\ustep = \mathbbm{1}_{(x_0,\ap)}$ in the weighted Sobolev space $\Hnukp{m}$; 3) \nolinebreak a \nolinebreak series of variational problems and their duals. The true purpose of Theorem \ref{thm:jump_summary} put forward therein, aside from its summarizing nature, is to justify the extra stability assumption enforced on the weight. The rather long proof of the theorem ends with an example of a non-stable weight $\wcrit \in L^\infty\bigl( (-1,1) \bigr)$ such that $x_0 =0 \in \Icr{1,p}(w)$ and notwithstanding this we show that $\ustep = \mathbbm{1}_{(0,1)} \notin \Hnukp{2}$. It will appear that the choice of $\wcrit$ is not trivial as it must admit some cunning oscillation about the point $x_0$. Eventually we establish that for $w \in L^\infty(I)$ the point (ii) does not hold in full generality and \textit{some} assumption is essential to eliminate the varying of the weight. Upon realizing that the condition $w \in BV(I)$ does not suffice, we can in good conscience hold on to the proposed assumption of stability, i.e. local monotonicity around critical points.
	
	The point (iii) virtually builds upon results from Sections \ref{sec:sufficient_conditions} and \ref{sec:necessary_conditions} where the key theorems were adopted for the scenario of $x_0$ being one of the boundary points $\am$ or $\ap$. In Section \nolinebreak \ref{sec:trace_extension} we also put forward Theorem \ref{thm:trace_approximation} that allows to approximate $u \in \Hnukp{m}$ by a smooth $u_\eps$ that has prescribed boundary values of the function and all its derivatives at $\am$ and $\ap$. This statement will play a role of a lemma in the next work \cite{bolbotowski2020a} from the present author -- it will allow constructing a smooth approximation of a function in weighted Sobolev space defined on the graph in $\Rd$. The work concludes with Section \nolinebreak \ref{sec:conclusions} where we revisit the problem of elasticity in beams and in detail we demonstrate how to apply the developed theory of weighted Sobolev spaces $\Hnukp{m}$ to solving variational problems where the minimized energy functional is a degenerate, weighted integral. 
	\vspace{2mm}
	
	\noindent \textbf{Notation:} Although throughout the text we tend to remind the notation, we agree to some most basic symbols here already. By $\D(\U)$ and $\D'(\U)$ we will denote the space of compactly supported test function and distributions in an open set $\U$. For differentiation of order $k$ we use symbol $D^k u$ indifferently on the real line $\R$ or in $d$-dimensional space $\Rd$, both for classical differentiation and the distributional one. With $\am,\ap\in\R$ we will denote an open interval $I=\Iaa \subset \nolinebreak \R$. For the exponent $p \in [1,\infty]$, $p'=p/(p-1)$ will stand for its H\"{o}lder conjugate. By $L^p$ and $L^p_\mu$ we understand standard Lebesgue spaces with respect to Lebesgue measure and, respectively, arbitrary compactly supported Radon measure $\mu$. For a subset $A \in \Rd$ the symbol $\mathbbm{1}_A$ will denote the characteristic function of $A$, while for the indicator function we will use $\mathbb{I}_A$. The set of positive natural numbers will be written as $\mathbb{N}_+$ whilst $\mathbb{N}$ shall include zero.
	
	\section{Definition of higher order weighted Sobolev space on the real line for strongly degenerate weights}
	
	\label{sec:Sobolev_real_line}
	
	\subsection{Few words on defining weighted Sobolev space via weak derivatives}
	\label{sec:Sobolev_real_line_intro}
	
	Given a domain $\Omega \subset \Rd$ and a weight that is a non-negative measurable function $w:\Omega \rightarrow \R_+ \cup \{0\}$, a natural way of defining a weighted Sobolev space employs the notion of weak derivatives. One of the pioneering discussions on the correctness of such definition depending on the weight $w$ may be found in \cite{kufner1984}. In this setting we say that a measurable function $u:\Omega \rightarrow \R$ belongs to a weighted Sobolev space $W^{1,p}_w(\Omega)$ for $p \in \prange$ if and only if: $u \in L^p_w(\Omega) \cap \Lloc{\Omega}$ and the distributional derivative $D u \in L^p_w(\Omega;\Rd)$. The symbol $L^p_w(\Omega)$ stands for the weighted Lebesgue space endowed with the norm $\norm{u}_{L^p_w(\Omega)} =\bigl( \int_\Omega w(x) \abs{u(x)}^p dx \bigr)^{1/p}$. Consequently $W^{1,p}_w(\Omega)$ is also a normed space with $\norm{u}_{W^{1,p}_w(\Omega)} = \bigl(\norm{u}_{L^p_w(\Omega)}^p+\norm{D u}_{L^p_w(\Omega)}^p \bigr)^{(1/p)}$.
	
	Such Sobolev space may not be complete, unless we impose a condition that controls the level of the weight's degeneracy. In \cite{kufner1984} we find a criteria for the completeness of $W^{1,p}_w(\Omega)$ that is called a $B_p$-condition and for $p \in (1,\infty)$ it reads
	\begin{equation}
	\label{eq:Bp_condition}
	\left(B_p\right): \qquad  \frac{1}{w^{p'/p}} = \frac{1}{w^{1/p-1}} \in \Lloc{\Omega},
	\end{equation}
	where $p'$ is the H\"{o}lder conjugate exponent to $p$. Henceforward we will agree that for $\alpha \geq 0$ and $\beta = 0$ we have $\alpha/\beta = 0$ if $\alpha = 0$ and $\alpha/\beta = \infty$ if $\alpha > 0$. Therefore, for any $p\in (1,\infty)$, the $B_p$-condition \eqref{eq:Bp_condition}  implies that the weight $w$ is almost everywhere positive. The condition may be extended to the case of $p = 1$ where we shall understand that it holds if and only if for every compact set $K \subset \Omega$ the weight $w$ is essentially bounded from below by a positive constant $C = C(K) >0$.
	
	We arrive at an implication that happens to be crucial for proving the completeness of $W^{1,p}_w(\Omega)$ for weights $w$ satisfying the $B_p$-condition:
	\begin{equation}
	\label{eq:standard_embedding}
	\frac{1}{w^{p'/p}} \in \Lloc{\Omega} \qquad \Rightarrow \qquad L^p_w(\Omega) \hookrightarrow \Lloc{\Omega}.
	\end{equation}
	The above easily follows from the H\"{o}lder inequality; we display the estimate below for we will repeatedly use a variation of it. Let $K$ be any compact set contained in $\Omega$, then for any measurable function $u$
	\begin{equation}
	\label{eq:B_p_estimate}
	\int_K \abs{u} \, dx = \int_K \biggl(w^{1/p}\abs{u}\biggr) \biggl(\frac{1}{w^{1/p}}\biggr)  dx \leq \left( \int_K w\, \abs{u}^p dx\right)^{1/p} \left( \int_K \frac{1}{w^{p'/p}} dx\right)^{1/p'},
	\end{equation}
	which is valid also for $p=1$ provided the last factor is rewritten as $\norm{1/w}_{L^\infty(K)}$.
	
	It is worth mentioning that in \cite{kufner1984} we find some denseness results for the space of smooth functions in the, adequately defined, space $W^{1,p}_{w,0}(\Omega)$. Furthermore, for instance in \cite{gol2009}, a stronger condition on the degeneracy of $w$ is imposed, which is called a Muckenhoupt or $A_p$-condition. It allows us to indifferently define weighted Sobolev space through both weak derivatives and completion of the space of smooth functions, namely the $A_p$-condition yields $W^{1,p}_{w}(\Omega) = H^{1,p}_{w}(\Omega)$.
	
	\subsection{The notion of Sobolev space with respect to measure as a point of departure in defining the higher order weighted Sobolev spaces for strongly degenerate weights}
	\label{sec:Sobolev_real_line_def}
	
	Throughout the rest of the Section \ref{sec:Sobolev_real_line} we will work in an open bounded interval on the real line denoted by $I = \Iaa \subset \nolinebreak \R$. We assume a weight that is a non-negative integrable function, i.e. $w \in L^1(I)$. The exponent $p$ may be any real number from $\prange$. For a function $u \in C^k(I)$ we will denote its derivative by $D^k u$ and treat it again as a scalar function on $\R$. By $\D(\R)$ we will denote the space of smooth functions compactly supported in $\R$: in particular $u \in \D(\R)$ may not vanish on the boundary of $I$. 
	
	The setting that we put ourselves in rules out the possibility of defining weighted Sobolev space through weak derivatives. In extreme case we allow the weight $w$ to vanish on subsets of $I$ of non-zero Lebesgue measure, for instance on some subinterval, which clearly violates the $B_p$-condition \eqref{eq:Bp_condition}. This scenario is, however, easy to handle, since such a subinterval clearly splits the domain $I$ into two. We shall be more concerned with the case when the weight degenerates to zero around a certain point $x_0 \in I$, e.g. $w(x) = \abs{x-x_0}^\gamma$ with $\gamma>0$. The weights $w \in L^1(I)$ that do not satisfy the $B_p$-condition will be herein referred as the title \textit{strongly degenerate weights}. Upon defining a weighted Sobolev space for such weight, it will be of particular interest to examine the continuity conditions for functions belonging to this space.
	
	For the weights that do not satisfy the $B_p$-condition the suitable method for defining weighted Sobolev spaces is by completion of the space of smooth functions. Such approach falls into a particular theory of Sobolev spaces with respect to measure that was originated in \cite{bouchitte1997}. It is based on the idea of space tangent to a measure at a point and, the inextricably linked, notion of tangential derivative (gradient). Below we quickly review the basics of the theory drawing upon a later work \cite{bouchitte2003}, where the tangent space at $x$ stems from the firstly-defined space that is normal to measure at this point. For a moment we will work with an arbitrary Radon measure $\mu \in \mathcal{M}_+(\Rd)$ in $d$-dimensional space $\R^d$ in order to grasp the true purpose of the theory. Next we will localize it for the one-dimensional setting and thus for a more comprehensive coverage in case of wider classes of measures $\mu$ in $\Rd$ the reader is referred to the aforementioned works and others: \cite{fragala1999}, \cite{bouchitte2002}, \cite{rybka2018}.
	
	For any $p \in \prange$ by $L^p_\mu\bigl(\R^d;V \bigr)$ we see the standard $V$-valued Lebesgue space with respect to $\mu$; we agree for the following abbreviations: $\Lpmuz := L^p_\mu\bigl(\R^d;\R \bigr)$ and $\Lpmudz := L^p_\mu\bigl(\R^d;\R^d \bigr)$. Independently of $V$ the norm in $L^p_\mu\bigl(\R^d;V \bigr)$ shall be denoted by $\norm{\, \cdot \,}_{\Lpmuz}$. We start by defining the linear space $G := \bigl\{ \left(u, \nabla u\right) \ : \ u\in \D(\Rd) \bigr\}$, where $\D(\Rd)$ denotes the space of compactly supported smooth functions. Below by $\overline{G}$ we will see the closure of $G$ in the Cartesian product $\Lpmuz \times \Lpmudz$. We introduce a subspace of $\Lpmudz$ that receives an interpretation of those vector fields that are point-wise orthogonal to measure $\mu$:
	\begin{equation*}
	\mathcal{N}^p_\mu := \left\{ v \in \left(L^{p}_\mu\right)^d : (0,v) \in \overline{G}   \right\}.
	\end{equation*}
	The space $\mathcal{N}^p_\mu$ enjoys an essential \textit{stability} property, see Lemma A.1 in \cite{bouchitte2003} for details. It allows us to infer existence of a $\mu$-measurable multifunction $N^p_\mu$ that point-wise gives a linear subspace of $\Rd$ and satisfies: $v \in \mathcal{N}^p_\mu \Leftrightarrow v(x) \in N^p_\mu(x)$ for $\mu$-a.e. $x$. The space tangent to the measure $\mu$ at a point $x$ can readily be defined by means of orthogonal complement:
	\begin{equation*}
	\Tmuz(x) := \left( N^p_\mu(x) \right)^\perp \qquad \text{for } \mu\text{-a.e. } x.
	\end{equation*}
	For $\mu$-a.e. $x$ by $\Pmuz(x): \Rd \rightarrow \Rd$ we will mean the operator of orthogonal projection onto the subspace $\Tmuz(x) \subset \Rd$ (we shall omit the dependence of $\Pmuz$ on the exponent $p$ although it may factually occur). The next step involves defining for smooth functions the  \textit{derivative tangential to} $\mu$ at $\mu$-a.e. point $x$:
	\begin{equation*}
	\tgradz u (x) = \Pmuz(x) \bigl( D u (x)\bigr) \qquad \text{for } u\in \D(\Rd).
	\end{equation*}
	where the classical derivative $D u$ may be viewed as gradient, i.e. we shall see $\tgradz u$ as an element of $\Lpmudz$. More precisely we introduce an unbounded operator on $\Lpmuz$ with the space $\D(\Rd)$ as its domain: $\tgradz : \D(\Rd) \subset \Lpmuz \rightarrow \Lpmudz$. Having the stability property of $\mathcal{N}_\mu^p$ at our disposal we may give the closability result:
	\begin{proposition}
		\label{prop:closabolity_first_order}
		The unbounded operator $\tgradz : \D(\Rd) \subset \Lpmuz \rightarrow \Lpmudz$ is closable: given a sequence $u_h \in \D(\Rd)$ such that $u_h \rightarrow u$ in $\Lpnu$ and $\tgradz u_h \rightarrow v$ in $\Lpmudz$ for some $v \in \Lpmudz$, there necessarily must hold: $v=0$ in $\Lpmudz$.
	\end{proposition}
	For the proof one may see for instance \cite{bouchitte1997} or \cite{bouchitte2003}. By the first order Sobolev space $\Hpmuz$ with respect to measure $\mu$ we define a domain of the closure of $\tgradz$ (denoted by the same symbol $\tgradz$). Endowing $\Hpmuz$ with the graph norm
	\begin{equation*}
	\norm{u}_{\Hpmuz} := \left( \norm{u}_{\Lpmuz}^p + \norm{\tgradz u}_{\Lpmuz}^p \right)^{1/p}
	\end{equation*}
	renders it Banach for $p \in \prange$. Moreover it is reflexive whenever $p \in (1,\infty)$, see \cite{bouchitte1997}. The space $\Hpmuz$ can be readily seen as the completion of the space of smooth functions $\D(\Rd)$ with respect to the norm given above. By definition $u \in \Lpmuz$ is an element of the Sobolev space $\Hpmuz$ if and only if there exist a sequence $u_h \in \D(\Rd)$ and a vector field $v \in \Lpmudz$ such that: $u_h \rightarrow u$ in $\Lpnu$ and $\tgradz u_h \rightarrow v$ in $\Lpmudz$. One may show that $v$, should it exist, is unique and in fact defines $\tgradz u := v$. We make, however, an important observations:
	\begin{remark}
		\label{rem:strong_weak}
		Since $\Lpmuz \times \Lpmudz$ is Banach and $G$ is its linear subspace and thus a convex subset, the closure $\overline{G}$ may be indifferently taken with respect to norm or weak topology in $\Lpmuz \times \Lpmudz$. As a consequence we can weaken the conditions for $u \in \Lpmuz$ to be an element of $\Hpmuz$ as follows:
		\begin{equation*}
		u \in \Hpmuz \quad \Leftrightarrow \quad \exists \, u_h\in \D(\Rd) \text{ such that } 	\left \{
		\begin{array}{ll}
		u_h \rightharpoonup u \ &\text{in } \Lpmuz, \\
		D u_h \rightharpoonup v \ & \text{in } \Lpmudz \text{ for some } v \in \Lpmudz.
		\end{array} \right.
		\end{equation*}
	\end{remark}
	The perspective of the Sobolev space $\Hpmuz$ as a completion of smooth functions justifies using the letter $H$ in its symbol, rather than $W$, see \cite{meyers1964} for notation. A discussion on defining a Sobolev space with respect to measures via the notion of weak derivatives can be for instance found in \cite{bouchitte2002}. Here we decide not to dwell on this topic, we only mention that the two Sobolev spaces $\Hpmuz$ and $W^{1,p}_\mu$ may diverge in general.
	
	We are ready to return to the one-dimensional case: with the given weight $w \in L^1(I)$ we conveniently enter the theory of Sobolev space with respect to measure through simply defining
	\begin{equation*}
	\muw := w \ \Leb^1 \mres I,
	\end{equation*}
	namely $\muw$ has a density $w$ with respect to Lebesgue measure restricted to the interval $I$. Since the theory displayed above was tailored for an arbitrary Radon measure, the first-order Sobolev space $\Hnukp{1}$ is already correctly defined and becomes precisely the weighted Sobolev space defined by completion of $\D(\R)$. It is crucial to remember that functions in $\D(\R)$ may admit non-zero values at $\am$ or $\ap$ and so may the functions in $\Hnukp{1}$ (provided the boundary values are meaningful). Our goal in this subsection is to follow this approach and define higher-order Sobolev spaces $\Hnukp{m}$ for arbitrary $m\in\mathbb{N}_+$. First we need to look into the structure of the first-order space, specifically we require a characterization of the tangent space $\Tmuw(x)$ at a point $x \in \R$; it is clear that on the real line this space can be either $\R$ or $\{0\}$.
	
	The very characterization of $\Tmuw(x)$ was the main topic of the work by J. Louet. In his work \cite{louet2014} a more general setting was approached as the measure addressed could be any Radon measure supported in $\bar{I}$, that is $\mu = \muw + \mu_s$, where $\mu_s$ is the singular part. For our purposes $\mu = \muw$ suffices and below we will quote results from \cite{louet2014} adapted for this simpler scenario. In the referenced work only the exponent $p=2$ was taken into account, however, all the proofs simply extend to the case of $p \in [1,\infty)$.
	
	After \cite{kufner1984} we define a subset of $I$ containing those points $x_0$ that decides the violation of the $B_p$-condition for our weight $w \in L^1(I)$:
	\begin{equation}
	\label{eq:def_Icrp}
	\Icr{p}(w) := \biggl\{ x_0 \in I : \forall\, \eps>0 \int\limits_{I \cap B(x_0,\eps)} \frac{1}{w^{\,p'/p}}\ dx = \infty \biggr\} \qquad \text{for } p\in (1,\infty)
	\end{equation}
	and
	\begin{equation}
	\label{eq:def_Icr1}
	\Icr{1}(w) := \left\{ x_0 \in I : \forall\, \eps>0 \quad {\mathcal{L}^1\mathrm{-ess \, sup}}\left\{ \frac{1}{w(x)} : x \in I \cap B(x_0,\eps) \right\} = \infty \right\}.
	\end{equation}
	After \cite{louet2014} we shall call $\Icr{p}(w)$ \textit{a critical set} for the weight $w \in L^1(I)$; consequently each point $x_0 \in \Icr{p}(w)$ will be called critical as well.
	It is straightforward to check that the $B_p$-condition \eqref{eq:Bp_condition} is equivalent to enforcing $\Icr{p}(w) = \nolinebreak \varnothing$, also for $p=1$. From the definition it is easy to infer that the set $\Icr{p}(w)$ is always closed (in a relative topology on $I \subset \R$).
	
	The main result of Louet was to observe that the critical set contains exactly those points in $I$ for which the tangent space to $\muw$ is trivial; we quote his statement:
	\begin{proposition}
		\label{prop:Louet}
		Assume for the interval $I = \Iaa \subset \R$ a weight $w \in L^1(I)$ and choose $\muw = w \ \Leb^1 \mres I$. Then a characterization of the tangent space to the measure $\muw$ follows:
		\begin{equation*}
		\Tmuw(x) =
		\left\{
		\begin{array}{cl}
		\R &\quad \mathrm{if}\ \ x \in I \backslash \Icr{p}(w) \\
		\{0\}&\quad \mathrm{if}\ \ x \in \Icr{p}(w)
		\end{array}
		\right. \qquad \text{for } \muw\text{-a.e. } x. 
		\end{equation*}
	\end{proposition}
	The proof of the first claim i.e. that $\Tmuw(x) = \R$ for $\muw$-a.e. $x \in  I \backslash \Icr{p}(w)$ is easy and a similar estimate to \eqref{eq:B_p_estimate} serves as its core. The rest of the proof is long and technical, the reader is encouraged to see the original work \cite{louet2014}.
	
	We will look at the possible "size" of the critical set $\Icr{p}(w)$. We can trivially choose a weight $w$ such that $\Leb^1\bigl( \Icr{p}(w)\bigr)>0$, in particular for $I = (0,1)$ and arbitrary $p\in \prange$ it suffices to set $w = w_B := \nolinebreak \mathbbm{1}_{B}$ with e.g. $B=(1/2,1)$, where $\mathbbm{1}_B$ denotes the characteristic function of the set $B$. Examining the measure $\muw\bigl( \Icr{p}(w)\bigr)$ is of course entirely different matter: for the case above we clearly end up with $\mu_{w_B}\bigl( \Icr{p}(w_B)\bigr) = 0$. Let, on the other hand, $\mathscr{C}$ denote a fat Cantor set in $I = (0,1)$, we choose $w = w_\mathscr{C} = \mathbbm{1}_\mathscr{C}$. Since the Cantor set is nowhere dense, every point happens to be critical for any $p \in \prange$, namely $\Icr{p}(w_\mathscr{C}) = I$. Ultimately we obtain $\mu_{w_\mathscr{C}} \bigl( \Icr{p}(w_\mathscr{C})\bigr) = \Leb^1(\mathscr{C}) > 0$.\
	
	We shall now pass to defining the weighted Sobolev space of higher order $m$. Strong degeneration of the weight dispose us of a global Poincar\'e-like inequality in the space $\Hnukp{1}$, see for instance \cite{bouchitte2003} or \cite{hajlasz2000}. Hence we cannot define the space $\Hnukp{m}$ by focusing only on the highest, $m$-th derivative. Later, in Section \ref{sec:conclusions} we shall see that a version of a generalized Poincar\'{e} is possible to obtain, although we will prove its validity only on some basic examples of weights $w$. At this point a natural way out is to define the higher-order Sobolev inductively, that is the definition of $\Hnukp{m}$ will depend on the space $\Hnukp{m-1}$. We will learn that this is easily done as long as we impose some additional conditions on the weight $w$, yet not as restrictive as the condition of standard Poincar\'e inequality. We observe that due to Proposition \ref{prop:Louet} for any $w \in L^1(I)$ we have
	\begin{equation}
	\label{eq:a.e.non-crit}
	\noncrit \qquad \Rightarrow \qquad \Tmuw(x) = \R \quad \text{for } \muw\text{-a.e. }x
	\end{equation}
	and this resulting property will enable a simple definition of higher-order Sobolev space that essentially draws upon the theory of the first-order space $\Hnukp{1}$.
	
	For any smooth function $u \in \D(\R)$  we introduce the $k$-th derivative tangential to $\muw$ \linebreak with $k \in \mathbb{N}$:
	\begin{equation}
	\label{eq:tgradnuk}
	\tgradnuk{k} u\, (x) := \Pmuw(x) \bigl(D^k u \, (x) \bigr) \quad \text{for } \muw\text{-a.e. }x
	\end{equation}
	where again $\Pmuw(x)$ is an orthogonal projection onto $\Tmuw(x)$, hence for any $u\in \D(\R)$ the tangential derivative $\tgradnuk{k} u$ is a scalar function and an element of \nolinebreak $\Lpnu$. Recall that according to Proposition \ref{prop:closabolity_first_order} the unbounded operator $\tgradnuk{1} := \tgradnuk{}:\D(\R) \subset \Lpnu \rightarrow \Lpnu$ is closable and the domain of the closure is precisely $\Hnukp{1}$.
	
	We observe that for weights $w$ satisfying $\noncrit$ we obtain through \eqref{eq:a.e.non-crit} that $\tgradnuk{k} u = D^k u \ $ $\muw$-a.e for any smooth $u\in \D(\R)$. This will easily provide us with another closability result and ultimately a definition of the higher-order weighted Sobolev space $\Hnukp{m}$ as below (we agree that $\Hnukp{0} = \Lpnu$):
	
	\begin{proposition}
		For an interval $I = \Iaa \subset \R$ let $w \in L^1(I)$ be a weight satisfying the condition $\noncrit$, where $p \in [1,\infty)$. Let $m \geq 1$, then the unbounded operator
		\begin{equation*}
		\tgradnuk{m}:\D(\R) \subset \Hnukp{m-1} \rightarrow \Lpnu
		\end{equation*}
		is closable in $\Hnukp{m-1}$ and by the weighted Sobolev space $\Hnukp{m}$ we mean the domain of this closure and endow it with the graph norm
		\begin{equation}
		\label{eq:norm_Hpnuk}
		\norm{u}_{\Hnukp{m}} := \left( \norm{u}_{\Hnukp{m-1}}^p + \norm{\tgradnuk{m} u}_{\Lpnu}^p \right)^{1/p},
		\end{equation}
		which renders $\Hnukp{m}$ Banach for $p \in [1,\infty)$ and reflexive for $p \in (1,\infty)$.
		\begin{proof}
			Let us take a sequence $u_h \in \D(\R)$ such that $u_h \rightarrow 0$ in $\Hnukp{m-1}$ and $\tgradnuk{m} u \rightarrow v$ in $\Lpnu$. In order to prove closability of $\tgradnuk{m}$ we must show that $v = 0$ in $\Lpnu$. This is established for $m=1$ (see the comment above the proposition), hence we may proceed inductively for $m \geq 2$ assuming that $\tgradnuk{m-1}$ is closed in $\Hnukp{m-1}$. Since $u_h \in \D(\R)$, we obtain for $\muw$-a.e. $x$ (see the definition \eqref{eq:tgradnuk}):
			\begin{equation}
			\label{eq:tgrad_induct}
			\tgradnuk{m} u_h\, (x) = \Pmuw(x) \biggl(D \left(D^{m-1} u_h \right)  (x) \biggr) = \tgradnuk{}\! \left(D^{m-1} u_h  \right) (x).
			\end{equation}
			At this point we use the condition $\muw\bigl( \Icr{p}(w) \bigr) = 0$ which guarantees that $D^{m-1} u_h$ is equal to $\tgradnuk{m-1} u_h\ $ $\muw$-a.e. Further, since $u_h \rightarrow 0$ in $\Hnukp{m-1}$, we have $D^{m-1} u_h = \tgradnuk{m-1} u_h \rightarrow 0$ in $\Lpnu$. Then the closedness of $\tgradnuk{}$ and \eqref{eq:tgrad_induct} give $v = 0$ in $\Lpnu$.
			
			The definition of $\Hnukp{m}$ as the domain of the closure of $\tgradnuk{m}$ is carried out analogically to defining $\Hpmuz$ below Proposition \ref{prop:closabolity_first_order}. The reflexivity for $p\in (1,\infty)$ also follows from precisely same reasons as in the case of $\Hpmuz$ for which, in turn, the argument may be conducted analogically as for the standard Sobolev spaces $W^{m,p}(\Omega)$, see for instance Chapter 3 in \cite{adams2003}.
		\end{proof}
	\end{proposition}
	
	By the definition of $\Hnukp{m}$, a function $u \in \Hnukp{m-1}$ belongs to $\Hnukp{m}$ with $v_m:= \tgradnuk{m} u \in \Lpnu$ if and only if there exists a sequence $u_h \in \D(\R)$ such that $u_h \rightarrow u$ in $\Hnukp{m-1}$ and $\tgradnuk{m} u_h \rightarrow v$ in $\Lpnu$. Again by induction and by acknowledging $\noncrit$ we obtain a characterization:
	\begin{equation*}
	u \in \Hnukp{m} \quad \Leftrightarrow \quad \exists \, u_h\in \D(\R) \text{ such that } 	\left \{
	\begin{array}{ll}
	u_h \rightarrow u \ &\text{in } \Lpnu, \\
	D^k u_h \rightarrow v_k \ & \text{in } \Lpnu
	\end{array} \right.
	\end{equation*}
	where for $k \in \{1,\ldots,m\}$ the functions $v_k$ are any elements from $\Lpnu$; moreover each $v_k$ is uniquely defined and by definition equals $\tgradnuk{k} u$. Based on the same argument as in Remark \ref{rem:strong_weak} strong convergence in $\Lpnu$ above may be replaced by weak convergence instead.
	
	On the other hand equations \eqref{eq:tgrad_induct} together with the condition $\noncrit$ unlocks an apparatus known from classical calculus, i.e. for $m \geq 1$:
	\begin{equation}
	\label{eq:inductive_Sobolev}
	u \in \Hnukp{m} \qquad \Leftrightarrow \qquad u \in \Hnukp{m-1} \quad \text{and} \quad \tgradnuk{m-1} u \in \Hnukp{1},
	\end{equation}
	where by $\tgradnuk{0} u$ we understand the function $u$ itself.  Moreover, for $u \in \Hnukp{m}$
	\begin{equation*}
	\tgradnuk{k} u = \tgradnuk{} \!\left( \tgradnuk{k-1} u \right) \qquad \text{for} \qquad k \in \left\{ 1, \ldots, m \right\} .
	\end{equation*}
	
	We have seen that the established definition of the weighted Sobolev space $\Hnukp{m}$ directly depends on the condition $\noncrit$. \textit{We will keep this assumption throughout the rest of this work}; such weights will be called $\muw$\textit{-a.e. non-critical}. In order to emphasize the range of applicability of our definition we put forward the following result:
	\begin{proposition}
		\label{prop:BV_crit}
		Let a non-negative function $w \in L^1(I)$ have a bounded variation, i.e. $w \in BV(I)$. Then $\noncrit$ for every $p \in [1,\infty)$.
		\begin{proof}
			We shall work with the precise representative $\breve{w}$ of the function $w$, i.e. $\breve{w}(x):= \lim_{r \rightarrow 0}  \fint_{B(x,r)} w(y)\, dy$; since $w\in BV(I)$ the function $\breve{w}$ is approximately continuous and $\Leb^1$-a.e. equal to $w$, see e.g. \cite{evans1992}.
			
			Let us consider a point $x_0 \in I$  and assume that $\breve{w}(x_0)>0$ and that $x_0$ is a continuity point of $\breve{w}$. Then for each $\eps \in \bigl(0,\breve{w}(x_0)\bigr)$ there exists $\delta>0$ such that $\breve{w}(x)>\eps>0$ for every $x\in B(x_0,\delta)$ and thus $x_0 \notin  \Icr{p}(w)$.
			
			We have showed that the set $\Icr{p}(w)$ is contained in the sum of the set $\left\{ x \in I : \breve{w}(x) = 0 \right\}$ and the set of discontinuity points of $\breve{w}$. It is obvious that the measure $\muw$ of the first set is zero. Since $\breve{w}$ is of bounded variation on interval $I$ in the classical sense, the second set is at most countable, and thus of Lebesgue measure zero.
		\end{proof}
	\end{proposition}
	
	The main focus of this work is to examine the continuity properties of elements $u \in \Hnukp{m}$ and, in case of higher $m$, continuity of derivatives $\tgradnuk{k}u$ for $k<m$. We start here with some initial remarks just for the first order space $\Hnukp{1}$.
	
	Let us assume that, for some interval $I$ and any $p\in \prange$, our weight $w$ satisfies the $B_p$-condition, that is $\Icr{p}(w) = \varnothing$. Then we have the embedding $\Lpnu \hookrightarrow \Lloc{I}$ and, by taking element $u \in \Hnukp{1}$ and (guaranteed by definition) a sequence $u_h\in\D(\R)$ that converges to $u$ in the norm of $\Hnukp{1}$, we obtain that $u_h$ and $\tgradnuk{} u_h$ are Cauchy in $\Lloc{I}$, but, since $\Icr{p}(w) = \varnothing$, we have $\tgradnuk{}u_h = D u_h $ $\muw$-a.e. Again by the embedding  $u, \tgradnuk{} u \in \Lloc{I}$ and it is straightforward to check that necessarily $u_h \rightarrow u$ and $D u_h \rightarrow \tgradnuk{}u$ in $\Lloc{I}$. Ultimately we obtain that the distributional derivative $D u$ is regular and induced by the function $\tgradnuk{}u$ which renders $u$ as an element of $W^{1,1}_\mathrm{loc}(I)$ by which we mean the classical Sobolev space, defined indifferently via weak derivatives or completion. Thus by a known result there exists a $\Leb^1$-a.e. equal representative of $u$ that is locally absolutely continuous in $I$. For more details see the proof of Corollary \ref{cor:higer_order_continuity}.
	
	Next we look for possible discontinuities of a function $u \in \Hnukp{1}$ in the case when the critical set $\Icr{p}(w)$ is non-empty. We examine a natural class of weights that degenerate to zero around a point $x_0$ at different rates:
	
	\begin{example}
		\label{ex:wa}
		Let $I = (-1,1)$, $p \in \prange$ and $x_0 \in I$, we consider a class of weights $w_\gamma \in L^\infty(I)$:
		\begin{equation*}
		w_\gamma(x) = \abs{x-x_0}^\gamma
		\end{equation*}
		for $\gamma \in [0,\infty)$. We check for which exponents $\gamma$ and $p$ the point $x_0$ belongs to the critical set $\Icr{p}(w_\gamma)$; for any $\eps>0$ and $p >1$
		\begin{equation*}
		\int\limits_{I \cap B(x_0,\eps)} \frac{1}{w_\gamma^{\,p'/p}}\ dx = \int\limits_{I \cap B(x_0,\eps)} \abs{x-x_0}^{-{\gamma}/{(p-1)}} dx
		\end{equation*}
		which is infinite if and only if $\gamma \geq p-1$, recalling that $p$ must be greater than 1. In case of $p=1$ the definition \eqref{eq:def_Icr1} immediately implies that $x_0$ is critical if and only if $\gamma$ is sharply bigger than zero. In summary
		\begin{equation}
		\label{eq:Icr_wa}
		x_0  \in \Icr{p}(w_\gamma) \qquad \Leftrightarrow \qquad
		\left\{
		\begin{array}{cl}
		\gamma \geq p-1 &\quad \mathrm{if}\ \ p\in (1,\infty), \\
		\gamma >0 &\quad \mathrm{if}\ \ p=1.
		\end{array}
		\right.
		\end{equation}
		Having established the above we now turn to check whether a step function
		\begin{equation}
		\label{eq:ustep}
		\ustep = \mathbbm{1}_{(x_0,1)}
		\end{equation}
		belongs to the space $\Hnukpa{1}$ for chosen $p \in \prange$ and $\gamma \in [0,\infty)$. We note that the distributional derivative $D \ustep$ is not regular as it is equal to the Dirac delta distribution $\delta_{x_0}$, therefore, provided that indeed $\ustep \in \Hnukpa{1}$, the tangential derivative $\tgradnuka{} \ustep \in \Lpnua$ and the distributional derivative $D \ustep$ must diverge. To put it differently, the tangential derivative $\tgradnuka{} \ustep$, should it exists, cannot be a commonly understood weak derivative of $\ustep$.
		
		In order to show that $\ustep \in \Hnukpa{1}$ we must find a sequence $u_h \in \D(\R)$ such that $u_h \rightarrow \ustep$ in $\Lpnua$ and also $\tgradnuka{}\! u_h \rightarrow v$ in $\Lpnua$ for some $v$; we recall that $\tgradnuka{}\! u_h = \nolinebreak D u_h \ $ $\mu_{w_\gamma}$-a.e. since the weight $w_\gamma$ is $\mu_{w_\gamma}$-a.e. non-critical. We propose a sequence $u_h \in \mathrm{Lip}(\R)$ instead, since each $u_h$ can be $\Hnukpa{1}$-cheaply smoothed out due to $w_\gamma \in L^\infty(I)$. To focus attention we further assume that $x_0=0$:
		\begin{equation}
		\label{eq:def_u_h_w_a}
		u_h(x) = \left\{
		\begin{array}{clc}
		0 &\quad \mathrm{if}\ \ &x \leq 0, \\
		h \,x  &\quad \mathrm{if}\ \ &0 < x < 1/h,\\
		1 &\quad \mathrm{if}\ \ &1/h \leq x.\\
		\end{array}
		\right.
		\end{equation}
		which gives the a.e. defined derivative $D u_h \equiv h$ in $(0,1/h)$ and $D u_h \equiv 0$ in $I \backslash (0,1/h)$.
		
		Due to the dominated convergence theorem it is obvious that $u_h \rightarrow \ustep$ in $\Lpnua$; we look at the $\Lpnua$-norm of $D u_h$:
		\begin{equation*}
		\norm{D u_h}^{\,p}_\Lpnua = \int_0^{1/h} \abs{x}^\gamma \ h^p\, dx = \frac{h^{\,(p-1)-\gamma}}{\gamma+1}.
		\end{equation*}
		Hence we see that, for any $p \in \prange$, there holds $D u_h \rightarrow v \equiv 0$ in $\Lpnua$ whenever $\gamma > p - 1$ and then we assess $\ustep \in \Hnukpa{1}$ with $\tgradnuka{}\ustep = 0$. In case of $\gamma < p-1$ our sequence $u_h$ is unbounded in $\Hnukpa{1}$, yet this does not settle whether $u \notin \Hnukpa{1}$. However, according to characterization \eqref{eq:Icr_wa}, for $\gamma < p-1$ the critical set for $w_\gamma$ is empty and there must not be a discontinuous function in $\Hnukpa{1}$.
		
		It is left to judge the case of $\gamma = p-1$ for which the sequence $D u_h$ is bounded in $\Lpnua$ yet not convergent to zero. If $p>1$, by reflexivity we find $v \in \Lpnua$ such that (up to choosing a subsequence) $D u_h\rightharpoonup v$ in $\Lpnua$. Hence, according to Remark \ref{rem:strong_weak}, for $p>1$ and $\gamma = p-1$ indeed $\ustep \in \Hnukpa{1}$ holds with $\tgradnuka{}\ustep =v$. Again $v$ must be zero in $\Lpnua$ and we outline the reason. Now that we know $\ustep \in \Hnukpa{1}$ we can choose a different sequence $u_h \in \D(\R)$ for which the convergence in $\Hnukpa{1}$ to $\ustep$ is strong. For the weight $w_\gamma$ there is no critical points other than $x_0$ and therefore $\Lpnua \hookrightarrow \Lloc{I\backslash \{x_0\}}$. Then $D u_h \rightarrow v$ in $\Lloc{I\backslash \{x_0\}}$ and we may infer that $u_h \rightarrow \ustep$ in $W^{1,1}_\mathrm{loc}(I\backslash\{x_0\})$ (see proof of Corollary \ref{cor:higer_order_continuity}). But $\ustep$ is constant in $(-1,x_0)$ and in $(x_0,1)$ therefore $v$ must be zero a.e. in \nolinebreak $I$. 
		
		The above cannot be repeated for $p=1$ and $\gamma = p-1 = 0$ due to lack of reflexivity, although then $w = w_0 \equiv 1$ and the critical set is empty, hence $\ustep$ cannot be an element of $H^{1,1}_{\mu_{w_0}}$. Regarding the characterization \eqref{eq:Icr_wa} our results for weights of class $w_\gamma$ may be summarized for any $p \in \prange$ and $\gamma \in \prange$:
		\begin{equation}
		\label{eq:disc_Icr_wa}
		\ustep = \mathbbm{1}_{(x_0,1)} \in \Hnukpa{1} \quad \text{with} \quad  \tgradnuka{}\ustep \equiv 0 \qquad \Leftrightarrow \qquad x_0  \in \Icr{p}(w_\gamma).
		\end{equation}
		We end the example with a short remark: the whole argument can be unchangeably repeated for the weight $w_\gamma$ redefined so that $w_\gamma(x)=1$ for $x\in (-1,x_0)$ or even $w_\gamma(x)=\abs{\log(\abs{x-x_0})}$ for $x\in (-1,x_0)$. This way we learn that for the point $x_0$ to be critical, or for the step function at $x_0$ to belong $\Hnukp{1}$, we need the weight $w$ to degenerate "fast enough" only on one side of $x_0$, while on the other the weight may even blow up.
	\end{example}
	
	\begin{example}
		\label{ex:wlog}
		For $I = (-1/2,1/2)$, $p \in \prange$ and $x_0 \in I$ we define a weight $w_\mathrm{log} \in \nolinebreak L^\infty(I)$:
		\begin{equation}
		\label{eq:wlog_def}
		w_\mathrm{log} (x) = \frac{1}{\abs{\log(\abs{x-x_0})}}.
		\end{equation}
		For any $\eps>0$ and $p >1$
		\begin{equation*}
		\int\limits_{I \cap B(x_0,\eps)} \frac{1}{w_\mathrm{log}^{\,p'/p}}\ dx = \int\limits_{I \cap B(x_0,\eps)} \abs{\log(\abs{x-x_0})}^{{1}/{(p-1)}} dx
		\end{equation*}
		which is finite for every $p \in (1,\infty)$ and thus $x_0 \notin \Icr{p}(w_\mathrm{log})$ for those $p$. On the other hand the weight $w_\mathrm{log}$ is not essentially bounded from below by a positive number and thus $x_0$ is a critical point for $p=1$, namely $x_0 \in \Icr{1}(w_\mathrm{log})$.
		
		Independently we will test whether the step function $\ustep = \mathbbm{1}_{(x_0,1/2)}$ is an element of $H^{1,p}_{\mu_{w_\mathrm{log}}}$ for different $p \in \prange$. Assuming that $x_0=0$ we take the sequence defined in \eqref{eq:def_u_h_w_a} and we compute
		\begin{equation*}
		\norm{D u_h}^{\,p}_{L^p_{\mu_{w_\mathrm{log}}}} = \int_0^{1/h} \frac{h^p}{\abs{\log(x)}} \ dx = h^p \ \abs{\,\mathrm{li}(1/h)\,}.
		\end{equation*}
		where $\mathrm{li}$ denotes the logarithmic integral special function. From the properties of $\mathrm{li}$ the right hand side above converges to zero if $p = 1$ and diverges to infinity otherwise. Therefore $\ustep \in H^{1,1}_{\mu_{w_\mathrm{log}}}$ with $D_{\mu_{w_\mathrm{log}}} \ustep \equiv 0$. Since for $p>1$ the critical set $\Icr{p}(w_\mathrm{log})$ is empty, a discontinuous function $\ustep$ cannot be an element of our Sobolev space. We give a conclusion analogical to the one from the previous example, this time for the weight $w_\mathrm{log}$ and every $p \in \prange$:
		\begin{equation}
		\label{eq:disc_Icr_wlog}
		\ustep = \mathbbm{1}_{(x_0,1/2)} \in H^{1,p}_{\mu_{w_\mathrm{log}}} \quad \text{with} \quad  D_{\mu_{w_\mathrm{log}}} \ustep \equiv 0 \qquad \Leftrightarrow \qquad x_0  \in \Icr{p}(w_\mathrm{log}).
		\end{equation}
	\end{example}
	
	\section{On sufficient conditions for continuity of a function and its derivatives in the higher order weighted Sobolev space $\Hnukp{m}$ on the real line}
	\label{sec:sufficient_conditions}
	
	In the previous subsection, for a given interval $I\subset \R$, a weight $w \in L^1(I)$ and the exponent $p\in \prange$, we have inferred continuity of a function $u \in \Hnukp{1}$ under the condition that the critical set $\Icr{p}(w)$ is empty. The argument was based on the resulting embedding $\Lpnu \hookrightarrow \Lloc{I}$, which then furnished $\Hnukp{1} \hookrightarrow W^{1,1}_\mathrm{loc}(I)$. Contrarily, the studied examples have pointed out that for strongly degenerate weights, i.e. if there exists at least one $x_0 \in \nolinebreak \Icr{p}(w)$, a function from $\Hnukp{1}$ may admit a jump-type discontinuity at $x_0$, although up till now that has been firmly established only for weights of classes $w_\gamma$ and $w_\mathrm{log}$, see \eqref{eq:disc_Icr_wa} and \eqref{eq:disc_Icr_wlog}.
	
	Now we pass to investigating continuity of functions from higher order weighted Sobolev space $\Hnukp{m}$; we remind that for this space to be well defined we keep the assumption that the weight $w$ is $\muw$-a.e. non-critical. Take for instance an element $u \in \Hnukp{2}$ and a weight $w$ that admits some critical points. Since the higher order spaces were defined inductively we may expect discontinuities of the first derivative $\tgradnuk{}u$, but is it perhaps possible to impose some extra condition on the weight $w$ such that, despite $\Icr{p}(w) \neq \varnothing$, we can deduce continuity of the function $u$ itself? We start by generalizing and adapting the notion of the critical set:
	\begin{definition}
		For an interval $I = \Iaa \subset \R$, a weight $w \in L^1(I)$ and an exponent $p\in\prange$ we define a critical set of order $\al\in [0,\infty)$ as a subset of the closure $\bar{I}$:
		\begin{equation}
		\label{eq:def_Icrap}
		\cIcr{\al,p}(w) := \biggl\{ x_0 \in \bar{I} : \forall\, \eps>0  \int\limits_{\bar{I} \cap B(x_0,\eps)} \! \left(\frac{\abs{x-x_0}^\al}{\bigl(w(x)\bigr)^{1/p}}\right)^{p'} \! dx = \infty \biggr\} \quad \text{for } p\in (1,\infty)
		\end{equation}
		and
		\begin{equation}
		\label{eq:def_Icra1}
		\cIcr{\al,1}(w) := \biggl\{ x_0 \in \bar{I} : \forall\, \eps>0 \quad {\mathcal{L}^1\mathrm{-ess \, sup}}\left\{ \frac{\abs{x-x_0}^\al}{w(x)} : x \in \bar{I} \cap B(x_0,\eps) \right\} = \infty \biggr\}.
		\end{equation}
	\end{definition}
	\begin{remark}
		Obviously, for any $p \in \prange$ the equality $\Icr{p}(w) = \cIcr{0,p}(w) \cap I$ holds. Note that, apart from introducing an order $\alpha$, we have additionally altered the definition of the critical set by considering the endpoints of the interval $I = \Iaa$, namely each $x_0$ in the closure $\bar{I}$ is now being tested. In \cite{kufner1984} testing the boundary $\partial \Omega$ was not necessary since the key embedding $L^p_w(\Omega) \hookrightarrow \Lloc{\Omega}$ for proving completeness of weighted Sobolev space was indifferent to behaviour of $w$ close to the boundary. Here, if $\cIcr{0,p}(w) = \varnothing$ we can write down an inequality of the form \eqref{eq:B_p_estimate} with integrals over the whole $\bar{I}$ and then we obtain more: $\Lpnu \hookrightarrow L^1(I)$. In this paper looking at criticality of boundary points $\am,\ap$ will prove essential for continuous extensions of the trace operators to the space $\Hnukp{m}$, see Corollary \ref{prop:trace_extension}. Handling the boundary points, however, will cause some minor technical difficulties, see Remark \ref{rem:open_sets}.
	\end{remark}
	We give some basic properties of the newly proposed critical set of order $\al$, starting from monotonicity both with respect to $p$ and $\al$. For a given $w \in L^1(I)$, fixed $\al \in [0,\infty)$ we have
	\begin{equation}
	\label{eq:monot_p}
	\cIcr{\al,p_2}(w) \subset \cIcr{\al,p_1}(w) \qquad \text{for} \quad p_1 \leq p_2,
	\end{equation}
	since the integrand in definitions above is raised to the power $p'$. Secondly, for a fixed $p \in \nolinebreak \prange$ it is straightforward that
	\begin{equation}
	\label{eq:monot_al}
	\cIcr{\al_2,p}(w) \subset \cIcr{\al_1,p}(w)  \qquad \text{for} \quad \al_1 \leq \al_2,
	\end{equation}
	namely for a higher order $\al$ the weight must degenerate faster around a point $x_0$ to furnish its criticality.
	
	Directly from the definitions it follows that, for every $\al\geq0$ and $p\in\prange$ the set $\cIcr{\al,p}(w)$ is closed. For $\al= 0, p >1$ let us take a point $x_0 \in \bar{I} \backslash \cIcr{\al,p}(w)$. Since $x_0$ is not critical we have $\eps>0$ such that the integral over $I \cap B(x_0,\eps)$ in \eqref{eq:def_Icrap} is finite. Then every point $\tilde{x}_0$ from $\bar{I} \cap B(x_0,\eps)$ is not critical as for each such point $\tilde{x}_0$ we may choose $\tilde{\eps}$ so that $B(\tilde{x}_0,\tilde{\eps}) \subset B(x_0,\eps)$. Thus the integral over $B(\tilde{x}_0,\tilde{\eps})$ is also finite. For $p = 1$ the argument is analogical, while for $\al >0$ we shall state a stronger result in Proposition \ref{prop:noncrit_punct_ball}.
	
	For $p>1$ let us assume that a closed, and thus compact, set $F \subset \bar{I}$ does not contain any critical points of any order, namely $F \cap \cIcr{0,p}(w) = \varnothing$. For each point $x \in F$ there exists $\eps_x>0$ such that the integral in \eqref{eq:def_Icrap} over $I \cap B(x,\eps_x)$ is finite. By compactness of $F$ we can choose a finite family of those balls $B(x_n,\eps_{x_n})$ which covers $F$. Ultimately we have established that for $p \in (1,\infty)$
	\begin{equation}
	\label{eq:closed_F_Icrp}
	\int_F \frac{1}{w^{\,p'/p}}\ dx <\infty \qquad \text{for every closed } F \subset \bar{I}\backslash \cIcr{0,p}(w)
	\end{equation}
	and, which can be shown analogically, for $p =1$
	\begin{equation}
	\label{eq:closed_F_Icr1}
	\norm{1/w}_{L^\infty(F)} <\infty \qquad \text{for every closed } F \subset \bar{I}\backslash \cIcr{0,1}(w).
	\end{equation}

	\begin{example}
		\label{ex:wgamma_wexp}
		For an interval $I = (-1,1)$, a point $x_0 \in \bar{I}$ and any $p \in \prange$ we consider two weights: $w_\gamma \in L^\infty(I)$ for some $\gamma \in [0,\infty)$ and $w_\mathrm{exp} \in L^\infty(I)$ as follows
		\begin{equation*}
		w_\gamma(x) = \abs{x-x_0}^\gamma, \qquad w_\mathrm{exp}(x) = \frac{1}{\exp(1/\abs{x-x_0})},
		\end{equation*}
		where for $x_0=0$ the function $w_\mathrm{exp}$ restricted to $(0,1)$ is the inverse of $w_\mathrm{log}\lvert_{(0,1)}$ defined in \eqref{eq:wlog_def}. First, for $p>1$ we look at
		\begin{equation*}
		\int\limits_{I \cap B(x_0,\eps)} \left(\frac{\abs{x-x_0}^\al}{\bigl(w_\gamma(x)\bigr)^{\,1/p}}\right)^{p'} dx = \int\limits_{I \cap B(x_0,\eps)} \abs{x-x_0}^{-(\gamma-\al p)/{(p-1)}} dx
		\end{equation*}
		and, also acknowledging the definition \eqref{eq:def_Icra1} of $\cIcr{\al,1}(w)$, we obtain
		\begin{equation}
		\label{eq:Icra_wa}
		x_0  \in \cIcr{\al,p}(w_\gamma) \qquad \Leftrightarrow \qquad
		\left\{
		\begin{array}{cl}
		\gamma \geq p(\al+1)-1 &\quad \mathrm{if}\ \ p\in (1,\infty), \\
		\gamma >\al &\quad \mathrm{if}\ \ p=1.
		\end{array}
		\right.
		\end{equation}
		Further, again for $p>1$
		\begin{equation*}
		\int\limits_{I \cap B(x_0,\eps)} \left(\frac{\abs{x-x_0}^\al}{\bigl(w_\mathrm{exp}(x)\bigr)^{\,1/p}}\right)^{p'} dx = \int\limits_{I \cap B(x_0,\eps)} \abs{x-x_0}^{\al p'} \bigl(\exp(1/\abs{x-x_0})\bigr)^{1/(p-1)}  dx
		\end{equation*}
		which is infinite for every $\al \geq 0$ and $p\in(1,\infty)$, hence
		\begin{equation}
		\label{eq:Icra_wexp}
		x_0 \in \cIcr{\al,p}(w_\mathrm{exp}) \qquad \text{for every }  \al \in[0,\infty) \text{ and } p\in \prange. 
		\end{equation}
	\end{example}
	\vspace{4mm}
	Let us for $p=2$ and a point $x_0\in \bar{I}$ consider a weight $w_\gamma(x)=\abs{x-x_0}^\gamma$ for any $\gamma\in [1,3)$. According to \eqref{eq:Icra_wa} we have $x_0 \notin \cIcr{1,2}(w_\gamma)$, although $x_0 \in \cIcr{0,2}(w_\gamma)$ which simply shows that the inclusion converse to \eqref{eq:monot_al} cannot hold. Notwithstanding this we are able to show a weaker result of this fashion:
	\begin{proposition}
		\label{prop:noncrit_punct_ball}
		Assume a point in the closure of the interval $x_0 \in \bar{I} = [\am,\ap] \subset \R$, an exponent $p \in \prange$ and a weight $w\in L^1(I)$. If $x_0 \notin \cIcr{\al,p}(w)$ for some $\alpha \geq 0$, then there exists a neighbourhood $\V \ni x_0$ that is relatively open in $\bar{I}$ and satisfies
		\begin{equation}
		\label{eq:noncrit_punct_ball}
		\left(\,\overline{\V}\, \backslash\{x_0\}\,\right) \cap  \cIcr{0,p}(w) = \varnothing.
		\end{equation}
		\begin{remark}
			\label{rem:open_sets}
			We are forced to operate with sets $\V$ that are relatively open in $\bar{I}$ whenever the point $x_0$ is one of the end-points of $\bar{I} = [\am,\ap]$. Then $\V \subset \bar{I}$ furnished by the proposition is an open neighbourhood of $x_0$ in the relative topology in $\bar{I}$, but it is not a neighbourhood of $x_0$ in the topological space $\R$. We agree that in this subsection we will for brevity call  such sets "$\bar{I}$-open" and use symbol $\V$ to denote it, while $\U$ will stand for the sets open in $\R$, which we shall shortly name "open". It is obvious that for any $\bar{I}$-open set $\V$ the set $\V \cap I$ is open, while the closure $\overline{\V}$ in \eqref{eq:noncrit_punct_ball} may be indifferently taken with respect to topologies on $\bar{I}$ or $\R$.
		\end{remark}
		\begin{proof}
			We shall assume $p>1$, since the proof for $p=1$ employs the same simple idea. The fact $x_0 \notin \cIcr{\al,p}$ furnishes $\eps>0$ such that the integral over $\bar{I} \cap B(x_0,\eps)$ in the definition \eqref{eq:def_Icrap} is finite. We set
			\begin{equation*}
			\V := B(x_0,\eps/2) \cap \bar{I}.
			\end{equation*}
			For any $\tilde{x}_0 \in \overline{\V} \backslash \{x_0\}$ we choose $\tilde{\eps} :=\abs{\tilde{x}_0 -x_0}/2$ and we note that $B(\tilde{x}_0,\tilde{\eps}) \subset B(x_0,\eps)$. Then, since for every $x \in B(\tilde{x}_0,\tilde{\eps})$ there holds $\abs{x-x_0}^\al \geq \abs{\tilde{x}_0-x_0}^\al/2^\al =:C>0$, we arrive at
			\begin{alignat*}{1}
			\int\limits_{\bar{I} \cap B(\tilde{x}_0,\tilde{\eps})} \frac{1}{w^{\,p'/p}} \ dx &\leq \frac{1}{C^{p'}} \int\limits_{\bar{I} \cap B(\tilde{x}_0,\tilde{\eps})} \left(\frac{\abs{x-x_0}^\al}{\bigl(w(x)\bigr)^{1/p}}\right)^{p'} dx\\ 
			&\leq \frac{1}{C^{p'}} \int\limits_{\bar{I} \cap B(x_0,\eps)} \left(\frac{\abs{x-x_0}^\al}{\bigl(w(x)\bigr)^{1/p}}\right)^{p'} dx < \infty
			\end{alignat*}
			which implies $\tilde{x}_0 \notin \cIcr{0,p}(w)$ and the proof is complete.
		\end{proof}
	\end{proposition}
	We put forward the main result of this section, we agree that $\Hnukp{0} = \Lpnu$:
	\begin{theorem}
		\label{thm:embedding}
		Assume an interval $I = \Iaa \subset \R$ with a weight $w \in L^1(I)$ and $p \in \prange$. Choose any point $x_0 \in \bar{I}=[\am,\ap]$ and fix any $k \in \mathbb{N}$ (zero included).
		
		If $x_0 \notin \cIcr{k,p}(w)$ then there exists an $\bar{I}$-open neighbourhood $\V \ni x_0$ such that
		\begin{equation}
		\label{eq:embedding}
		\Hnukp{k} \  \hookrightarrow \ L^1(\V).
		\end{equation}
		In particular, if $\cIcr{k,p}(w) = \varnothing$ then $\V$ may be chosen as $\bar{I}$, namely $\Hnukp{k} \hookrightarrow L^1(\bar{I})$.
	\end{theorem}
	
	For $k=0$ the theorem roughly boils down to the well established fact \eqref{eq:standard_embedding}. Prior to proving the statement we shall first see what are the consequences as far as continuity of functions from $\Hnukp{m}$ are concerned. Henceforward we agree that $\bk \in \{0,\ldots m-1\}$ will stand for the order of derivative $\tgradnuk{\bk}u$ whose continuity is examined (affirmed in this section and denied in the next) at a point $x_0$. By means of induction we show that Theorem \nolinebreak \ref{thm:embedding} implies
	\begin{corollary}
		\label{cor:higer_order_continuity}
		Let us for $p\in\prange$, an interval $I \subset  \R$ and a weight $w \in L^1(I)$ denote by $\Hnukp{m}$ the $m$-th order weighted Sobolev space, where $m \in \mathbb{N}_+$. We choose a point $x_0 \in \bar{I}$, an order $\bk \in \{0,\ldots,m-1\}$ and denote $\Dk := m - \bk$.
		
		If $x_0 \notin \cIcr{\Dk-1,p}(w)$ then there exists an $\bar{I}$-open neighbourhood $\V \ni x_0$ such that
		\begin{equation}
		\label{eq:embedding_higher}
		\Hnukp{m} \  \hookrightarrow \ W^{(\bk+1),1}(\V \cap I)
		\end{equation}
		and, as a result, for a function $u \in \Hnukp{m}$ there exists a function $\breve{u}\in C^\bk(\V \cap I)$  such that for all $k \leq \bk$
		\begin{equation*}
		u = \breve{u} \quad \text{and} \quad \tgradnuk{k} u = D^k \breve{u}  \qquad \Leb^1\text{-a.e. on }\, \V,
		\end{equation*}
		where $D^k \breve{u}$ is intended in the classical sense.
		
		In the case when $\cIcr{\Dk-1,p}(w) = \varnothing$ the set $\V \cap I$ above may be replaced by $I$.
		\begin{proof}
			We choose an $\bar{I}$-open neighbourhood $\V \ni x_0$ in accordance with Theorem \ref{thm:embedding}. Let us take any function $u \in \Hnukp{m}$. Recall \eqref{eq:inductive_Sobolev}; then, since $x_0 \notin \cIcr{\Dk-1,p}(w)$, by Theorem \ref{thm:embedding} we obtain
			\begin{equation*}
			\tgradnuk{\bk+1} u \in \Hnukp{\left(m-(\bk+1)\right)} = \Hnukp{\Dk -1} \hookrightarrow L^1(\V)
			\end{equation*}
			and the same applies to derivatives $\tgradnuk{k} u$ for any $k \in \{0,\ldots,\bk+1\}$. We choose a sequence of smooth functions $u_h \in \D(\R)$ that converges to $u$ in $\Hnukp{m}$, then for every order $k \in \{0,\ldots,\bk+1\}$
			\begin{equation*}
			\norm{D^k u_h - \tgradnuk{k} u }_{L^1(\V)} \leq C \norm{u_h -u}_\Hnukp{m},
			\end{equation*}
			where $C>0$ is given by the embedding \eqref{eq:embedding}. The convergence $D^k u_h \rightarrow \tgradnuk{k} u$ in $L^1(\V)$ for every $k \in \{0,\ldots,\bk+1\}$ follows.
			
			Since $\V$ is merely $\bar{I}$-open further we work with $\U :=  \V \cap I$ that is also open in $\R$; note that $\V \subset \U \cup \{a_-,a_+\}$. Clearly all the convergences in $L^1(\V)$ above also hold in $L^1(\U)$.  Now for every $k\in \{0,\ldots, \bk+1\}$ we may compute the distributional derivative $D^k u$ on $\U$. For a function $\varphi \in \D(\U)$
			\begin{alignat*}{1}
			\pairing{\varphi, D^k u} = (-1)^k\int_\U \left(D^k\varphi\right)  u\, dx &= \lim_{h \rightarrow \infty} (-1)^k\int_\U \left(D^k\varphi\right)  u_h\, dx \\
			& = \lim_{h \rightarrow \infty} \int_\U \varphi  \left(D^k u_h\right) dx \\
			& = \int_\U \varphi  \left(\tgradnuk{k} u_h\right) dx,
			\end{alignat*}
			thus we infer that for each  $k \in \{0,\ldots,\bk+1\}$ we have $D^k u = \tgradnuk{k} u \in L^1(\U) \cap \Lpnu$ ($\tgradnuk{k}u$ induces a regular distribution $D^k u$ on $\U$) and hence $u \in W^{(\bk+1),1}(\U)$ with $u_h \rightarrow u$ in $W^{(\bk+1),1}(\U)$. Moreover
			\begin{equation*}
			\norm{u}_{W^{(\bk+1),1}(\U)} \leq (\bk+1)\, C \,\norm{u}_\Hnukp{m}
			\end{equation*}
			which ultimately establishes the embedding \eqref{eq:embedding_higher}. The rest of the corollary is a result of a well-known fact that a function $u \in W^{(\bk+1),1}(\U)$ has an almost everywhere equal representative $\breve{u} \in C^k(\U)$ with $D^k \breve{u}$ being absolutely continuous.
		\end{proof}
	\end{corollary}
	Let us return to proving Theorem \ref{thm:embedding}. We start with a simple lemma that explains how the factor $(x-x_0)^k$, appearing in the definition of $\cIcr{k,p}(w)$, comes into play in the inequality that yields the embedding \eqref{eq:embedding}:
	\begin{lemma}
		\label{lem:inequality}
		Let $x_0$ be any point on the real line and choose its open neighbourhood $\U$. Then, for $k \in \mathbb{N}$ and any $\phi \in C^k_c(\U)$ ($\phi$ has a compact support in $\U$) an inequality holds:
		\begin{equation}
		\label{eq:lemma_inequality}
		\int\limits_{\U_-} \abs{\phi(x)} \, dx \leq \frac{1}{k!} \int\limits_{\U_-} \vert \der{k}{\phi}\,(x) \rvert \abs{x-x_0}^k dx.
		\end{equation}
		where $\U_- :=\U \cap (-\infty,x_0]$. The same independently applies to the integrals taken over the set $\, \U_+:=  \nolinebreak \U \cap \nolinebreak {[x_0,\infty)}$.
		\begin{proof}
			For an arbitrary  $\phi \in C^k_c(\U)$ we define
			\begin{equation*}
			\tilde{\phi}(x) := \int_{-\infty}^{x} \abs{\der{k}{\phi}(y)} \frac{(x-y)^{k-1}}{(k-1)!} dy,
			\end{equation*}
			which gives a non-negative function $\tilde{\phi} \in C^{k-1}(\R)$ with absolutely continuous derivative $\der{(k-1)}{\tilde{\phi}}$ such that a.e. $\der{k}{\tilde{\phi}} = \abs{ \der{k}{\phi}}$.  Since $\mathrm{spt}(\tilde{\phi}) \subset [a,\infty)$ for some $a > -\infty$ and all, except the $k$-th, derivatives of monomial $P_k = P_k(x)=(x-x_0)^k$ vanish at $x_0$, we obtain through integration by parts (being valid due to $\der{(k-1)}{\tilde{\phi}} \in \mathrm{AC}(\R)$)
			\begin{alignat}{1}
			\nonumber
			\int_{-\infty}^{x_0} \abs{\phi(x)} \, dx \leq \int_{-\infty}^{x_0} \tilde{\phi}(x) \, dx &= \frac{(-1)^k}{k!} \int_{-\infty}^{x_0} \der{k}{\tilde{\phi}}\,(x)\ (x-x_0)^k \, dx \\
			\label{eq:working_lemma_ineq}
			&= \frac{1}{k!} \int_{-\infty}^{x_0} \vert \der{k}{\phi}\,(x) \rvert \abs{x-x_0}^k dx,
			\end{alignat}
			where in addition we have used the fact that $\lvert{\phi}\rvert \leq \lvert\tilde{\phi}\rvert =\tilde{\phi} $ and $(-1)^k (x-x_0)^k = \abs{x-x_0}^k$ for $x\leq x_0$. Since  $\mathrm{spt}(\phi) \subset \U$, the first and last integral above may be equivalently taken over the set $\U_- = \U \cap (-\infty,x_0]$ furnishing inequality \eqref{eq:lemma_inequality}. The same idea may be applied to the integral over $\U_+ =\U \cap [x_0,\infty)$ and the proof is complete.
		\end{proof}
	\end{lemma}
	
	\begin{proof}[Proof of Theorem \ref{thm:embedding}]
		To focus attention we will assume that $p>1$ and we shall comment on the case $p=1$ later.
		
		Since $x_0 \notin \cIcr{k,p}(w)$ for some $k \in \mathbb{N}$, we may choose an $\bar{I}$-open neighbourhood $\V \subset \bar{I}$ of $x_0$ in compliance with Proposition \ref{prop:noncrit_punct_ball}. From the proof of this proposition there holds
		\begin{equation}
		\label{eq:C1}
		\left(\int_{\V} \left(\frac{\abs{x-x_0}^k}{(w(x))^{1/p}}\right)^{p'} \! dx\right)^\frac{1}{p'} =: C_1 < \infty.
		\end{equation}
		We remind that the set $\V \ni x_0$ is in general $\bar{I}$-open; we now modify the set $\V$ to arrive at a neighbourhood $\U \ni x_0$ that is also open in $\R$ (recall that $I =\Iaa$ and $x_0 \in \bar{I}\,$):
		\begin{equation}
		\label{eq:V_def}
		\U := \left\{
		\begin{array}{cll}
		\V \cap I &\quad \mathrm{if}\ \ &x_0 \in I \, , \\
		(-\infty,x_0] \cup (\V \cap I) &\quad \mathrm{if}\ \ & x_0 = a_- \, ,\\
		(\V \cap I) \cup [x_0,\infty) &\quad \mathrm{if}\ \ & x_0 = a_+ \, .
		\end{array}
		\right.
		\end{equation}
		Next we take any open neighbourhood $\U_1\ni x_0$ that is compactly contained in $\U$, i.e $x_0 \in \U_1 \Subset \U$. Then by \eqref{eq:noncrit_punct_ball} the set $\overline{\V}\backslash \U_1$ is a closed subset of $\bar{I}$ with no critical points of any order, hence due to \eqref{eq:closed_F_Icrp}
		\begin{equation}
		\label{eq:C2}
		\left(\int_{\,\overline{\V}\backslash \U_1} \frac{1}{w^{\,p'/p}}\  dx\right)^\frac{1}{p'} =: C_2 < \infty.
		\end{equation}
		We also propose and fix a cut-off function $\varphi \in \D(\U)$ such that
		\begin{equation*}
		0\leq \varphi \leq 1 \qquad \text{and} \qquad \varphi \equiv 1 \quad \text{on} \quad \U_1.
		\end{equation*}
		
		We consider any function $u \in \D(\R)$, we stress that $u$ may not be compactly supported in \nolinebreak $\U$. To prove the embedding \eqref{eq:embedding} we must show that there exists a constant $C$ independent of $u$ such that
		\begin{equation}
		\label{eq:ineq_for_embedding}
		\qquad	\int_\V \abs{u} \, dx \ \leq \ C \ \norm{u}_\Hnukp{k} \qquad \forall \, u\in \D(\R). 
		\end{equation}
		For a picked $u$ we multiply by the smooth cut-off function:
		\begin{equation*}
		\tilde{u} := \varphi \, u.
		\end{equation*}
		According to Lemma \ref{lem:inequality} we arrive at inequality
		\begin{equation}
		\label{eq:inequality_for_Cc_part}
		\int_{\V} \abs{\tilde{u}(x)} \, dx \leq \frac{1}{k!} \int_{\V} \vert \der{k}{\tilde{u}}\,(x) \rvert \abs{x-x_0}^k dx.
		\end{equation}
		Indeed, we have $x_0 \in \U$, $\tilde{u} \in \D(\U)$ and, by the definition of $\U$, the set $\V$ is (up to elements $a_-,a_+$) equal to:  $\ \U_- \cup \U_+$ for $x_0 \in I$; $\ $ to $\U_+$ for $x_0 = a_-$; $\ $ to $\U_-$ for $x_0 = a_+$, where $\U_-$ and $\U_+$ are defined in the lemma. Therefore the inequality above can be composed from independent inequalities for $\U_-$ and $\U_+$ from Lemma \ref{lem:inequality}.
		
		We start proving the inequality \eqref{eq:ineq_for_embedding}, we observe that $u-\tilde{u} \equiv 0$ on $\U_1$:
		\begin{alignat}{1}
		\nonumber
		\int_\V \abs{u} \, dx &\leq \int_{\V} \abs{u -\tilde{u}} \, dx + \int_\V \abs{\tilde{u}} \, dx\\
		&\leq  \int_{\V\backslash \U_1} \abs{u -\tilde{u}} \, dx + \frac{1}{k!} \int_{\V} \vert \der{k}{\tilde{u}}\,(x) \rvert \abs{x-x_0}^k dx.
		\end{alignat}
		Both the integrals may be estimated by H\"{o}lder inequality in analogy to \eqref{eq:B_p_estimate}:
		\begin{alignat}{1}
		\nonumber
		\int_\V \abs{u} \, dx  \leq &
		\left(\int_{\V\backslash \U_1} w \, \abs{u -\tilde{u}}^p \, dx\right)^{\frac{1}{p}} \left(\int_{\V\backslash \U_1} \frac{1}{w^{\, p'/p}} dx\right)^{\frac{1}{p'}}\\ 
		\nonumber
		&+ \frac{1}{k!} \left(\int_{\V} w\, \vert \der{k}{\tilde{u}} \rvert^p dx\right)^{\frac{1}{p}} \left(\int_{\V} \biggl(\frac{\abs{x-x_0}^k}{(w(x))^{1/p}}\biggr)^{p'} dx\right)^\frac{1}{p'}\\
		\leq &  \ C_2\ \norm{u-\tilde{u}}_\Lpnu + C_1/k! \ \norm{\der{k}{\tilde{u}}}_\Lpnu,
		\end{alignat}
		where we have utilized \eqref{eq:C1} and \eqref{eq:C2}. Since $\varphi$ is fixed an estimate follows from the Leibniz differentiation formula:
		\begin{equation*}
		\norm{\der{k}{\tilde{u}}}_\Lpnu \leq \left(\max_{0\leq n \leq k} \norm{\der{n}{\varphi}}_\infty\right) \biggl( \int_{\V} w\,  \bigg\lvert \sum_{0\leq n \leq k} \!{k \choose n} \der{n}{u}\, \bigg\rvert^p dx \biggr)^{1/p} \leq C_3 \norm{u}_\Hnukp{k}
		\end{equation*}
		for a finite $C_3 >0$. In addition we notice that $\abs{u - \tilde{u}} \leq \abs{u}$ everywhere and ultimately
		\begin{equation*}
		\int_\V \abs{u} \, dx  \leq C_2\ \norm{u}_\Lpnu + C_1 C_3/k! \ \norm{u}_\Hnukp{k} \leq C \ \norm{u}_\Hnukp{k}
		\end{equation*} 
		establishing the inequality \eqref{eq:ineq_for_embedding}. One may easily verify that for $p = 1$ the proof of the inequality runs identically up to redefining the constants as  $C_1 := \norm{P_k/w}_{L^\infty(\V)}$ and $C_2 := \norm{1/w}_{L^\infty(\V\backslash \U_1)}$, where $P_k(x) = (x-x_0)^k$.
		
		In the remainder of the proof the function $u$ will be an element of the weighted Sobolev space $\Hnukp{k}$. Let $u_h \in \D(\R)$ denote a sequence of smooth functions that converges to $u$ in $\Hnukp{k}$. The sequence is Cauchy in $\Hnukp{k}$, hence by the inequality \eqref{eq:ineq_for_embedding} $u_h$ is also Cauchy in $L^1(\V)$ and thus has there a limit which we denote by $\bar{u}$. On the other hand $u_h \rightarrow u$ in $\Lpnu$ and, since $\bigl( \V \backslash \{x_0\}\bigr) \cap \cIcr{0,p}(w) = \varnothing$, the embedding $\Lpnu \hookrightarrow \Lloc{\V \backslash \{x_0\}}$ furnishes (up to choosing a subsequence)  $u_h(x) \rightarrow u(x)$ a.e. in $\V$. Therefore there must hold $u = \bar{u}$ a.e. and the inequality \eqref{eq:ineq_for_embedding} extends to $\Hnukp{k}$ which proves the embedding \eqref{eq:embedding}.
		
		Finally, if $\cIcr{k,p}(w) = \varnothing$, then for each $x \in \bar{I}$ there exists an $\bar{I}$-open neighbourhood $\V_x$ for which the inequality \eqref{eq:ineq_for_embedding} holds for smooth functions with a constant $C_x<\infty$. The family $\{\V_x : x\in \bar{I}\}$ is an open covering for $\bar{I}$ which is compact thus we can choose a finite subcover $\{\V_{x_n} : x_n \in \bar{I} \text{ for } 1\leq n \leq N \}$. Then inequality \eqref{eq:ineq_for_embedding} holds for the set $\V = \bar{I}$ and the constant $C = \sum_{1\leq n \leq N} C_{x_n} < \infty$. The embedding $\Hnukp{k} \hookrightarrow L^1(\bar{I})$ follows from the argument used in the previous paragraph. It is worth observing that, for $\cIcr{k,p}(w) = \varnothing$ with any $k \in \mathbb{N}$, in process we have  obtained $\cIcr{0,p}(w) \subset \{ x_n : 1\leq n \leq N \}$, namely $\cIcr{0,p}(w)$ is finite.
	\end{proof}
	
	\section{On sufficient conditions for jump-type discontinuities of a function and its derivatives in higher order weighted Sobolev space $\Hnukp{m}$ on the real line}
	\label{sec:necessary_conditions}
	
	For an interval $I =\Iaa \subset \R$ and a (possibly strongly degenerate) weight $w \in L^1(I)$ we continue to look at the $m$-th order weighted Sobolev space $\Hnukp{m}$ defined through completion of the space $\D(\R)$, the exponent $p$ is any real number from $\prange$. In the previous section we have established that, at a given point $x_0 \in I$ and chosen order $\bk\in \{0,\ldots,m-1\}$, the condition $x_0 \notin \cIcr{\Dk-1,p}(w)$, with $\Dk = m-\bk$, is sufficient to deduce continuity of $u$ and all the derivatives $\tgradnuk{k}u$ up to order $k=\bk$ at the point $x_0$, $u$ being any function in $\Hnukp{m}$. Further we ask whether this condition is optimal or, in other words, if the condition $x_0 \notin \cIcr{\Dk-1,p}(w)$ is necessary to have the aforementioned continuity of $u$ and its derivatives at $x_0$. To put it yet differently, we must check if criticality $x_0 \in \cIcr{\Dk-1,p}(w)$ implies existence of $\ustep \in \Hnukp{m}$ such that for the order $\bk$ the derivative $\tgradnuk{\bk} \ustep$ is discontinuous at $x_0$. More precisely we will verify whether there exists a function $\ustep \in \Hnukp{m}$ with
	\begin{equation}
	\label{eq:step_function}
	\tgradnuk{\bk}\ustep = \mathbbm{1}_{(x_0,\ap)} \qquad \muw\text{-a.e.}
	\end{equation}
	We agree that henceforward the symbol $\ustep$ will be consistently used to denote a candidate for a function from $\Hnukp{m}$ with $\bk$-th tangential derivative being a step function \eqref{eq:step_function}.
	
	\subsection{Application of Legendre-Fenchel transformation to showing potential discontinuities of functions in the first order weighted Sobolev space $\Hnukp{1}$}
	\label{sec:LFconjugate}

	Before stating the result for arbitrary $m \geq 1$ and $\bk \in \{0,\ldots,m-1\}$ we will look into the case of the first order Sobolev space  $\Hnukp{1}$, we thus specify $m=1$, $\bk=1$, $\Dk-1 =\linebreak m-\bk-1 =0$. Essentially, for a point $x_0 \in I$ we ask if the criticality $x_0 \in \cIcr{0,p}(w)$ guarantees that $\ustep = \mathbbm{1}_{(x_0,\ap)}$ is an element of $\Hnukp{1}$. In order to answer this question positively we must, by definition, find a sequence of smooth functions $\ustep_h \in \D(\R)$ such that $\ustep_h \rightarrow \ustep$ and $D \ustep_h \rightarrow v$ in $\Lpnu$ for some $v$. This was achieved for weights of the class $w_\gamma$ and $w_\mathrm{log}$, see the conclusions \eqref{eq:disc_Icr_wa} and \eqref{eq:disc_Icr_wlog} in Examples \ref{ex:wa} and \ref{ex:wlog} respectively. In both cases $v$, that is the tangential derivative of the step function $\tgradnuk{}\ustep$, turned out to be zero in $\Lpnu$. From those examples we learn that for a general weight $w \in L^1(I)$ we should seek a sequence of functions $\hphi_h \in \D(\R)$ such that
	\begin{equation}
	\label{eq:sequence_phi}
	\hphi_h \geq 0, \quad \mathrm{spt}\bigl( \hphi_h \bigr) \subset B(x_0,1/h), \quad \int_I \hphi_h \, dx = 1, \quad \norm{\hphi_h}_\Lpnu \xrightarrow{h \rightarrow \infty} 0.
	\end{equation}
	The first three conditions describe a sequence that approximates a Dirac delta measure at $x_0$. If such a sequence $\hphi_h$ exists, then by defining $\ustep_h(x):= \int_{-\infty}^x \hphi_h(y) \, dy$ we obtain $\ustep_h \rightarrow \ustep$ in $\Lpnu$ from the Lebesgue dominated convergence theorem. The forth condition guarantees that $D \ustep_h = \hphi_h \rightarrow 0$ in $\Lpnu$, which eventually (upon restricting $\ustep_h$ to a compact support in \nolinebreak $\R$) establishes that $\ustep \in \Hnukp{1}$.
	
	In order to show that, for a given weight $w$, the existence of a sequence \eqref{eq:sequence_phi} for $x_0$ stems from the fact that $x_0 \in \cIcr{0,p}(w)$ we must find a more intrinsic relation between the two properties of the point $x_0 \in I$. The idea proposed herein puts them in duality.
	
	For a weight $w \in L^1(I)$, an exponent $p \in \prange$ and an open subset $\U \subset I$ we define a convex energy functional $\J{\U}: L^p(\U) \rightarrow \Rb$, where $\Rb= \R \cup \{-\infty,\infty\}$. We stress that the Lebesgue space $L^p(\U)$ is intended with respect to Lebesgue measure instead of the measure with density $\muw$. For any $ v \in L^p(\U)$ we put
	\begin{equation}
	\label{eq:J_def}
	\J{\U}(v) :=\frac{1}{p} \int_{\U} w(x) \, \abs{v(x)}^p\, dx,
	\end{equation}
	which can be alternatively written as $	\J{\U}(v) = \int_{\U} f\bigl(x,v(x) \bigr)\, dx$ with the integrand $f: \nolinebreak \U \times \nolinebreak \R \rightarrow \nolinebreak \Rb$ defined by $f(x,\nu) :=\frac{1}{p}\, w(x) \, \abs{\nu}^p$.
	The convex conjugate, or Legendre-Fenchel conjugate of the functional $\Jc{\U}: L^{p'}(\U) \rightarrow \Rb$ is defined for any $v^* \in L^{p'}(\U)$ by the formula
	\begin{equation}
	\label{eq:Jc_def}
	\Jc{\U}\bigl(v^*\bigr) := \sup_{v \in L^p(\U)} \biggl\{\, \int_{\U} \, v \ v^* \, dx \ - \ \J{\U}(v) \,   \biggr\}
	\end{equation}
	where we have used the fact the integral of the product $v \, v^*$ is a natural duality pairing for the pair $\langle L^p,L^{p'} \rangle$; note that the weight $w$ is missing from the integral. The integrand $f$ is normal for every $p \in \prange$ and from the celebrated result by Rockafellar (see Theorem 2 in \cite{rockafellar1968}) we find that the operations of conjugation and integration in \eqref{eq:Jc_def} commute, more precisely $\Jc{\U}\bigl(v^*\bigr) = \int_\U f^*\bigl(x,v^*(x)\bigr) dx$ where the integrand $f^*: \U \times \R \rightarrow \Rb$ denotes the Legendre-Fenchel conjugate of $f(x, \,\cdot\,)$ with respect to the second argument. The closed formula for $f^*$ requires handling the case $p=1$ separately: for every $x\in \U$ and \nolinebreak $\nu^* \in \R$
	\begin{alignat*}{1}
	f^*\bigl(x,\nu^*\bigr) &= \frac{1}{p'} \left(\frac{\abs{\nu^*}}{(w(x))^{1/p}}\right)^{p'} \hspace{1.65 cm} \text{for } p\in (1,\infty), \\
	f^*\bigl(x,\nu^*\bigr) &=
	\left\{
	\begin{array}{cl}
	0 &\quad \mathrm{if}\ \ \abs{\nu^*}/w(x) \leq 1, \\
	\infty &\quad \text{otherwise}
	\end{array} \right. \hspace{0.6 cm} \text{for } p =1,
	\end{alignat*}
	hence the formula for $\Jc{\U}$ follows for $v^* \in L^{p'}(\U)$:
	\begin{alignat}{1}
	\label{eq:Jc_form_p}
	\Jc{\U}\bigl(v^*\bigr) &=  \frac{1}{p'} \int_\U \left(\frac{\abs{v^*(x)}}{(w(x))^{1/p}}\right)^{p'}\!dx \hspace{2cm} \text{for } p\in (1,\infty),\\
	\nonumber\\
	\label{eq:Jc_form_1}
	\Jc{\U}\bigl(v^*\bigr) &=
	\left\{
	\begin{array}{cl}
	0 &\quad \mathrm{if}\ \ \norm{v^*/ \, w}_{L^\infty(\U)} \leq 1, \\
	\infty &\quad \text{otherwise}
	\end{array} \right. \hspace {1.1cm}\text{for } p =1.
	\end{alignat}
	
	Next assume a point $x_0 \in I$ which is critical, i.e. $x_0 \in \cIcr{0,p}(w)$. For a fixed $\eps>0$ we denote an open neighbourhood $\Ue := B(x_0,\eps) \cap I$ and propose
	\begin{equation*}
	v^*_\eps := \eps = \mathrm{const} \quad \text{in } L^{p'}(\Ue). 
	\end{equation*}
	We will show that $\Jc{\Ue}\!\bigl( v_\eps^* \bigr) = \infty$ for any $p \in \prange$. It is straightforward that
	\begin{equation}
	\label{eq:infty_Jc_p}
	\text{for } p\in (1,\infty) \qquad \Jc{\Ue}\!\bigl( v_\eps^* \bigr) = \frac{\eps^{p'}}{p'} \int_{B(x_0,\eps)} \frac{1}{w^{p'/p}}\,  dx = \infty
	\end{equation}
	by the very definition \eqref{eq:def_Icrap} of the critical set $\cIcr{0,p}(w)$. In the case of $p = 1$ we need an extra argument: if $x_0 \in \cIcr{0,1}(w)$, then for arbitrarily small $\eps>0$ there exists a subset $A \subset \Ue = B(x_0,\eps) \cap I$ with positive Lebesgue measure such that $1/w(x) > 1/\eps$ for every $x\in A$, therefore
	\begin{equation}
	\label{eq:infty_Jc_1}
	\norm{v_\eps^* / w }_{L^\infty(\Ue)} \geq \norm{\eps / w }_{L^\infty(A)} > 1, \quad \text{hence} \quad \Jc{\Ue}\!\bigl( v_\eps^* \bigr) = \infty \quad \text{for } p=1.
	\end{equation}
	We can next confront the above results \eqref{eq:infty_Jc_p}, \eqref{eq:infty_Jc_1} with the general definition \eqref{eq:Jc_def} of the conjugate functional $\Jc{\Ue}$ at $v^* = v^*_\eps \equiv \eps$. As a result, for arbitrarily small $\eps >0$, we infer the existence of a sequence as follows:
	\begin{equation}
	\label{eq:sequence_vh}
	v_h \in L^p(\Ue), \qquad v_h \geq 0, \qquad \eps \int_\Ue \!v_h \, dx \ - \ \J{\Ue}(v_h) \ \xrightarrow{h \rightarrow \infty} \ 0,
	\end{equation}
	where we recall that $\Ue = B(x_0,\eps) \cap I$. Non-negativity of $v_h$ follows from the fact that otherwise we could always take the absolute value. Having the sequence $v_h$ at our disposal we are almost done with finding the sequence $\hphi_h$ satisfying \eqref{eq:sequence_phi} and thus proving that $\ustep$ is an element of $\Hnukp{1}$. We must obtain a smooth version $\phi_h$ of $v_h$ and then rescale it to $\hphi_h$. The first step shows small trouble, as the smooth $\phi_h$ must be close to $v_h$ in terms of both $\lvert\int_{\Ue} \phi_h dx - \int_{\Ue} v_h dx \rvert$ and $\abs{\J{\Ue}(\phi_h) - \J{\Ue}(v_h)}$. We observe that $\J{\Ue}(v) = 1/p\, \norm{v}^p_\Lpnu$ for any $v \in L^p(\Ue)$. For every $\delta>0$ each function $v_h$ can be smoothly approximated so that $\abs{\J{\Ue}(\phi_h) - \J{\Ue}(v_h)} < \delta$, but this is not enough since, in extreme case, $w$ could be zero function and then we would have no control over $\abs{\int_{\Ue} \phi_h dx - \int_{\Ue} v_h dx}$. We thus require
	\begin{equation}
	\label{eq:add_assum_Linfty}
	\text{an additional assumption on the weight } w\text{:} \quad \exists r>0 \text{ such that } w \in L^\infty\bigl( B(x_0,r) \cap I \bigr),
	\end{equation}
	in which case, for $\eps\leq r$, we have the continuous embedding $L^p(\Ue) \hookrightarrow \Lpnu$. We may now mollify $v_h$ by standard convolution, obtaining for arbitrary $\delta>0$ a function $\phi_h$ (with possibly slightly larger support than $\Ue$, due to arbitrariness of $\eps$ this fact is irrelevant and thus skipped later) such that $\norm{v_h - \phi_h}_{L^p(\Ue)} < \delta$. Both the terms in the divergent sequence \eqref{eq:sequence_vh} are thus well approximated with $v_h$ replaced by $\phi_h$. Obviously each $\phi_h$ is non-negative, as it was obtained by mollification of a non-negative function $v_h$.
	
	Since $\eps$ in \eqref{eq:sequence_vh} (and also for the sequence of smooth $\phi_h$) is arbitrary, through a diagonalization argument we can choose $\phi_h$ so that $\phi_h \in \D(\U_{\eps_h})$ and $\eps_h \int_{\U_{\eps_h}} \!\phi_h \, dx \ - \ \J{\U_{\eps_h}}(\phi_h) \rightarrow \infty$ with $\eps_h :=1/h$. We put $\Delta_h := \int_{\U_{\eps_h}} \!\phi_h \, dx$ and rescale our sequence:
	\begin{equation*}
	\hphi_h := \frac{1}{\Delta_h} \phi_h.
	\end{equation*}
	Obviously $\int_{B(x_0,1/h)} \!\hphi_h \, dx = 1$ for every $h$, whilst
	\begin{alignat}{1}
	\nonumber
	\eps_h \Delta_h - \J{\U_{\eps_h}}\!(\phi_h) > 0 \qquad &\Rightarrow \qquad \eps_h \Delta_h - \Delta_h^p \  \J{\U_{\eps_h}}\!(\hphi_h) > 0 \\
	\label{eq:estimate_on_J}
	&\Rightarrow \qquad \J{\U_{\eps_h}}\!(\hphi_h) < \frac{\eps_h}{\Delta^{p-1}_h}.
	\end{alignat}
	We recall that $\eps_h =1/h$ and $\Delta_h$ necessarily diverge to infinity, therefore $\J{\U_{\eps_h}}\!(\hphi_h) \rightarrow 0$ for any $p \in \prange$ or equivalently $\norm{\hphi_h}_\Lpnu \rightarrow 0$. We have found a sequence $\hphi_h$ that precisely satisfies the conditions \eqref{eq:sequence_phi}. 
	This establishes that the step function $\ustep = \mathbbm{1}_{(x_0,\ap)}$ is an element of the weighted Sobolev space $\Hnukp{1}$.
	
	\begin{remark}
		\label{rem:p_and_1_eps_h}
		We give a short comment on the choice of the function $v_\eps^* = \eps $ above. Eventually it has landed as $\eps_h = 1/h$ in the inequality \eqref{eq:estimate_on_J} that estimates $\J{\U_{\eps_h}}\!(\hphi_h)$. This inequality was to yield $\norm{\hphi_h}_\Lpnu \rightarrow 0$; we note that for $p>1$ it would still have done so notwithstanding $\eps_h$, which is due to $\Delta_h \rightarrow \infty$. In summary, for $p >1$ it was enough to take $v_\eps^* = 1 $ while in case of $p=1$ the trick with $\eps$ was essential. The same idea will motivate $\eps$ in \eqref{eq:sequence_phi_k} in the proof of Theorem \ref{thm:dirac_delta_approx_crit_point}.
	\end{remark}
	\begin{remark}
		For an interval (or in fact any open set) $I \subset \R$, an exponent $p\in \prange$ and a weight $w \in L^1(I)$ we write down the two following statements:
		\begin{enumerate}[(i)]
			\item the critical set is empty, namely $\cIcr{0,p}(w) = \varnothing$;
			\item the embedding $\Lpnu \hookrightarrow L^1(I)$ holds.
		\end{enumerate}  
		The implication (i) $\Rightarrow$ (ii) has been showed (up to locality) already by \cite{kufner1984} by means of H\"{o}lder inequality, see \eqref{eq:B_p_estimate}, and was repeatedly used here in Section \nolinebreak \ref{sec:sufficient_conditions}. If on the other hand the critical set $\cIcr{0,p}(w)$ is non-empty then our duality argument above has showed that there exists a sequence $v_h$ satisfying \eqref{eq:sequence_vh}. Then, upon  rescaling by $\Delta_h := \int_I v_h \,dx$, the sequence $\hat{v}_h:= v_h/\Delta_h$ satisfies $\norm{\hat{v}_h}_{L^1(I)} = 1$ and $\norm{\hat{v}_h}_{\Lpnu} \rightarrow 0$. This disqualifies the embedding (ii) and ultimately puts the two statements (i) and (ii) in equivalence. For the case $p=2$ this result was given in \cite{louet2014}, see Lemma 2.4 in this work. To prove the lemma Louet has used more elementary methods of measure theory and functional analysis. The approach proposed herein sheds light on the duality relation between (i) and (ii) by means of the  Legendre-Fenchel transformation.
	\end{remark}
	
	\subsection{Examining jump-type discontinuities of a function and its derivatives in higher order weighted Sobolev space $\Hnukp{m}$ via general duality theory. A stability assumption on the weight}
	
	\label{sec:general_duality}
	
	The previous subsection, that concerned the first order space $\Hnukp{1}$ only, was aimed to serve rather as demonstration of applying duality in examining the step functions $\ustep$ as elements of the weighted Sobolev space. For a higher order space $\Hnukp{m}$ and its function $\ustep$ we wish to infer that the fact $x_0 \in \cIcr{\Dk -1}$ makes it possible for the derivative $\tgradnuk{\bk} \ustep$ to be a step function $\mathbbm{1}_{(x_0,\ap)}$, recall that $\Dk = m-\bk$. First of all we must show that there exists a sequence $\hphi_h$ satisfying \eqref{eq:sequence_phi}, then we may define a sequence $\ustep_h$ so that $D^\bk \ustep_h(x) := \int_{-\infty}^x \hphi_h(y) \,dy$.  We obtain that $D^\bk \ustep_h \rightarrow \mathbbm{1}_{(x_0,\ap)}$ in $\Lpnu$ as desired and also $D^{\bk+1} \ustep_h = \hphi_h \rightarrow 0$ in $\Lpnu$. This, however, is not enough for $\ustep_h$ to converge to $\ustep$ in $\Hnukp{m}$. For that we need all the derivatives $D^k \ustep_h$ for $k\in \{\bk+1,\ldots,m\}$ to converge to zero in $\Lpnu$. In other words all the functions $\hphi_h, D\hphi_h,\ldots, D^{m-(\bk+1)} \hphi_k = D^{\Dk-1} \hphi_h$ must converge to zero. It is already established that the convergence of $\hphi_h$ is guaranteed whenever $x_0 \in \cIcr{0,p}(w)$. Now we must show that for convergence of the highest derivative: $D^{\Dk-1} \hphi_h \rightarrow 0$ in $\Lpnu$ all we require is $x_0 \in\cIcr{\Dk-1,p}(w)$. For that purpose we shall apply duality once more, only this time we will employ a more general theory of duality in calculus of variation (cf. \cite{ekeland1999}), which will considerably automate the proof, for instance it will furnish a sequence of functions that are already smooth, as opposed to \eqref{eq:sequence_vh}. Beforehand, for the sake of generality, we shall specify the definition of the critical point $x_0$ with respect to the side of $x_0$ where the degeneration of the weight $w$ occurs.
	
	Further we agree for the following notation: by $\Bl(x_0,\eps):= (x_0-\eps,x_0)$ we shall see the \textit{left open half-ball} around $x_0$ and, analogically, $\Br(x_0,\eps):= (x_0,x_0+\eps)$ will denote the \textit{right open half-ball}. During the construction of a sequence $\hphi_h$ in Section \ref{sec:LFconjugate} we have concluded that we need an additional assumption on the weight $w$: it had to be essentially bounded in some neighbourhood of $x_0$, see \eqref{eq:add_assum_Linfty}. The short comment at the end of the Example \ref{ex:wa} illustrates that this is too restrictive, since the weight can degenerate on e.g. right side of $x_0$ and blow up to infinity on the left side. In this scenario the sequence $\hphi_h$ can be  supported in the right half-balls $\Br(x_0,1/h)$. We require the following definitions:
	\begin{definition}
		\label{def:Icr_left_right}
		For a given weight $w \in L^1(I)$, an exponent $p \in [1,\infty)$ and an order $\al \geq \nolinebreak 0$ we shall say that $x_0 \in I$ is a \textit{right-sided} (or \textit{left-sided}) critical point, which will be denoted by $x_0 \in \cIcrr{\al,p}(w)$ (or $x_0 \in \cIcrl{\al,p}(w)$), whenever $x_0$ belongs to the sets in definitions (\ref{eq:def_Icrap},\ref{eq:def_Icra1}) of $\cIcr{\al,p}(w)$  with the ball $B(x_0,\eps)$ replaced by the half-ball $\Br(x_0,\eps)$ (or the half-ball $\Bl(x_0,\eps)$).
	\end{definition}
	
	It is clear that $x_0 \in \cIcrl{\al,p}(w)$ or $x_0 \in \cIcrr{\al,p}(w)$ implies $x_0 \in \cIcr{\al,p}(w)$. Conversely, if $x_0 \in \cIcr{\al,p}(w)$, then we have $x_0 \in \cIcrl{\al,p}(w)$ or $x_0 \in \cIcrr{\al,p}(w)$ or both. In case of boundary points of the interval $I = \Iaa$ the left end-point $\am$ is a critical point if and only if it is right-sided and the right end-point $\ap$ is critical if and only if it is left-sided.
	
	We will give a theorem that by means of duality furnishes a sequence $\hphi_h$ that approximates Dirac delta around a critical point $x_0$ with a small $\Lpnu$-norm of its $k$-th derivative. Beforehand we need a technical lemma which is obvious for smooth functions:
	\begin{lemma}
		\label{lem:positive_distribution}
		Let us be given a point $x_0$ on the real line $\R$. For a fixed $k \in \mathbb{N}$ we consider a set
		\begin{equation*}
		V_k := \left\{ v \in \Lloc{\R} \ : \ D^k v - \mathbbm{1}\ \text{ is a non-negative distribution in } \D'(\R)  \right\},	
		\end{equation*}
		where $\mathbbm{1}$ denotes a regular distribution induced by a constant function $1 \in \Lloc{\R}$.
		
		Then, for each element $v \in V_k$, there exists $\delta = \delta(v) >0$ such that
		\begin{equation}
		\label{eq:v_geq_Pk}
		\abs{v(x)} \geq \frac{\abs{x-x_0}^k}{k!} \qquad \text{ for a.e. } x\in B(x_0,\delta). 
		\end{equation}
		\begin{proof}
			It is enough to prove the thesis with \eqref{eq:v_geq_Pk} holding for a right half-ball $\Br(x_0,\delta)$ only. In the proof we do not distinguish distributions and measures or functions inducing them.
			
			The constraint $v \in V_k$ implies that $D^k v$ is a positive distribution itself and thus, by a version of Riesz representation theorem,  $\mu := D^k v$ is a positive Radon measure (see e.g. \cite{rudin1991}). Moreover $\mu(A) \geq \mathcal{L}^1(A)$ for every Borel set $A \subset \R$. In the case when $k = 0$ this is equivalent to $v$ being greater or equal to 1 a.e. in $\R$, which gives \eqref{eq:v_geq_Pk}.
			
			Further we may thus assume that $k \geq 1$; then $D^{k-1} v$ is an increasing and locally bounded function on $\R$. Therefore we may work with a right-sided continuous representative of $D^{k-1} v$ for which the formula below holds for each $x > x_0$:
			\begin{equation*}
			\bigl(D^{k-1} v\bigr) \!\bigl( x \bigr) = \bigl(D^{k-1} v\bigr)\!\bigl(x_0^+\bigr) + \mu \bigl((x_0,x] \bigr)
			\end{equation*}
			where $\left(D^{k-1} v\right)\!\left(x_0^+\right)$ is the right-sided limit. We notice that there exists $\bar\delta > 0$ such that
			\begin{equation}
			\label{eq:geq_x_x_0}
			\abs{\bigl(D^{k-1} v\bigr)\! \bigl(x\bigr)} \geq \abs{x-x_0} \qquad \forall \, x \in B_+(x_0,\bar\delta).
			\end{equation}
			Indeed, if $ \left(D^{k-1} v\right)\!\left(x_0^+\right) \neq 0$ this follows from right-sided continuity of $D^{k-1} v$; if, however, $\left(D^{k-1} v\right)\!\left(x_0^+\right) = 0$ then simply $\abs{\left(D^{k-1} v\right) (x)} = \mu \bigl((x_0,x]\bigr) \geq \mathcal{L}^1 \bigl((x_0,x] \bigr) = \abs{x-x_0}$ and this establishes the thesis for $k=1$.
			
			Now we look at the case when $k\geq 2$; since $D^{k-1} v \in \Lloc{\R}$ we observe that $v \in \nolinebreak C^{k-2}(\R)$ and $D^{k-2} v$ is locally absolutely continuous. From the Taylor expansion it is clear that whenever $\left(D^n v^* \right)(x_0^+) \neq 0$ for some $n \in \{0, \ldots , k-2\}$, there exists $\delta$ such that \eqref{eq:v_geq_Pk} holds. We assume otherwise and then the Taylor expansion for $x \geq x_0$ reduces to
			\begin{equation*}
			v(x) = \int_{x_0}^x \bigl(D^{k-1} v\bigr)\!(y) \ \frac{(x-y)^{k-2}}{(k-2)!}  \ dy 
			\end{equation*}
			which by \eqref{eq:geq_x_x_0} easily gives $\abs{v(x)} \geq \abs{x-x_0}^k/k!$ for every $x \in B_+(x_0,\bar{\delta})$. The proof concludes here.
		\end{proof}
	\end{lemma}
	\begin{theorem}
		\label{thm:dirac_delta_approx_crit_point}
		Let us take an open interval $I = \Iaa \subset \R$, an exponent $p \in \prange$ and a weight $w \in L^1(I)$. We choose a point $x_0 \in \bar{I}$ and an order $k \in \mathbb{N}_+ \cup \{0\}$. We assume that there exists $r>0$ such that $w \in L^\infty\bigl(\Br(x_0,r) \cap I)\bigr)$.
		
		The following claim holds: if $x_0 \in \cIcrr{k,p}(w)$, then there exists a sequence $\hphi_h$ satisfying:
		\begin{equation}
		\label{eq:sequence_hphi_k}
		\hphi_h \in \D\bigl( \Br(x_0,1/h) \bigr), \quad \hphi_h \geq 0, \quad \int_I \hphi_h \, dx = 1, \quad \norm{\der{k}{\hphi_h}}_\Lpnu \xrightarrow{h \rightarrow \infty} 0.
		\end{equation}
		The same result can be obtained for the left-sided critical point $x_0 \in \cIcrl{k,p}(w)$ provided the half-balls $\Br$ above are replaced with the half-balls $\Bl$.
		\begin{proof}
			We will display the proof only for the case when $x_0 \in \cIcrr{k,p}(w)$ since in the case of a left-sided critical point $x_0$ the argument runs analogically. We note that such $x_0 \in \bar{I}$ cannot be the right end-point $\ap$. We will construct the sequence $\hphi_h$ starting from the index $h_0$ for which $\Br(x_0,1/h_0) \subset \Br(x_0,r) \cap I$.
			
			For a fixed $\eps\leq r$ we set $\Ue := \Br(x_0,\eps) \cap I$. We will show that there exists a sequence of smooth functions $\phi_n$ satisfying
			\begin{equation}
			\label{eq:sequence_phi_k}
			\phi_n \in \D\bigl( \Ue \bigr),\qquad \phi_n \geq 0, \qquad \eps \int_\Ue \phi_n \, dx - \J{\Ue}\bigl( \der{k}{\phi_n} \bigr) \ \xrightarrow{n \rightarrow \infty} \ \infty,
			\end{equation}
			where the energy functional $\J{\Ue}:L^p\bigl( \Ue \bigr) \rightarrow \Rb$ is defined in \eqref{eq:J_def}. 
			We explain how our thesis follows from the existence of such a sequence $\phi_n$. Starting from $h \geq h_0$ we may put $\eps_h := 1/h$ in \eqref{eq:sequence_phi_k} above and for each $h$ we find a respective sequence $\phi_{h,n}$ that varies in \nolinebreak $n$. The diagonalization argument yields a non-negative sequence $\phi_{h,h} \in \D\bigl( \Br(x_0,1/h) \bigr)$ that varies in $h$, then $\eps_h \int_{\U_{\eps_h}} \phi_{h,h} dx - \J{\U_{\eps_h}}\bigl(\der{k}{\phi_{h,h}}\bigr) \rightarrow \infty$. We put $\Delta_h:= \nolinebreak \int \phi_{h,h}\, dx$ and define a rescaled sequence $\hphi_h \in \D\bigl( \Br(x_0,1/h) \bigr)$
			\begin{equation}
			\label{eq:scaling}
			\hphi_h := \frac{1}{\Delta_h} \, \phi_{h,h}.
			\end{equation}
			An estimate analogical to \eqref{eq:estimate_on_J} yields $\J{\U_{\eps_h}}\bigl(\der{k}{\hphi_{h}}\bigr) < \eps_h / \Delta_h^{p-1}$. Since $\eps_h = 1/h \rightarrow \nolinebreak 0$ and $\Delta_h \rightarrow \infty$, we obtain $\J{\U_{\eps_h}}\bigl(\der{k}{\hphi_{h}}\bigr) \rightarrow 0$ and thus $\norm{\der{k}{\hphi_{h}}}_\Lpnu \rightarrow 0$. We have therefore constructed a sequence $\hphi_h$ satisfying \eqref{eq:sequence_hphi_k}.
			
			We have showed that the proof of the theorem boils down to pointing to a sequence $\phi_n$ in accordance with \eqref{eq:sequence_phi_k}. We stress that throughout the rest of the proof $\eps>0$ together with $\Ue= \Br(x_0,\eps) \cap I$ stay fixed, in addition we assume that $\eps \leq r$. While entering the duality theory we shall employ the notation from Chapter III in \cite{ekeland1999}. We put the pairs of spaces: $X := \D\bigl(\Ue \bigr)$, $X^*:= \D'\bigl(\Ue \bigr) $ and $Y := L^p\bigl(\Ue \bigr) ,$  $Y^* :=L^{p'}\bigl(\Ue \bigr) $ in duality with their standard pairings/topologies. Moreover we denote a continuous linear operator $\Lambda := (-1)^{k} D^k: X \rightarrow Y$ where $D^k$ is the classical $k$-th derivative (we multiply by the factor $(-1)^{k}$ for convenience further in); the conjugate operator $\Lambda^*: Y^* \rightarrow X^*$ is well defined and it is expressed by $\Lambda^* = D^k$ with the derivative intended in the distributional sense. In addition we introduce a closed half-space $X_+ := \nolinebreak \left\{ \phi \in X : \phi \geq 0 \right\}$. We define for every $v\in Y = L^p\bigl(\Ue \bigr) $ a $v$-perturbed variational problem
			\begin{equation}
			\label{eq:h}
			h(v) := \inf \biggl\{ F(\phi) + G\left(\Lambda \phi + v\right) \ : \ \phi \in X \biggr\},
			\end{equation}
			where $F: X \rightarrow \Rb$, $\,G:Y \rightarrow \Rb$  and for any $\phi \in X$ and $v \in Y$
			\begin{equation}
			\label{eq:fun_F}
			F(\phi) := - \eps \int_\Ue \phi \, dx + \mathbb{I}_{X_+}\! \left(\phi\right),\qquad G(v):= \J{\Ue}(v),
			\end{equation}
			where $\mathbb{I}_{X_+}$ denotes the indicator function of $X_+$. We recognize that the existence of a sequence $\phi_n$ satisfying \eqref{eq:sequence_phi_k} is equivalent to $-h(0)$ being equal to $+\infty$. For the time being let us assume that it is not, i.e. that $h(0) > -\infty$. By using a standard duality argument (see formula (4.18) in Chapter III of \cite{ekeland1999}) we arrive at the dual to the problem $h(0)$ with respect to the perturbation $v$:
			\begin{equation}
			\label{eq:h_star_star_abstract}
			h^{**}(0) = \sup \biggl\{ -F^*\!\left(-\Lambda^*v^*\right) - G^*\!\left(v^*\right) \ : \ v^* \in Y^* \biggr\}
			\end{equation}
			where $F^*,G^*$ are the Legendre-Fenchel conjugates of $F,G$. Obviously the functionl $G^*$ is equal to $\Jc{\Ue}$ whereas its explicit formula can be found in \eqref{eq:Jc_form_p} and \eqref{eq:Jc_form_1} for $p>1$ and $p=1$ respectively. For any $\phi^* \in X^* = \D'\bigl( \Ue \bigr)$ we compute
			\begin{equation}
			\label{eq:F_star}
			F^*(\phi^*) = \sup\left\{ \pairing{\phi,\phi^*} + \eps \int_\Ue \phi \, dx \ : \ \phi \in X_+ \right\} =
			\left \{
			\begin{array}{cl}
			0 \quad & \text{if} \quad \phi^* + \eps \,\mathbbm{1}_\Ue \in X_+^{\,0}, \\
			+ \infty  \quad & \text{otherwise},
			\end{array} \right.
			\end{equation}
			where by $X_+^{\,0} \subset \D'\bigl(\Ue \bigr) $ we see the polar of $X_+$, whereas $\mathbbm{1}_\Ue$ is a distribution induced by the characteristic function of $\Ue$. The functional $G$ is convex and continuous in $ Y = L^p\bigl(\Ue \bigr) $; its continuity follows from the fact that $w \in L^\infty(\Ue)$ (recall that $\eps \leq r$ and compare the assumptions of the theorem). Together with the assumption on $h(0)$ being finite we obtain stability of the duality problem which furnishes $h^{**}(0) = h(0)$, see Theorem 4.1 in Chapter III of \cite{ekeland1999}. By a careful substitution in \eqref{eq:h_star_star_abstract} we arrive at
			\begin{equation}
			\label{eq:h_star_star}
			-h(0) = -h^{**}(0) = \inf\biggl\{  \Jc{\Ue}(v^*)\ : \ v^* \in L^{p'}\bigl(\Ue \bigr) , \ -  D^k v^* + \eps \, \mathbbm{1}_\Ue \in X_+^{\,0}  \biggr\}.
			\end{equation}
			
			We will show that $\Jc{\Ue}(v^*) = \infty$ for every $ v^* \in L^{p'}\bigl(\Ue \bigr)$ satisfying the constraint above. We fix such a function $v^*$; our constraint says that $D^k v^* - \eps \, \mathbbm{1}_\Ue$ is a non-negative distribution on $\Ue$. From Lemma \ref{lem:positive_distribution} we infer through scaling by $\eps$ that there exists $\delta>0$ such that
			\begin{equation*}
			\abs{v^*(x)} \geq \frac{\eps}{k!} \abs{x-x_0}^{k} \qquad \text{ for a.e. } x\in \Br(x_0,\delta);
			\end{equation*}
			we may additionally require that $\delta\leq \eps$ so that $\Br(x_0,\delta) \cap I \subset \Ue$. We compute $\Jc{\Ue}(v^*)$ separately for the case when $p>1$ and $p=1$. For $p\in (1,\infty)$ the formula \eqref{eq:Jc_form_p} yields
			\begin{equation*}
			\Jc{\Ue}\bigl(v^*\bigr) \!=\!  \frac{1}{p'} \int_\Ue \left(\frac{\abs{v^*(x)}}{(w(x))^{1/p}}\right)^{p'}\!dx \geq  \frac{\eps^{p'}}{p' \left(k!\right)^{p'}} \! \int\limits_{B_+(x_0,\delta) \cap I} \! \left(\frac{\abs{(x-x_0)^{k}}}{\bigl(w(x)\bigr)^{1/p}}\right)^{p'} \! dx = \infty,
			\end{equation*}
			where the last integral is infinite by the very definition \eqref{eq:def_Icrap} (up to the right-sidedness) of the critical set $\cIcrr{k,p}(w)$. Next, for $p=1$ the definition \eqref{eq:def_Icra1} of $\cIcrr{k,1}(w)$ implies that for arbitrarily small $\delta>0$ there exists a subset $A\subset \Br(x_0,\delta) \subset \Ue$ of positive Lebesgue measure such that $\abs{x-x_0}^k/w(x) > k!/\eps$ for every $x\in A$. Then
			\begin{equation*}
			\norm{v^*/w}_{L^\infty(\Ue)} \geq  \Leb^1\text{-}\mathrm{ess\,sup} \left\{ \frac{\abs{x-x_0}^k/w(x)}{k!/\eps} : x\in A \subset \Br(x_0,\delta) \right\} > 1
			\end{equation*}
			and hence, by the formula \eqref{eq:Jc_form_1}, again $\Jc{\Ue}(v^*) = \infty$ for $p=1$.
			
			We have thus obtained that the infimum in $\eqref{eq:h_star_star}$ equals infinity, or alternatively that $h^{**}(0) = -\infty$ which contradicts stability of the duality problem -- in this setting this is possible only if $h(0) = - \infty$. Upon decoding the problem \eqref{eq:h} we have in fact proven that, for any $\eps$ satisfying $0<\eps\leq r$, there exists a sequence of smooth functions $\phi_n$ such that \eqref{eq:sequence_phi_k} holds. The proof is complete.
		\end{proof}
	\end{theorem}
	
	We recall the first paragraph of this subsection --  accordingly, Theorem \ref{thm:dirac_delta_approx_crit_point} for given $k\in \nolinebreak \mathbb{N}_+$ will prove useful in seeking discontinuities of functions in $\Hnukp{m}$ only if from properties \eqref{eq:sequence_hphi_k} ($\norm{\der{k}{\hphi_h}}_\Lpnu \rightarrow 0 $ in particular) we are able to infer $\norm{\der{n}{\hphi_h}}_\Lpnu \rightarrow 0$ for every other $n\in \{0,\ldots,k-1\}$. Although, we must remember that no condition was imposed on the measure $\muw$ (thus on the weight $w$ itself) that would guarantee a Poincar\'{e}-like inequality in the space $\Hnukp{1}$ (see \cite{hajlasz2000} or \cite{bouchitte2003}). The criticality of the point $x_0 \in \cIcr{k,p}(w)$ does not help, we give a simple example to illustrate the issue: 
	
	\begin{example}
		For an interval $I = (-2,2)$ we choose a point $x_0 = 0$. We define a comb-like weight function $w_\mathrm{comb} \in L^\infty(I)$ as follows
		\begin{equation*}
		w_\mathrm{comb}(x) := \sum_{k=0}^\infty w_0\!\left(2^k \left(x-\frac{1}{2^k}\right)\right) \quad \text{with} \quad w_0(x):= \left\{
		\begin{array}{cl}
		1 &\quad \mathrm{if}\ \ \abs{x} \leq 1/8, \\
		0 &\quad \mathrm{if}\ \ \abs{x} > 1/8;
		\end{array}
		\right.
		\end{equation*}
		the function $w_\mathrm{comb}$ is illustrated in Fig. \ref{fig:Poincare_argument_1}. It is straightforward that $x_0 =0$ is a right-sided critical point for any $p \in \prange$ and any order $\al\geq 0$, in particular $x_0 \in \cIcrr{1,p}\left( w_\mathrm{comb}\right)$.
		\begin{figure}[h]
			\centering
			\includegraphics*[width=0.7\textwidth]{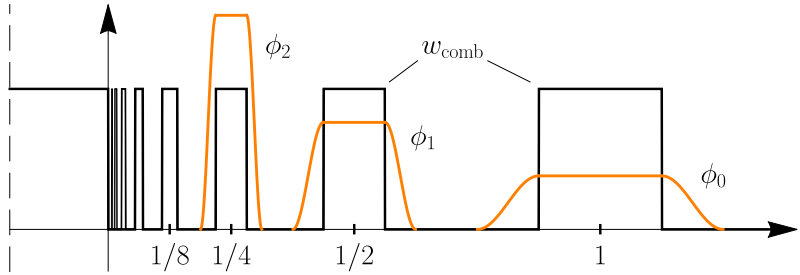}
			\caption{The comb-like weight $w_\mathrm{comb}$ and three first functions of the sequence $\hphi_h$ that smoothly approximates Dirac delta at $x_0=0$; different scales apply to $w_\mathrm{comb}$ and $\hphi_h$.}
			\label{fig:Poincare_argument_1}       
		\end{figure}
		
		Next, by $\eta_0$ we will denote any smooth function with compact support in $B(0,1/4)$ that in addition satisfies $0\leq \eta_0 \leq 1$ and  $\eta_0\equiv 1$ on $B(0,1/8)$. Then we normalize $\eta_0$ with respect to $L^1$-norm: $\eta:= \eta_0 \, /\, \norm{\eta_0}_{L^1(\R)}$. For each $h \in \mathbb{N}$ (including $h=0$) we define a function $\hphi_h \in \D\bigl(\Br(x_0,1/2^{h-1}) \bigr)$:
		\begin{equation*}
		\hphi_h(x):= 2^h \ \eta\!\left(2^h \left(x-\frac{1}{2^h}\right)\right);
		\end{equation*}
		in Fig. \ref{fig:Poincare_argument_1} we present the three first functions in the sequence $\{\hphi_h\}$.
		
		The weight $w_\mathrm{comb}$ together with the point $x_0$ meet the assumptions of Theorem \ref{thm:dirac_delta_approx_crit_point} and, for any $k \geq 1$, we easily verify that $\hphi_h$ could be the resulting sequence that satisfies \eqref{eq:sequence_hphi_k}. Indeed, we have $\int_I \hphi_h dx = 1$ and $\norm{\der{k}{\hphi_h}}_{L^p_{\mu_{w_\mathrm{comb}}}} \!= \int_I w_\mathrm{comb} \abs{\der{k}{\hphi_h}}^p dx = 0 $ for every \nolinebreak $h$. On the other hand
		\begin{equation*}
		\norm{\hphi_h}^p_{L^p_{\mu_{w_\mathrm{comb}}}} = \int_I w_\mathrm{comb} \abs{\hphi_h}^p dx  = \left(2^h C\right)^p \cdot \left(\frac{1}{4} \cdot\frac{1}{2^h}\right) = 2^{h(p-1)} \cdot \frac{C}{4},
		\end{equation*}
		where $C := \eta(0)$. To focus attention let us further choose $k=1$ and $p>1$; then we see that the sequence $\hphi_h$ cannot be utilized to show that the step function $\ustep = \mathbbm{1}_{(0,2)}$ is an element of $\Hnukp{2}$. Upon defining $\ustep_h(x) := \int_{-\infty}^x \hphi_h(y)\, dy$ we will obtain $\ustep_h \rightarrow \nolinebreak \ustep,\ \der{2}{\ustep_h} = \der{}{\hphi_h} \rightarrow \nolinebreak 0$ in $L^p_{\mu_{w_\mathrm{comb}}}$ whilst $\der{}{\ustep_h} = \hphi_h$ is therein unbounded.
	\end{example}
	
	\vspace{4mm}
	
	We should clarify what exactly the investigation of a comb-like weight above has brought; we assume a fixed $k \in \mathbb{N}_+$. We have not yet pointed to a weight $\wcrit$ for which, under the assumptions of Theorem \ref{thm:dirac_delta_approx_crit_point} (the assumption $x_0 \in \cIcr{k,p}(\wcrit)$ in particular), there is no sequence $\hphi_h$ satisfying \eqref{eq:sequence_hphi_k} and additionally $\norm{\der{n}{\hphi_h}}_{L^p_{\mu_\wcrit}}\rightarrow 0$ for all $n \in \{0,\ldots,k-2\}$. We have merely showed that such a sequence cannot be found with the use of Theorem \ref{thm:dirac_delta_approx_crit_point} as such. For the weight $w_\mathrm{comb}$ such a sequence does exists trivially, since functions $\phi_h$ can be squeezed into subsequent gaps where the weight is zero. Notwithstanding this, the critical weight $\wcrit$ spoken of above indeed can and will be constructed in the last part of proof of Theorem \ref{thm:jump_summary}. At this point we shall introduce an extra assumption on the behaviour of weight $w$ around a critical point $x_0$. Our goal is to locally retrieve a Poincar\'{e}-like inequality that is to furnish convergence of lower-order derivatives based solely on $\norm{\der{k}{\hphi_h}}_{\Lpnu}\rightarrow 0$. The example of comb-like weight suggests that it is the oscillation of the weight around $x_0$ that deprives us of this inequality. To keep a fair generality of the additional condition  we bear in mind that it should matter only on that side of $x_0$ where the weight degenerates -- we propose
	
	\begin{definition}
		\label{def:stab_crit_point}
		For a given interval $I = \Iaa \subset \R$, an exponent $p \in \prange$ and a weight $w \in L^1(I)$ we choose a point $x_0 \in \bar{I}$ that is a right-sided (or left-sided) critical point, namely $x_0 \in \cIcrr{0,p}(w)\ $ (or  $x_0 \in \cIcrl{0,p}(w)$).
		
		The point $x_0$ shall be called a \textit{right-sided stable} (or \textit{left-sided stable}) critical point if and only if there exists $r>0$ for which
		\begin{equation*}
		w=w(x) \text{ does not decrease with distance } \abs{x-x_0} \text{ in  } \Br(x_0,r)\cap I\ \  \bigl(\text{or } \Bl(x_0,r) \cap I \bigr),
		\end{equation*}
		where by monotonicity of $w$ we understand that there exists its a.e. equal representative $\breve{w}$ that is monotonic in the classical sense.
		\vspace{2mm}
		
		Moreover, we will shortly say that a critical point $x_0 \in \cIcr{0,p}(w)$ is \textit{stable} provided it is either right-sided stable or left-sided stable. 
	\end{definition}
	
	\begin{remark}
		One may easily notice that we may equivalently replace the condition of non-decreasing of the weight $w$ above by monotonicity only, the latter implies existence of a one-sided limit which has to be zero due to definition of a critical point -- therefore a non-negative monotonic function $w$ must increase with distance from $x_0$.
	\end{remark}
	
	The weights $w$ whose every critical point $x_0 \in \cIcr{0,p}(w)$ is stable shall be called \textit{stable weights}. We wish to henceforward deal with stable weights only, however we recall that we have already made another assumption on weights: a higher order Sobolev space $\Hnukp{m}$ was well-defined only for weights that are $\muw$-a.e. non-critical, namely $\cnoncrit$. We prove that the condition implying stability is stronger and, as a result, it suffices for handling Sobolev spaces $\Hnukp{m}$:
	
	\begin{proposition}
		Let $p \in \prange$, then every stable weight $w \in L^1(I)$ is $\muw$-a.e. non-critical, i.e. $\cnoncrit$.
		\begin{proof}
			We will only prove that $\muw\bigl(\cIcrr{0,p}(w)\bigr) = 0$. For the set of left-sided stable critical points $\cIcrl{0,p}(w)$ the proof is analogical and, since $\cIcr{0,p}(w) = \cIcrl{0,p}(w) \cup \cIcrr{0,p}(w)$, the thesis $\muw\bigl(\cIcr{0,p}(w)\bigr) = 0$ will follow.
			
			Let us thus take a right-sided stable critical point $x_0 \in \cIcrr{0,p}(w)$, then we can find $r>0$ such that $w$ is non-decreasing in $\Br(x_0,r) \subset I$. We consider, should it exist, any critical point $\tilde{x}_0 \in \cIcr{0,p}(w)$ in the half-ball $\Br(x_0,r)$ and we show that necessarily $w = 0$ a.e. in the interval $(x_0,\tilde{x}_0)$. Assume otherwise, then there would exist a subset of positive Lebesgue measure $A \subset (x_0,\tilde{x}_0)$ and a number $c>0$ such that $w(x) \geq c$ for every $x \in A$; monotonicity of $w$ would yield that $w \geq c >0$ in some neighbourhood of $\tilde{x}_0$ which disqualifies it as a critical point. We thus have proved that
			\begin{equation}
			\label{eq:zero_muw_on_ball}
			\muw\bigl( \Br(x_0,r) \cap \cIcrr{0,p}(w)  \bigr) = 0.
			\end{equation}
			The same argument can be repeated for every right-sided critical point $x \in \cIcrr{0,p}(w)$ yielding respectively a positive radius $r_x$. For convenience further we shall denote $F := \cIcrr{0,p}(w)$. We will show that $F \subset C \cup U $, where $C$ is at most countable subset of $F$ and $U :=\bigcup_{x \in F}  \Br(x,r_x)$. It is enough to put $C := \bigl\{ x \in F : x \notin \nolinebreak \Br(x',r_{x'}) \ \ \forall\, x' \in \nolinebreak F \bigr\}$ and prove that it is at most countable. We introduce a family of open intervals $\mathcal{C} := \bigl\{ \Br(x,r_x) : x\in C \bigr\}$ and note that, since $r_x$ is only one per each $x$, the sets $\mathcal{C}$ and $C$ have the same cardinality. According to definition of $C$ the open intervals in $\C$ must be pair-wise disjoint which implies that $\C$ must be at most countable and so must be the set \nolinebreak $C$.
			
			Since $\bigl\{\Br(x,r_x) : x \in F \bigr\}$ is an open cover of $U$, by Lindel\"{o}f's lemma (cf. \cite{kelley1955}) the set $U$ may be rewritten as $U = \bigcup_{n \in \mathbb{N}}  \Br(x_n,r_{x_n})$, for some sequence $\{x_n\}\subset F$. Ultimately $F = \cIcrr{0,p}(w) \subset C \cup \bigl(\bigcup_{n \in \mathbb{N}}  \Br(x_n,r_{x_n}) \bigr)$. The set $C$ is countable and for each $n$ we have $\muw\bigl( \Br(x_n,r_{x_n}) \cap \cIcrr{0,p}(w)  \bigr) = 0$ due to \eqref{eq:zero_muw_on_ball}, therefore, recalling that $\muw << \Leb^1$,  we arrive at  $\muw\bigl(\cIcrr{0,p}(w)\bigr) = 0$ and the proof is complete.
		\end{proof}
	\end{proposition}
	
	For stable weights Theorem \ref{thm:dirac_delta_approx_crit_point} can be directly utilized for proving existence of functions $\ustep \in \Hnukp{m}$ whose $\bk$ derivative admits jump-type discontinuity at critical points $x_0$ of suitable order:
	
	\begin{corollary}
		\label{cor:step_function}
		For an interval $I =\Iaa \subset \R$ and exponent $p \in \prange$ let $w \in L^1(I)$ be a stable weight. For an order $m \in \mathbb{N}_+$ we consider a weighted Sobolev space $\Hnukp{m}$. We pick $x_0 \in I$ and an order $\bk \in \{0,\ldots,m-1\}$.
		
		For $\Dk = m -\bk$ a claim follows: if $x_0 \in \cIcr{\Dk-1,p}(w)$, then there exists a function $\ustep \in \Hnukp{m}$ such that its $\bk$-th tangential derivative is a step function:
		\begin{equation*}
		\tgradnuk{\bk} \ustep = \mathbbm{1}_{(x_0,\ap)} \qquad \muw\text{-a.e.}
		\end{equation*}
		while its higher tangential derivative $\tgradnuk{k} \ustep$ for $k \in \{\bk+1,\ldots, m\}$ are zero in $\Lpnu$.
		
		Moreover, under the same assumptions, for any other $\tilde{k} \in \{\bk,\ldots,m\}$ a function $\ustep \in \Hnukp{m}$ may be found such that $\tgradnuk{\tilde{k}} \ustep = \mathbbm{1}_{(x_0,\ap)} \ $ $\muw$-a.e.
		\begin{proof}
			In this proof we agree that the classical $k$-th derivative of a function $u$ we will be shortly denoted by $\dr{k}{u}$. Our objective is to point to a sequence $\ustep_h \in \D(\R)$ such that $\ustep_h \rightarrow \ustep$ in $\Hnukp{m}$. This, by definition of $\Hnukp{m}$, requires functions $\vstep_k \in \Lpnu$ for each $k \in \{0,\ldots,m\}$ such that: $\dr{k}{\ustep_h} \rightarrow \vstep_k$ in $\Lpnu$ and, in particular, $\vstep_\bk = \mathbbm{1}_{(x_0,\ap)} \ $ $\muw$-a.e. 
			
			By definition the stable critical point $x_0\in \cIcr{\Dk-1,p}(w)$ is either left-sided stable or right-sided stable. For convenience and without loss of generality we shall assume it is left-sided; let $r>0$ be the radius of the left half-ball appearing in the Definition \ref{def:stab_crit_point}, in addition we enforce that $\Bl(x_0,r) \subset I$. We observe that the point $x_0$ satisfies the assumption of Theorem \ref{thm:dirac_delta_approx_crit_point} with $k$ substituted by $\Dk - 1 \geq 0$. Indeed, it suffices to observe that $w \in L^\infty(\Bl(x_0,r_1))$ for any $r_1 < r$ since $w$ is non-increasing in $x$ on $\Bl(x_0,r)$. Therefore we obtain a sequence of non-negative functions $\hphi_h \in \D\bigl(\Bl(x_0,1/h) \bigr)$ with $\int_I \hphi_h \, dx = 1$ and $\norm{\der{\Dk-1}{\hphi_h}}_\Lpnu \rightarrow 0$; we agree to start the sequence from $h = h_0$ such that $\Bl(x_0,1/h_0) \subset \Bl(x_0,r) \subset I$.
			
			Upon the sequence $\hphi_h$ we build our target sequence $\ustep_h$. The functions $\hphi_h$ approximate Dirac delta at $x_0$ and thus, for the sequence $\dr{\bk}{\ustep_h}$ to converge to $\mathbbm{1}_{(x_0,\ap)}$ we should define $\ustep_h$ so that  $\dr{\bk+1}{\ustep_h} = \hphi_h$ for each $h$, therefore we put
			\begin{equation}
			\label{eq:uh_def}
			\ustep_h (x) := \int_\am^x \hphi_h(y)\ \frac{(x-y)^{\bar{k}}}{\bar{k}!} \ dy
			\end{equation}
			and we will show that indeed $\ustep_h \rightarrow \ustep$ in $\Hnukp{m}$. Although the functions $\ustep_h$ do not have compact supports in $\R$ in general, we may easily remedy this via multiplying $\ustep_h$ by any fixed cut-off function $\varphi \in \D(\R)$ such that $\varphi \equiv 1$ on $I$; we shall omit this aspect as we carry \nolinebreak on.
			
			First we look at the $\bk$-th derivative which for any $x \in I$ equals
			\begin{equation*}
			\dr{\bk}{\ustep_h}(x) = \int_\am^x \hphi_h(y) \, dy.
			\end{equation*}
			Since $\hphi_h \in \D\bigl( \Bl(x_0,1/h) \bigr)$ and $\int_{\Bl(x_0,1/h)} \hphi_h \, dx = 1$ we infer that
			\begin{equation}
			\label{eq:pointwise_convergence_to_step_fun}
			0 \leq \dr{\bk}{\ustep}_h \leq  1,\qquad 
			\dr{\bk}{\ustep}_h(x) =
			\left \{
			\begin{array}{cl}
			0 \quad & \text{if} \quad x \leq x_0 - 1/h, \\
			1  \quad & \text{if} \quad x \geq  x_0,
			\end{array} \right.
			\end{equation}
			whereas the first property follows from the fact that each $\hphi_h$ is non-negative. Therefore the non-negative sequence $\dr{\bk}{\ustep}_h$ is uniformly bounded from above by $1$ and on $I$ it is point-wise convergent to the step function $\mathbbm{1}_{(x_0,\ap)}$. The Lebesgue dominated convergence theorem furnishes $\dr{\bk}{\ustep}_h \rightarrow \mathbbm{1}_{(x_0,\ap)}$ in $\Lpnu$ and also in $L^1(I)$.
			
			Next, assuming that $\bk>0$, we consider any $k \in \{0,\ldots,\bk-1\}$, by formula \eqref{eq:uh_def} we obtain 
			\begin{equation}
			\label{eq:uhk_def}
			\dr{k}{\ustep}_h (x) = \int_\am^x \hphi_h(y)\ \frac{(x-y)^{\bar{k}-k}}{(\bar{k}-k)!} \ dy = \int_\am^x \dr{\bk}{\ustep}_h(y)\ \frac{(x-y)^{\bar{k}-k-1}}{(\bar{k}-k-1)!}  \ dy.
			\end{equation}
			By convergence $\dr{\bk}{\ustep}_h \rightarrow \mathbbm{1}_{(x_0,\ap)}$ in $L^1(I)$ we arrive at a point-wise convergence for every $x\in I$ and every $k \in \{0,\ldots,\bk-1\}$
			\begin{equation}
			\label{eq:pointwise_convergence_to_polynomial}
			\lim_{h \rightarrow \infty}\dr{k}{\ustep_h}(x) = \int_\am^x \mathbbm{1}_{(x_0,\ap)}(y)\ \frac{(x-y)^{\bk-k-1}}{(\bk-k-1)!} \ dy =  \frac{1}{(\bar{k}-k)!}(x-x_0)^{\bar{k}-k} \ \mathbbm{1}_{(x_0,\ap)}(x).
			\end{equation}
			For each $k \in \{0,\ldots,\bk-1\}$ we observe that due to \eqref{eq:pointwise_convergence_to_step_fun} and \eqref{eq:uhk_def} an inequality $0 \leq \dr{k}{\ustep_h}(x) \leq (x-\am)^{\bar{k}-k}/(\bar{k}-k)!$ holds for all $x\in I$. The Lebesgue dominated convergence theorem once again guarantees that the point-wise convergence in \eqref{eq:pointwise_convergence_to_polynomial} implies convergence in $\Lpnu$.
			
			It remains to verify convergence of $\dr{k}{\ustep_h}$ for $k \in \{\bk+1,\ldots,m\}$, we note that in this range of $k$ we obtain
			\begin{equation*}
			\dr{k}{\ustep_h} = \dr{k-\bk-1}{\hphi_h},
			\end{equation*}
			in particular $\dr{\bk+1}{\ustep_h} = \hphi_h$ and $\dr{m}{\ustep_h} = \dr{m-\bk -1}{\hphi_h} = \dr{\Dk-1}{\hphi_h}$. We will show that all  $\dr{k}{\ustep_h}$ for $k \in \{\bk+1,\ldots,m\}$ converge to zero in $\Lpnu$ which will readily verify that $\ustep \in \Hnukp{m}$. The convergence $\dr{m}{\ustep_h} = \dr{\Dk -1}{\hphi_h} \rightarrow 0$ is guaranteed directly by Theorem \ref{thm:dirac_delta_approx_crit_point}. In case when $\Dk >1$ for convergence of lower derivatives $\dr{k}{\ustep_h}$ with $k \in \{\bk+1,\ldots,m-1 \} $ we must show that $\dr{n}{\hphi_h} \rightarrow 0$ in $\Lpnu$ for all $n \in \{0,\ldots, \Dk-2\}$ as well. To this aim
			we shall employ the stability condition that will provide us with a Poincar\'{e}-like inequality. We recall that, due to $x_0$ being a left-sided stable critical point, $w$ is non-increasing (below we work with the non-increasing representative) in the half-ball $\Bl(x_0,r)$. Since for indices $h\geq h_0 \geq 1/r$ the supports of $\hphi_h$ are contained in $\Bl(x_0,1/h) \subset \Bl(x_0,r) \subset I$, we may for any $n \in \mathbb{N}$ utilize the Fundamental Theorem of Calculus combined with H\"{o}lder inequality to write for every $x \in \Bl(x_0,r)$
			\begin{alignat}{1}
			\label{eq:poincare_deriv}
			&\,w(x)\, \abs{\dr{n}{\hphi_h}(x)}^p = w(x) \, \abs{\int_{x_0-r}^x \dr{n+1}{\hphi_h}(y) \,dy }^p \\ 
			\leq & \,w(x) \, \int_{x_0-r}^{x} \abs{\dr{n+1}{\hphi_h}(y)}^p dy \ \ \big\lvert x-(x_0-r) \big\rvert^{p/p'}  \nonumber  \\ \nonumber
			\leq  & \ C \int_{x_0-r}^{x} w(x) \abs{\dr{n+1}{\hphi_h}(y)}^p dy
			\leq \, C \int_{x_0-r}^{x} w(y) \abs{\dr{n+1}{\hphi_h}(y)}^p dy \leq C \, \norm{\dr{n+1}{\hphi_h}}_{\Lpnu}^p,
			\end{alignat}
			which is valid for $p\in (1,\infty)$ with constant $C = r^{p/p'}$, whilst for $p = 1$ it may be rewritten as  $w(x)\, \abs{\dr{n}{\hphi_h}(x)} \leq \norm{\dr{n+1}{\hphi_h}}_{L^1_{\muw}}$. In the inequality next to the last one we have explicitly used monotonicity of $w$, i.e. that for each $x \in \Bl(x_0,r)$ we have $w(y)\geq w(x)$ for every $y \in (x_0-r, x)$. By integrating the inequality above with respect to $x$ in the ball $\Bl(x_0,r)$ and raising to the power of $1/p$ we obtain a desirable Poincar\'{e}-like inequality for any $p \in \prange$ and any $n  \in\mathbb{N}$ (including $n=0$):
			\begin{equation}
			\label{eq:poincare}
			\norm{\dr{n}{\hphi_h}}_{\Lpnu} \leq \ r \ \norm{\dr{n+1}{\hphi_h}}_{\Lpnu} \qquad \forall \, h \geq h_0.
			\end{equation}
			Since $\dr{\Dk-1}{\hphi_h} \rightarrow 0$ in $\Lpnu$ we obtain by induction that
			\begin{equation*}
			\dr{n}{\hphi_h} \rightarrow 0 \quad \text{in } \Lpnu \qquad \text{for every } n \in  \{0,\ldots,\Dk-1\},
			\end{equation*}
			which in turn furnishes $\dr{k}{\ustep_h} \rightarrow 0$ in $\Lpnu$ for $k \in \{\bk+1,\ldots,m\}$.
			We sum up our results for the sequence $\ustep_h$ defined in \eqref{eq:uh_def}:
			\begin{enumerate}[(i)]
				\item (only in case when $\bk>0$) for $k \in \{0,\ldots,\bk-1 \}$
				\begin{alignat*}{1}
				&\dr{k}{\ustep_h} \rightarrow \vstep_k \quad \text{in} \quad \Lpnu \\
				&\text{with} \quad \vstep_k(x):= \frac{1}{(\bar{k}-k)!}(x-x_0)^{\bar{k}-k} \ \mathbbm{1}_{(x_0,\ap)}(x)
				\quad \text{ for } \muw\text{-a.e. } x; 
				\end{alignat*}
				\item for $k = \bk$
				\begin{equation*}
				\dr{\bk}{\ustep_h} \rightarrow \vstep_\bk  \quad \text{in} \quad \Lpnu \quad \text{with} \quad \vstep_\bk:= \mathbbm{1}_{(x_0,\ap)} \quad \muw\text{-a.e.};
				\end{equation*}
				\item for $k \in \{\bk+1,\ldots,m\}$
				\begin{equation*}
				\dr{k}{\ustep_h} \rightarrow \vstep_k  \quad \text{in} \quad \Lpnu \quad \text{with} \quad \vstep_k:= 0 \quad \muw\text{-a.e.}
				\end{equation*}
			\end{enumerate}
			and the proof is complete, its "moreover part" is a direct consequence of monotonicity $\cIcr{\Dk-1,p}(w) =\cIcr{m-\bk-1,p}(w) \subset \cIcr{m-\tilde{k}-1,p}(w)$ whenever $\tilde{k}\geq \bk$, see \eqref{eq:monot_al}.
		\end{proof}
	\end{corollary}
	
	We now apply the results of this subsection to all standard weights considered in this work:
	
	\begin{example}
		For an interval $I = (-1/2,1/2)$, a point $x_0 \in \bar{I}$ and any $p \in \prange$ we consider weights $w_\gamma \in L^\infty(I)$ for $\gamma \in [0,\infty)$ and $w_\mathrm{log}, w_\mathrm{exp} \in L^\infty(I)$ as below:
		\begin{equation*}
		w_\mathrm{log} (x) = \frac{1}{\abs{\log(\abs{x-x_0})}}, \qquad w_\gamma(x) = \abs{x-x_0}^\gamma, \qquad w_\mathrm{exp}(x) = \frac{1}{\exp(1/\abs{x-x_0})}.
		\end{equation*}
		We may write for any $m \geq 1$
		\begin{equation*}
		\mathbbm{1}_{(x_0,1/2)} \in H^{m,p}_{\mu_{w_\gamma}} \qquad \Leftrightarrow \qquad
		\left\{
		\begin{array}{cl}
		\gamma \geq p\, m-1 &\quad \mathrm{if}\ \ p\in (1,\infty), \\
		\gamma >m-1 &\quad \mathrm{if}\ \ p=1.
		\end{array}
		\right.
		\end{equation*}
		The RHS of the above decides that $x_0 \in \cIcr{m-1,p}(w_\gamma)$, see the characterization \eqref{eq:Icra_wa}. The weight $w_\gamma$ is stable and thus the implication $\Leftarrow$ is a direct consequence of Corollary \ref{cor:step_function}. The converse can be inferred from Corollary \ref{cor:higer_order_continuity}: if the RHS is false then $x_0 \notin \cIcr{m-1,p}(w_\gamma)$ which tells us that any $u \in \Hnukp{m}$ must have a continuous representative.
		
		The weight $w_\mathrm{exp}$ is stable as well and \eqref{eq:Icra_wexp} immediately yields through Corollary \ref{cor:step_function} that
		\begin{equation*}
		\mathbbm{1}_{(x_0,1/2)} \in H^{m,p}_{\mu_{w_\mathrm{exp}}} \qquad \text{for every }  m \geq 0 \text{ and } p\in \prange,
		\end{equation*}
		which, if we agree upon such a convention, can be rewritten as
		\begin{equation*}
		\mathbbm{1}_{(x_0,1/2)} \in H^{\infty,p}_{\mu_{w_\mathrm{exp}}} \qquad \text{for every }  p\in \prange.
		\end{equation*}
		Since $x_0$ is the only critical point for the weight $w_\mathrm{exp} \in L^\infty(I)$, the fact that  $\mathbbm{1}_{(x_0,1/2)} \in \Hnukp{\infty}$ may be interpreted as splitting the domain $I$ into $(-1/2,x_0)$ and $(x_0,1/2)$ in terms of theory of weighted Sobolev spaces presented in this work: there hold embeddings $\Lpmu \hookrightarrow \Lloc{(-1/2,x_0)}$ and $\Lpmu \hookrightarrow \Lloc{(x_0,1/2)}$ which allow to equivalently define the weighted Sobolev spaces via weak derivatives separately in the two subdomains.
		
		A somewhat opposite result is obtained for the weight $w_\mathrm{log}$. In Example \ref{ex:wlog} we have showed that $x_0 \in \cIcr{0,p}(w_\mathrm{log})$ if and only if $p = 1$. It is even easier to show that $x_0 \notin \cIcr{m-1,p}(w_\mathrm{log})$ for every $m > 1$ and any $p$, including $p=1$. The weight $w_\mathrm{log}$ is obviously stable, similarly as above we may infer
		\begin{equation*}
		\mathbbm{1}_{(x_0,1/2)} \in H^{m,p}_{\mu_{w_\mathrm{log}}} \qquad \Leftrightarrow \qquad p = 1 \text{ and } m\leq 1. 
		\end{equation*}
	\end{example}
	
	\subsection{Some additional remarks on the links between weight's criticality and jump-type discontinuities of functions in $\Hnukp{m}$. A discussion on optimality of the stability assumption}
	\label{sec:stab_cond_opt}
	
	We agree again that $x_0 \in I=\Iaa$ and we are given a weight $w \in L^1(I)$ that is $\muw$-a.e. non-critical. For a function $u \in \Hnukp{m}$ with any order $m \in \mathbb{N}_+$ we want to look at its $\bk$-th derivative $\tgradnuk{\bk} u$; we denote $\Dk := m - \bk$. The previous subsection was aimed at showing that criticality $x_0 \in \cIcr{\Dk-1,p}(w)$ implies existence of a function $\ustep \in \Hnukp{m}$ with its $\bk$-th tangential derivative $\tgradnuk{\bk} \ustep$ being a step function. We have succeeded, although not in full generality, for we had imposed an additional assumption on the weight --  the stability, which, roughly speaking, enforces degeneracy of the weight to be monotonic around critical points. Currently we shall analyse whether any extra condition was indeed necessary, and if so, whether it could be weakened, for instance it is perhaps possible to only assume $w$ to be of bounded variation. We shall start with a fact that somewhat summarizes the methodology employed so far: the path, leading from criticality $x_0 \in \cIcr{\Dk-1,p}(w)$ to $\bk$-th derivative of function from $\Hnukp{m}$ admiting a jump at $x_0$, passes through mutually dual variation problems. To focus attention we shall consider only the case when $\bk = 0$ resulting in $\Dk = m$. This way we will be checking whether the step function $\ustep = \mathbbm{1}_{(x_0,\ap)}$ itself is an element of the space $\Hnukp{m}$. The proof of Corollary \ref{cor:step_function} shows that this is not restrictive as the derivatives of order below $\bk$ are not the issue.
	\begin{theorem}
		\label{thm:jump_summary}
		For an exponent $p \in \prange$, an interval $I = \Iaa$, a weight $w \in \nolinebreak L^1(I)$ that is $\muw$-a.e. non-critical, any $m \in \mathbb{N}_+$ let us choose a point $x_0\in I$. We additionally assume that there exists $r>0$ such that $w \in L^\infty\bigl( B(x_0,r) \bigr)$; for $0<\eps\leq r$ we set $\Ue := B(x_0,\eps) \cap I$. We recall the energy functional for any $v \in L^p(\Ue)$:
		\begin{equation*}
		\J{\Ue}(v) =\frac{1}{p} \int_{\Ue} w(x) \, \abs{v(x)}^p\, dx,
		\end{equation*}
		while its Legendre-Fenchel conjugate $\Jc{\Ue}(v^*)$ for $v^* \in L^{p'}(\Ue)$ is given in \eqref{eq:Jc_form_p}, \eqref{eq:Jc_form_1}. We write down the following statements:
		\begin{enumerate}[(i)]
			\setlength{\itemsep}{3pt}
			\item  for every $\eps >0$
			\begin{equation*}
			\Pro_{\Sigma,+} := \sup \left\{ \eps \int_{\Ue} \phi \, dx - \sum_{k = 0}^{m-1} \J{\Ue}\bigl( D^k \phi \bigr) \ : \ \phi \in \D(\Ue), \ \phi \geq 0  \right\} = \infty ;
			\end{equation*}
			\item there holds
			\begin{equation*}
			\ustep = \mathbbm{1}_{(x_0,\ap)} \in \Hnukp{m};
			\end{equation*}
			\item  for every $\eps >0$
			\begin{equation*}
			\Pro_{\Sigma} := \sup \left\{ \eps \int_{\Ue} \phi \, dx - \sum_{k = 0}^{m-1} \J{\Ue}\bigl( D^k \phi \bigr) \ : \ \phi \in \D(\Ue)  \right\} = \infty ;
			\end{equation*}
			\item for every $\eps >0$
			\begin{equation*}
			\Pro^*_{m-1} := \inf \left\{ \Jc{\Ue}(v^*) \ : \ v^* \in L^{p'}(\Ue), \ \ D^{m-1}v^* = \eps \, \mathbbm{1}_{(x_0,\ap)} \right\} = \infty;
			\end{equation*}
			
			\item $x_0$ is a critical point of order $m-1$, namely
			\begin{equation*}
			x_0 \in \cIcr{m-1,p}(w);
			\end{equation*}
			\item for every $\eps >0$
			\begin{equation*}
			\Pro^*_{m-1,+} := \inf \left\{ \Jc{\Ue}(v^*) \ : \ v^* \in L^{p'}(\Ue),  \ \ D^{m-1}v^* \geq \eps\,\mathbbm{1}_{(x_0,\ap)} \right\} = \infty.
			\end{equation*}
		\end{enumerate}
		The following relations hold for any weight $w$ satisfying the conditions above:
		\begin{equation*}
		(i) \Rightarrow (ii) \Rightarrow (iii) \quad \Rightarrow \quad (iv) \Leftrightarrow (v) \Leftrightarrow (vi).
		\end{equation*}
		Moreover:
		\begin{enumerate}[(I)]
			\item if either: $m =1$ and $w$ satisfies only the conditions above, or $m>1$ and $w$ is a stable weight, then all the sentences are equivalent for sufficiently small $\eps$, in particular
			\begin{equation*}
			x_0 \in \cIcr{m-1,p}(w) \qquad \Leftrightarrow \qquad  \, \ustep =  \mathbbm{1}_{(x_0,\ap)} \in \Hnukp{m};
			\end{equation*}
			\item  for $m > 1$ the above equivalence does not hold in general for non-stable weights; in particular for every $m > 1$ and $p \in \prange$:
			\begin{equation*}
			\text{there is a non-stable weight } \wcrit \text{ such that } x_0 \in \cIcr{m-1,p}(\wcrit) \text{ and } \mathbbm{1}_{(x_0,\ap)} \notin H^{m,p}_{\mu_\wcrit};
			\end{equation*}
			whereas $\wcrit$ can be chosen from $BV(I)$ whenever $p>1$.
		\end{enumerate}
		
		\begin{proof}
			
			We begin by showing implications that are either straightforward or that have already been an element of some proof carried out earlier:
			
			\vspace{2mm}
			\noindent \underline{Proof of (iv) $\Leftrightarrow$ (v) $\Leftrightarrow$ (vi)}:
			\vspace{1mm}
			
			Assuming that either (iv) or (vi) holds, we obtain that $\Jc{\Ue}(v_\eps^*) = \infty$ for a function $v_\eps^*(x):= \eps\, (x-x_0)^{m-1}/(m-1)!$. Comparing the formulas \eqref{eq:Jc_form_p} or \eqref{eq:Jc_form_1} for $\Jc{\Ue}$ and, respectively, definitions \eqref{eq:def_Icrap} or \eqref{eq:def_Icra1} of the critical set we infer that $x_0 \in \cIcr{m-1,p}$. Further, according to Lemma \ref{lem:positive_distribution} every function $v^*$ satisfying the constraints either in $\Pro^*_{m-1}$ or $\Pro^*_{m-1,+}$ satisfies $\abs{v^*(x)} \geq \abs{v_\eps^*(x)}$ in some ball $B(x_0,\delta)$. Then, if (v) holds, the formula for $\Jc{\Ue}$ gives $\infty$ for each such $v^*$. The equivalences (iv) $\Leftrightarrow$ (v) $\Leftrightarrow$ (vi) are thus established.
			
			\vspace{2mm}
			\noindent \underline{Proof of (i) $\Rightarrow$ (ii)}:
			\vspace{1mm}
			
			Next we depart from (i) and let $\phi_h$ be the maximizing sequence for $\Pro_{\Sigma,+}$. Upon using a diagonal argument along with scaling as in \eqref{eq:scaling} we arrive at a non-negative sequence $\hphi_h \in \D\bigl( B(x_0,1/h) \bigr)$ satisfying $\int_{B(x_0,1/h)} \hphi_h \, dx = 1$ and $\norm{\der{k}{\hphi_h}}_\Lpnu \rightarrow 0$ for all $k \in \{0,\ldots,m-1\}$. Then, the sequence $\ustep_h$ defined by $\ustep_h(x):= \int_{\am}^{x} \hphi_h(y)\, dy$ proves to converge to $\ustep = \mathbbm{1}_{(x_0,\ap)}$ in $\Hnukp{m}$. The argument runs almost identically to the proof of Corollary \ref{cor:step_function}, we do not, however, require the stability condition which was essential therein: unlike here, only the highest derivative $\der{m-1}{\hphi_h}$ was guaranteed to converge to zero in $\Lpnu$. The implication (i) $\Rightarrow$ (ii) is obtained.
			
			\vspace{2mm}
			The next result is new and shall require more effort:
			
			\vspace{2mm}
			\noindent \underline{Proof of (ii) $\Rightarrow$ (iii)}: \nopagebreak
			\vspace{1mm}

			We depart from the fact that $\ustep= \mathbbm{1}_{(x_0,\ap)} \in \Hnukp{m}$. Our goal is to  construct a sequence $\hphi_h$ of smooth functions satisfying
			\begin{equation}
			\label{eq:sequence_phi_sum}
			\hphi_h \in \D\bigl( \Ue \bigr), \quad \int_\Ue \hphi_h \, dx  \rightarrow 1, \quad \norm{\der{k}{\hphi_h}}_\Lpnu \rightarrow 0 \quad \text{for } \ k\in \{0,\ldots,m-1\};
			\end{equation}
			(note that we skip the condition $\hphi_h \geq 0$ as we will not be able to guarantee it and hence we only show implication (ii) $\Rightarrow$ (iii) instead of stronger implication (ii) $\Rightarrow$ (i), see also Remark \ref{rem:pos_phi}).
			We explain how \eqref{eq:sequence_phi_sum} furnishes a maximizing sequence $\phi_h$ for $\Pro_{m-1}$. It is possible to propose a sequence $\gamma_h$ of positive numbers such that $\gamma_h \rightarrow \infty$ and still $\lim_{h\rightarrow \infty }\gamma_h \norm{\der{k}{\hphi_h}}_\Lpnu = \nolinebreak 0$ for any $k$ considered. Performing scaling delivers the target sequence:
			\begin{equation*}
			\phi_h := \gamma_h \,\hphi_h.
			\end{equation*}
			Indeed, we recall that $\Jc{\Ue}(v) = \frac{1}{p}\,\norm{v}_\Lpnu^p$ for every $v \in L^p(\Ue)$, then for any $\eps>0$ we have
			\begin{equation*}
			\eps \int_{\Ue} \phi_h \, dx - \sum_{k = 0}^{m-1} \J{\Ue}\bigl( D^k \phi_h \bigr) = \gamma_h \left(\eps \int_{\Ue} \hphi_h \, dx \right) - \sum_{k = 0}^{m-1} 1/p \left(\gamma_h \norm{D^k \hphi_h}_\Lpnu \right)^p
			\end{equation*}
			which diverges to infinity validating (iii). While seeking the sequence $\hphi_h$ satisfying \eqref{eq:sequence_phi_sum} we distinguish two cases as below; for given $\delta>0$ we agree to denote $\U^\mathrm{ncr}_{\delta,+} := \bigl(\Br(x_0,\delta) \cap I\bigr) \backslash \cIcr{0,p}(w)$ and $\U^\mathrm{ncr}_{\delta,-}$ analogically (note that both defined sets are open):
			\begin{enumerate}[{Case} (1):]
				\item There exists $\delta>0$ such that either $\U^\mathrm{ncr}_{\delta,-} = \varnothing$ or $\U^\mathrm{ncr}_{\delta,+} = \varnothing$,
			\end{enumerate}
			we may for instance assume the case when $\U^\mathrm{ncr}_{\delta,+} = \varnothing$. Then, since $w$ is $\muw$-a.e. non-critical, for every $\eps$ the open intersection $\Ue \cap \Br(x_0,\delta)$ is of zero $\muw$ measure -- we may trivially find a functions $\hphi \in \D\bigl( \Ue \cap \Br(x_0,\delta) \bigr)$ with $\int_\Ue \hphi \, dx =1$ and $\norm{\der{k}{\hphi}}_\Lpnu = 0$ for any natural \nolinebreak $k$.
			\begin{enumerate}[{Case} (2):]
				\item For every $\delta>0$ both $\U^\mathrm{ncr}_{\delta,-}$ and $\U^\mathrm{ncr}_{\delta,+}$ are non-empty,
			\end{enumerate}
			we look at a sequence $\ustep_h$ converging in $\Hnukp{m}$ to $\ustep = \mathbbm{1}_{(x_0,\ap)}$ , due to (ii) such sequence is guaranteed by definition of the Sobolev space itself. We have $\der{k}{\ustep_h} \rightarrow \tgradnuk{k} \ustep$ in $\Lpnu$ for every $k \in \{0,\ldots,m\}$ where $v_k := \tgradnuk{k} \ustep$ is some element from $\Lpnu$, in particular $v_0 = \mathbbm{1}_{(x_0,\ap)}$. Firstly we must show that all the tangential derivatives $v_k =\tgradnuk{k}\ustep$ for $k \in \{1,\ldots,m\}$ must be zero functions. It suffices to prove that $\tgradnuk{} \ustep =0$ and the rest will follow from the operator $\tgradnuk{}$ being closed in $\Hnukp{1}$. We define an open set $\U^\mathrm{ncr}_+ := \bigl((x_0,\infty) \cap I\bigr) \backslash \cIcr{0,p}(w)$. Upon recalling the established embedding $\Lpnu \hookrightarrow \Lloc{\U^\mathrm{ncr}_+}$ we follow the proof of Corollary \ref{cor:higer_order_continuity} to infer that $\Hnukp{1} \hookrightarrow W^{1,1}_{\mathrm{loc}}\bigl(\U^\mathrm{ncr}_+\bigr)$. By the iterative definition of higher order Sobolev space $\Hnukp{m}$ we have $\ustep = \mathbbm{1}_{(x_0,\ap)} \in \Hnukp{1}$ and the derived embedding yields that in $\U^\mathrm{ncr}_+$ the function $\tgradnuk{} \ustep=\tgradnuk{}\left( \mathbbm{1}_{(x_0,\ap)}\right)$ is the distributional derivative of $\mathbbm{1}_{(x_0,\ap)}$, yet the latter is constantly equal to 1 in that set, hence $\tgradnuk{} \ustep = 0 \ $ $\muw$-a.e. in $\U^\mathrm{ncr}_+$. We proceed analogically for the set $\U^\mathrm{ncr}_- := \bigl((-\infty,x_0) \cap I\bigr) \backslash \cIcr{0,p}(w)$ where $\ustep \equiv 0$. Since the weight $w$ is $\muw$-a.e. non-critical the sum $\U^\mathrm{ncr}_- \cup \U^\mathrm{ncr}_+$ is of full measure $\muw$ which eventually gives $\tgradnuk{} \ustep = 0$ in $\Lpnu$.
			
			Next, for a given $\eps >0$ we choose and fix a smooth cut-off function $\varphi_\eps$ such that
			\begin{equation}
			\label{eq:varphi_eps}
			\varphi_\eps \in \D(\Ue), \quad 0\leq \varphi_\eps \leq 1 \quad \mathrm{spt} \left( \varphi'_\eps \right) \subset \U^\mathrm{ncr}_{\eps,-} \cup \U^\mathrm{ncr}_{\eps,+}, \quad \varphi(x_0) = 1, 
			\end{equation}
			which roughly means that $\varphi_\eps$ increases in a compact subset of $ \U^\mathrm{ncr}_{\eps,-}$ to equal 1 around $x_0$ and then it decreases to zero in a compact subset of $ \U^\mathrm{ncr}_{\eps,-}$. We define a sequence
			\begin{equation*}
			\hphi_h := \varphi_\eps \, \der{}{\ustep_h} \ \in \D(\Ue)
			\end{equation*}
			and prove that it satisfies \eqref{eq:sequence_phi_sum} which will ultimately give the implication (ii) $\Rightarrow$ (iii). We have obtained above that $\der{k}{\ustep_h} \rightarrow 0$ in $\Lpnu$ for $k \in \{1,\ldots,m\}$, therefore, since $\varphi$ is fixed, the Leibniz differentiation formula furnishes $\norm{\der{k}{\hphi_h}}_\Lpnu \rightarrow 0$ for each $k \in \nolinebreak \{0,\ldots,m-1\}$. It is thus enough to check that $\int_\Ue \hphi_h \, dx \rightarrow 1$; through integration by parts we obtain
			\begin{equation*}
			\int_\Ue \hphi_h \, dx = - \int_{\Ue} \varphi'_\eps \ \ustep_h \, dx 
			= - \int_{\Ue}\varphi'_\eps \, \left( \ustep_h - \mathbbm{1}_{(x_0,\ap)} \right)  dx - \int_{\Br(x_0,\eps) \cap I} \varphi'_\eps \, dx,
			\end{equation*}
			where the last term may be rewritten as $-\int_{\U^\mathrm{ncr}_{\eps,+}} \varphi'_\eps \, dx =1$. We must show that the other term converges to zero; we use a trick that is standard for this work:
			\begin{equation*}
			\abs{\int_{\Ue} \varphi'_\eps \, \left( \ustep_h - \mathbbm{1}_{(x_0,\ap)} \right)  dx} \leq  \biggl(\int_{\Ue} w \, \abs{ \ustep_h - \mathbbm{1}_{(x_0,\ap)} }^p  dx \biggr)^\frac{1}{p} \biggl( \int_{\Ue} \frac{\abs{\varphi'_\eps}^{p'}}{w^{p'/p}} dx \biggr)^\frac{1}{p'},
			\end{equation*}
			where the second factor (that for $p=1$ should be read as $\norm{\varphi'_\eps/w}_{L^\infty(\Ue)}$) is finite since $\mathrm{spt}\left(\varphi'_\eps \right)$ is a compact subset of an open set $\U^\mathrm{ncr}_{\eps,-} \cup \nolinebreak \U^\mathrm{ncr}_{\eps,+}$ which is disjoint with $\cIcr{0,p}(w)$. The first factor converges to zero due to $\ustep_h\rightarrow  \mathbbm{1}_{(x_0,\ap)}$ in $\Lpnu$. The implication (ii) $\Rightarrow$ (iii) is now proved.
			
			\vspace{2mm}
			We have showed that the question whether the step function $\ustep = \mathbbm{1}_{(x_0,\ap)}$ is an element of $\Hnukp{m}$ revolves around two variational problems $\Pro_\Sigma$ and $\Pro_{\Sigma,+}$. Similarly, criticality $x_0 \in \cIcr{m-1,p}(w)$ was proved to be equivalent to the problem $\Pro^*_{m-1}$ or $\Pro^*_{m-1,+}$. At this point we link the two groups of statements (i),(ii),(iii) and (iv),(v),(vi) by recognizing duality between the respective variational problems:
			
			\vspace{2mm}
			\noindent \underline{Proof of (iii) $\Rightarrow$ (iv)}:
			\vspace{1mm}
			
			It has already been established in the proof of Theorem \ref{thm:dirac_delta_approx_crit_point} that the problem $\Pro^*_{m-1,+}$ is dual to:
			\begin{equation*}
			\Pro_{m-1,+} = \sup \left\{ \eps \int_{\Ue} \phi \, dx - \J{\Ue}\bigl( D^{m-1} \phi \bigr)  \ : \ \phi \in \D(\Ue), \ \phi \geq 0  \right\} = \Pro^*_{m-1,+}
			\end{equation*}
			and no duality gap occurs provided $\eps\leq r$. The primal to $\Pro^*_{m-1}$ is identical up to enforcing non-negativity of smooth functions:
			\begin{equation*}
			\Pro_{m-1} = \sup \left\{ \eps \int_{\Ue} \phi \, dx - \J{\Ue}\bigl( D^{m-1} \phi \bigr)  \ : \ \phi \in \D(\Ue) \right\} = \Pro^*_{m-1}.
			\end{equation*}
			It is straightforward that $\Pro_{\Sigma} = \infty$ implies $\Pro_{m-1} = \infty$ which by the above gives $\Pro^*_{m-1} = \nolinebreak \infty$. We have thus arrived at (iii) $\Rightarrow$ (iv).
			
			\vspace{2mm}
			\noindent \underline{Proof of claim (I) through verifying (vi) $\Rightarrow$ (i)}:
			\vspace{1mm}
			
			The implication (vi) $\Rightarrow$ (i) is essentially the main part of the proof of Corollary \ref{cor:step_function}, nevertheless we give a short argument to show how the stability of the weight enters here to render the statements (i)-(vi) equivalent.
			
			The point (vi) states that $\Pro^*_{m-1,+} = \infty$, therefore the duality above yields $\Pro_{m-1,+} = \nolinebreak \infty$ and we may pick the maximizing sequence $\phi^\eps_h$ for a given $\eps>0$. If $m = 1$ then naturally problems $\Pro_{m-1,+}$ and $\Pro_{\Sigma,+}$ coincide and (vi) implies (i). For $m>1$ we assume that $w$ is a stable weight and, due to $x_0 \in \cIcr{m-1,p}(w)$ by the equivalence (v) $\Leftrightarrow$ (vi), we have the monotonicity of the weight on one of the sides of $x_0$ where the weight degenerates. We may thus assume that each of the functions $\phi^\eps_h$ is supported on this very side. Then the Poincar\'{e}-like inequality \eqref{eq:poincare} holds for sufficiently small $\eps$, which eventually gives a finite constant $C>0$ such that $\sum_{k = 0}^{m-1} \J{\Ue}\bigl( D^k \phi^\eps_h \bigr) \leq C \cdot  \J{\Ue}\bigl( D^{m-1} \phi^\eps_h \bigr)$ for every $h$. Let us now consider the problem  $\Pro_{\Sigma,+}$ for $\eps = \bar{\eps}$. One may check that for $\tilde{\eps} \nolinebreak:= \min \{\bar{\eps},\bar{\eps}/C \}$ the functions $\phi^{\tilde{\eps}}_h$ serve as a maximizing sequence for $\Pro_{\Sigma,+}$ rendering it infinite, hence (i) is achieved. 
			
			\vspace{2mm}
			\noindent \underline{Proof of claim (II) through contradicting implication (v) $\Rightarrow$ (iii)}:
			\vspace{1mm}
			
			To prove the claim (II), which ultimately justifies our assumption of stability of the weight, it suffices to come up with a counter-example in the form of a weight $\wcrit$ for which (v) holds whilst (iii) does not. For this purpose we shall require the problem dual to $\Pro_{\Sigma}$; we will also dualize $\Pro_{\Sigma,+}$ in order to have a complete view on duality for the four variational problems appearing in statements (i)-(vi). We only sketch the derivations as the details are fully analogical to duality argument in the proof of Theorem \ref{thm:dirac_delta_approx_crit_point}. For the dual of  $\Pro_{\Sigma,+}$ all it takes is to redefine spaces $Y$, $Y^*$, the functional $G$ and the operator $\Lambda$, the spaces $X= \D(\Ue)$ and $X^* = \D'(\Ue)$ remains unchanged, together with the functional $F: X \rightarrow \Rb$ defined in \eqref{eq:fun_F}. We propose $Y$ to be a Cartesian product $Y := \bigl(L^p(\Ue) \bigr)^{m}$ and so $Y^* = \bigl(L^{p'}(\Ue) \bigr)^{m}$; for the pairing we naturally choose $\pairing{(v_0,\ldots,v_{m-1})\, , \,(v^*_0,\ldots,v^*_{m-1})}_{\pairing{Y,Y^*}} := \sum_{k = 0}^{m-1} \int_{\Ue} v_k\,v^*_k \, dx$. We define $G\bigl( (v_0,\ldots,v_{m-1}) \bigr) := \sum_{k = 0}^{m-1} \J{\Ue}(v_k)$ and the continuous linear operator $\Lambda:X \rightarrow Y$ is chosen such that  $\pi_k\bigl(\Lambda\, \phi\bigr) = (-1)^k \der{k}{\phi}$ for $\phi \in X$ and $k \in \{0,\ldots,m-1\}$; $\pi_k$ denotes the projection on $k$-th coordinate and the derivatives are understood in the classical sense. We arrive at the adjoint operator $\Lambda^*:Y^* \rightarrow X^*$ and the Legendre-Fenchel conjugate $G^*:Y^* \rightarrow \Rb$ as follows: $\Lambda^*(v^*_0,\ldots,v^*_{m-1}) = \sum_{k = 0}^{m-1} D^k v^*_k$, where here each derivative is distributional, and $G^*\bigl( (v^*_0,\ldots,v^*_{m-1}) \bigr) := \sum_{k = 0}^{m-1} \Jc{\Ue}(v^*_k)$. The algorithm given in Chapter III of \cite{ekeland1999} furnishes the dual problem, again the stability of duality is guaranteed by the fact that $w \in L^\infty(\Ue)$ (being true due to the assumption $\eps\leq r$):
			\begin{equation*}
			\Pro^*_{\Sigma,+} = \inf \left\{ \sum_{k = 0}^{m-1} \Jc{\Ue}(v^*_k) \ : \ v^*_k \in L^{p'}(\Ue),  \ \sum_{k = 0}^{m-1} D^k v^*_k \geq \eps\,\mathbbm{1}_{(x_0,\ap)} \right\} = 	\Pro_{\Sigma,+}.
			\end{equation*}
			In the case of the problem $\Pro_\Sigma$ we must drop the term $\mathbbm{I}_{X_0}$ in the definition of $F$, see \eqref{eq:fun_F}, then we arrive at
			\begin{equation*}
			\Pro^*_{\Sigma} = \inf \left\{ \sum_{k = 0}^{m-1} \Jc{\Ue}(v^*_k) \ : \ v^*_k \in L^{p'}(\Ue),  \ \sum_{k = 0}^{m-1} D^k v^*_k = \eps\,\mathbbm{1}_{(x_0,\ap)} \right\} = 	\Pro_{\Sigma}.
			\end{equation*}
			We can now see that our objective of validating the claim (II) can be reduced to finding a weight $\wcrit \in BV(I)$ for which $x_0 \in \cIcr{m-1,p}(\wcrit)$ whereas $\Pro^*_\Sigma <\infty$, then we will also have $\Pro_\Sigma <\infty$ contradicting (iii) and thus (ii). To focus attention, our strategy will be to first give an example of such a non-stable weight $\wcrit$ for $m=2$ and $p>1$, afterwards we shall explain how to adapt the example for other cases. We also start with a weight that is not of bounded variation, which we shall fix towards the end.
			
			Assuming that $p>1$ for an interval $I = (-1,1)$ we propose a weight $\wcrit \in L^\infty(I)$
			\begin{equation*}
			\wcrit(x) :=   \left\{
			\begin{array}{cl}
			x^\gamma &\ \ \mathrm{if}\ \ x\in \Ione, \\
			1 &\ \ \mathrm{if}\ \ x \in I\backslash \Ione,
			\end{array}
			\right. \qquad \Ione:= \bigcup_{n=0}^\infty \big(x_-(n)\, ,\, x_+(n) \big], \qquad \gamma := p\left(1 +\frac{2}{p'} \right),
			\end{equation*}
			where for any natural number $n$ we set $x_+(n):= \frac{1}{2^n}$, $\ x_-(n) := \frac{1}{2^n}-\frac{1}{2^{2(n+1)}}$ so that
			\begin{equation}
			\label{eq:theta}
			\frac{\abs{x_+(n)-x_-(n)}}{\abs{x_+(n)-x_+(n+1)}} = \frac{1}{2^{n+1}}.
			\end{equation}
			The weight $\wcrit$ is plotted in Fig. \ref{fig:wcrit}(a) for the case $p=2$; we observe that "contribution of the part $x^\gamma$" decreases when approaching $x_0:=0$, which is reflected in \eqref{eq:theta} above.
			\begin{figure}[h]
				\centering
				\subfloat[]{\includegraphics*[width=0.45\textwidth]{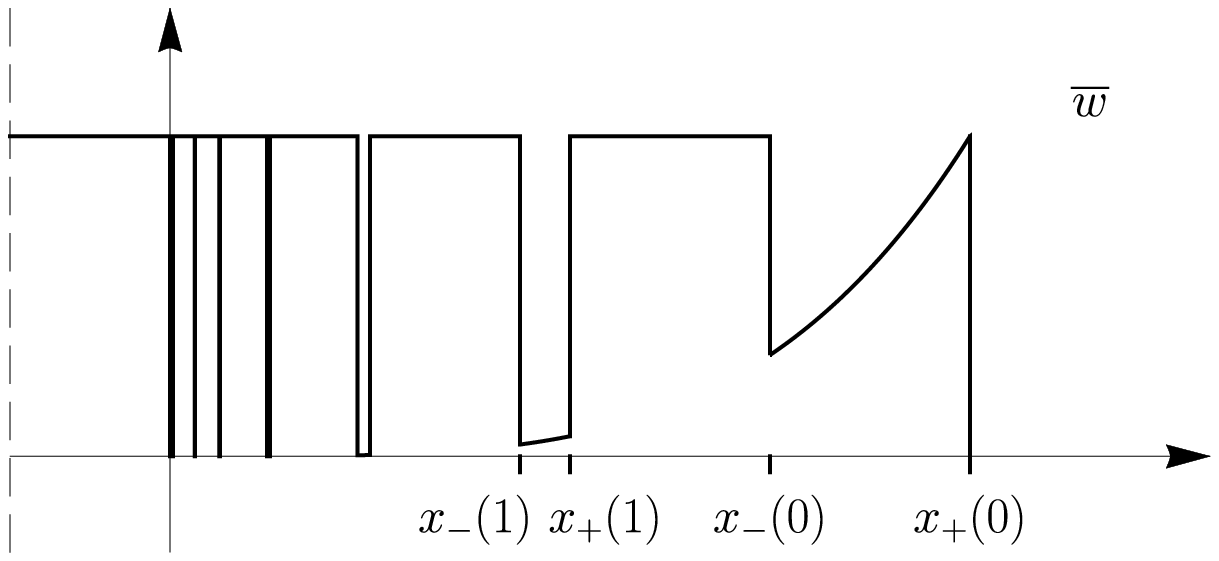}}\hspace{1cm}
				\subfloat[]{\includegraphics*[width=0.45\textwidth]{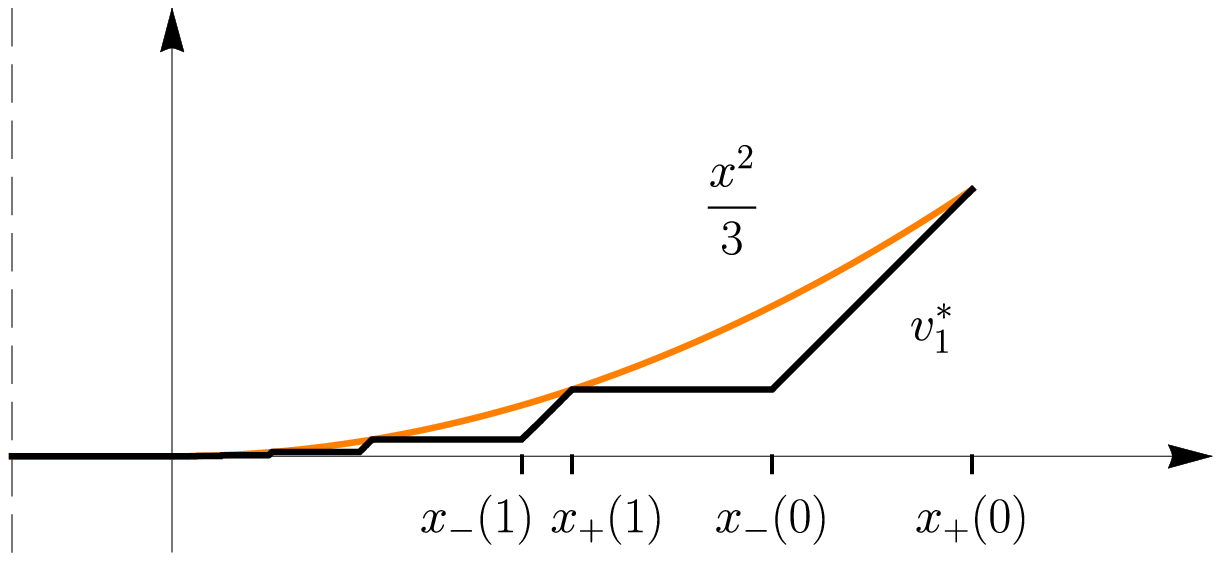}}
				\caption{(a) Example of a non-stable weight $\wcrit$ for which $\mathbbm{1}_{(x_0,\ap)} \notin H^{2,2}_{\mu_\wcrit}$ despite $x_0 \in \cIcr{1,2}(\wcrit)$; (b) a function $v^*_1 \in L^2(I)$ satisfying $\Jc{I}(v^*_1) < \infty$ and its upper bound.}
				\label{fig:wcrit}       
			\end{figure}
			
			Firstly we verify that $x_0 =0 \in \cIcr{1,p}(\wcrit)$: for $k \in \mathbb{N}$ we put $\eps_k := 1/2^k$ and compute that
			\begin{alignat}{1}
			\nonumber
			\int\limits_{I \cap B(x_0,\eps_k)} \! \biggl(\frac{\abs{x-x_0}}{(\wcrit(x))^{1/p}}\biggr)^{p'} \! dx \geq \! \int\limits_{\Ione \cap B(x_0,\eps_k)} \! \frac{x^{p'}}{x^{p'+2}} \ dx &= \sum_{n=k}^\infty -\frac{1}{x}\,\bigg\rvert_{x_-(n)}^{x_+(n)}\\
			\label{eq:est_crit}
			&= \sum_{n=k}^\infty \frac{2^n}{2^{n+2}-1} = \infty
			\end{alignat}
			for any $k \in \mathbb{N}$, therefore, according to definition \eqref{eq:def_Icrap} of critical set, $x_0\in\cIcr{1,p}$ indeed holds. It can be checked that this would be no longer true if $\gamma$ was replaced with $\tilde{\gamma} =p(1+1/p') = 2 p - 1$, whereas $x_0 \in \cIcr{1,p}(w_{\tilde{\gamma}})$ for $w_{\tilde{\gamma}}(x) = \abs{x}^{\tilde{\gamma}}$ in the whole $I$, see Example \nolinebreak \ref{ex:wa} and characterization \eqref{eq:Icra_wa}. Our conclusion is that we had to increase degeneracy in $\Ione$ in order to make up for fast shrinking of this set in proximity of $x_0$.
			
			Our goal is to show that $\Pro^*_\Sigma < \infty$ for $\wcrit$, we recall that $m = 2$. It is enough to point to functions $v^*_0,v^*_1 \in L^p(\Ue)$ satisfying $D v^*_1 +v^*_0 = \mathbbm{1}_\Ue$ for which $\Jc{\Ue}(v^*_0) + \Jc{\Ue}(v^*_1) \linebreak[1]< \infty$ and we may do so only for $\eps = 1$ that gives $\Ue = B(x_0,1) = I$. We note that $\Pro_{m-1}$ is surely infinite by the equivalence (iv) $\Leftrightarrow$ (v) which tells us we will not succeed by choosing either $v^*_0=0$ or $v^*_1=0$. Surely we need the two functions to smartly complement each other around the point $x_0 =0$. We propose for a.e. $x \in I$
			\begin{equation}
			v^*_0(x) := \mathbbm{1}_{I \backslash \Ione}(x), \qquad v^*_1(x) := \int_{-1}^x \mathbbm{1}_{\Ione}(y) \, dy;
			\end{equation}
			the function $v^*_1$ is illustrated in Fig. \ref{fig:wcrit}(b). It is straightforward to check that indeed  $D v^*_1 + \nolinebreak v^*_0 = \nolinebreak \mathbbm{1}_I$; in addition we easily compute
			\begin{equation*}
			\Jc{I}(v^*_0) = \frac{1}{2} \int_I \biggl(\frac{\abs{v^*_0}}{\wcrit^{1/p}} \biggr)^{p'} dx = \frac{1}{2} \int_{I\backslash \Ione} \frac{1}{1} \ dx < \infty.
			\end{equation*}
			It is thus left to show that $\Jc{I}(v^*_1) < \infty$. The continuous function $v^*_1$ satisfies a crucial estimate that has essentially predetermined the form of set $\Ione$: 
			\begin{equation}
			\label{eq:v_leq_g}
			\abs{v^*_1(x)} \leq g(x):= \frac{x^2}{3} \qquad  \text{for a.e. } x \in I.
			\end{equation}
			Indeed, we check that for every $k\in \mathbb{N}$
			\begin{equation}
			\label{eq:v_leq_g_arg}
			v^*_1\bigl(x_+(k)\bigr) = \int_{0}^{x_+(k)} \mathbbm{1}_{\Ione}(y) \, dy = \sum_{n=k}^{\infty} \abs{x_+(n)-x_-(n)} = \frac{1}{4} \sum_{n=k}^{\infty} \frac{1}{4^n} = \frac{1}{3} \left( \frac{1}{2^k} \right)^2
			\end{equation}
			hence $v^*_1\bigl(x_+(k)\bigr) = g\bigl(x_+(k)\bigr)$ for any natural $k$. Moreover $D g(x) < 1$ for any $x \in I$, thus $v^*_1(x) \leq g(x)$ for $x \in \bigl( x_-(k), x_+(k) \big)$ where $D v^*_1$ is a.e. equal to $1$. For the interval $\bigl( x_+(k+1), x_-(k) \big)$, on the other hand, $v^*_1$ is constantly equal to $g\bigl(x_+(k+1)\bigr)$. This holds for every $k$ and the inequality  $\abs{v^*_1(x)} \leq x^2/3$ is therefore obtained, see also the visual comparison in Fig. \ref{fig:wcrit}(b). Since $p'>1$ for every $p\in \prange$, we may readily check that
			\begin{alignat}{1}
			\label{eq:est_Jc1}
			2 \, \Jc{I}(v^*_1) &= \int_I \biggl(\frac{\abs{v^*_1}}{\wcrit^{1/p}} \biggr)^{p'} dx  \leq \int_I \frac{\abs{x^2/3}^{p'}}{(\wcrit(x))^{p'/p}} \, dx \\ & = \int_{I\backslash \Ione}  \frac{\abs{x^2/3}^{p'}}{1} \, dx + \int_\Ione \frac{\abs{x^2/3}^{p'}}{x^{p'+2}} dx
			=C +\frac{1}{3^{p'}} \int_\Ione x^{p'-2} dx < \infty, \nonumber
			\end{alignat}
			where $C$ is finite. This ultimately proves that $\Pro^*_{\Sigma} < \infty$ and thus also $\Pro_{\Sigma} < \infty$ which rules out the possibility of the step function $\mathbbm{1}_{(x_0,\ap)}$ being an element of $H^{2,p}_{\mu_\wcrit}$ in case when $p>1$. For $p=1$, after careful look at definition \eqref{eq:def_Icra1} for $\cIcr{1,1}(\wcrit)$ and formula \eqref{eq:Jc_form_1} for $\Jc{I}$, we may easily deduce that an analogical argument would check out if we put any $\gamma$ from $(1,2]$ in the definition of $\wcrit$.
			
			We shall outline how the idea above can be modified for orders $m > 2$. The set $\Ione$ remains unaltered, while in the definition of $\wcrit$ we put $\gamma := p\left(m-1 +2/{p'} \right)$. We find that criticality $x_0 \in \cIcr{m-1,p}(\wcrit)$ may be proved by applying precisely the same estimate as in \eqref{eq:est_crit}. We define $v^*_0$ identically as above and for every $k \in \{1,\ldots,m-2\}$ we put $v^*_k=0$, this trivially yields $\Jc{I}(v^*_k)$ for each $k<m-1$. The trick lies in defining
			\begin{equation*}
			v^*_{m-1}(x) := \int_{-1}^x \mathbbm{1}_{\Ione}(y)\frac{(x-y)^{m-2}}{(m-2)!} \, dy
			\end{equation*}
			that, similarly to \eqref{eq:v_leq_g}, furnishes $\abs{v^*_{m-1}(x)} \leq C \abs{x}^m$ for a constant $C \in (0,1)$; the proof of this fact demands slightly more work than in \eqref{eq:v_leq_g_arg} and we decide not to display it herein. An estimate analogical to $\eqref{eq:est_Jc1}$ gives $\Jc{I}(v^*_{m-1}) <\infty$ as well, which eventually disqualifies the function $\mathbbm{1}_{(x_0,1)}$ as an element of $\Hnukp{m}$.
			
			The weight $\wcrit$ that was put forward is clearly not of bounded variation and we now show this can be easily remedied for $p>1$. We redefine the weight function on $I \backslash \Ione$:
			\begin{alignat*}{1}
			\wcrit_{BV}(x) := \left\{
			\begin{array}{cl}
			x^\gamma &\ \ \text{if}\ \ x\in \Ione, \\
			1/\beta^n &\ \ \text{whenever}\ \ x \in \bigl[ x_+(n+1)\, ,\, x_-(n) \bigr) \text{ for some } n\in \mathbb{N}, 
			\end{array}
			\right.
			\end{alignat*}
			where $\gamma$ is defined as before depending on $m$. It is clear that $\wcrit_{BV}$ is in $BV(I)$ for any $\beta>1$. It is straightforward that still $x_0 \in \cIcr{1,p}\bigl( \wcrit_{BV} \bigr)$, since $\wcrit_{BV} \leq \wcrit$. With the functions $v^*_k$ kept as before it is also clear that $\Jc{I}(v^*_k)$ is finite for any $k \in \{1,\ldots,m-1\}$, since the weight did not change on $\Ione$ and all such functions $v^*_k$ are non-negative only there. It remains to compute
			\begin{equation*}
			2 \Jc{I}(v^*_0) =\int_I \biggl(\frac{\abs{v^*_0}}{\wcrit_{BV}^{\,1/p}} \biggr)^{p'} \, dx = \sum_{n=0}^{\infty} \bigl(\beta^n\bigr)^{p'/p} \bigl\lvert x_-(n)-x_+(n+1) \bigr\rvert \leq \sum_{n=0}^{\infty} \biggl( \frac{\beta^{1/(p-1)}}{2} \biggr)^n
			\end{equation*}
			which is finite whenever $\beta^{1/(p-1)} <2$. Recalling that $p>1$, we have arrived at a weight $\wcrit_{BV}$ that for any $\beta \in (1,2^{p-1})$ gives the claim (II) for each $m \geq 2$. The proof concludes now.
		\end{proof}
	\end{theorem}
	
	\begin{remark}
		In the last part of the proof we have constructed a weight $\wcrit$ that furnishes the claim (II), the weight for the case $m=2$ and $p=2$ is displayed in Fig. \ref{fig:wcrit}(b). Our point of departure was the family of weights $w$ satisfying $x_0 \in \cIcr{m-1,p}(w)$ and from it we had to choose a particular weight such that  $\mathbbm{1}_{(x_0,\ap)}$ is not an element of $\Hnukp{m}$. We have delivered this by proposing a weight $\wcrit$ that gives an abstract result $\Pro^*_{\Sigma} < \infty$. In order to provide an intuition behind the choice of $\wcrit$ we recall that finiteness of $\Pro^*_{\Sigma}$ translates by means of duality to $\Pro_{\Sigma} < \infty$. Regardless of $\Pro_{\Sigma}$ and $\Pro^*_{\Sigma}$, due to criticality of $x_0$ we still have $\Pro_{m-1} = \Pro^*_{m-1} =\infty$. In both variation problems $\Pro_{m-1}$ and $\Pro_{\Sigma}$ we seek a sequence $\phi_h \in \D(\Ue)$ with arbitrarily big integral $\int_\Ue \phi_h \,dx$ and arbitrarily small $L^p_{\mu_\wcrit}$-norms: in case of $\Pro_{m-1}$ we only bother with the norm of the highest derivative $\der{m-1}{\phi_h}$, while in $\Pro_{\Sigma}$ we wish to additionally control all the smaller derivatives, including the function $\phi_h$ itself. From now on we concentrate on the case $m=2$ and the explanation can be readily given: we may find a sequence $\phi_h$ with small norm $\norm{\der{}{\phi_h}}_{L^p_{\mu_\wcrit}}$ and we roughly do that by enforcing $\mathrm{supp}\bigl( \der{}{\phi_h} \bigr) \subset \bigl(x_-(h+1), x_+(h+1) \bigr) \ \cup \ \bigl(x_-(h), x_+(h) \bigr)$, one may see Fig. \nolinebreak \ref{fig:Poincare_argument_1} to feel the idea -- this argument explains why $\Pro_{m-1} = \infty$. This way, however, we do not control $\norm{\phi_h}_{L^p_{\mu_\wcrit}}$, since $\phi_h$ may be big on $I\backslash \Ione$ where $\wcrit \equiv 1$. The example with $\wcrit_{BV}$ proves that geometrical decay of the weight in $I \backslash \Ione$ around $x_0$ does not help. One may try to squeeze each function $\phi_h$ into one interval $\bigl(x_-(h), x_+(h) \bigr)$. This would give control over $\norm{\phi_h}_{L^p_{\mu_\wcrit}}$, but the estimate on the dual energy $\Jc{I}$ in \eqref{eq:est_Jc1} states that: (1) those intervals are narrowing too fast with $h$ to control the derivative of $\phi_h$ (see \eqref{eq:theta}); (2) at the same time $x^\gamma$ in $\Ione$ is big enough to capture this by making the norm $\norm{ \der{}{\phi_h}}_{L^p_{\mu_\wcrit}}$ blow up. 
	\end{remark}
	
	\begin{remark}
		\label{rem:pos_phi}
		In this subsection we have encountered pairs of variational problems: $\Pro_{m-1}$, $\Pro_{m-1,+}$ and $\Pro_{\Sigma}$, $\Pro_{\Sigma,+}$ and two pairs of their respective duals. The pairs of primal problems differ only by the constraint on non-negativity of smooth functions $\phi$, which we have needed for proving Corollary \ref{cor:step_function} (and thus for the implication (i) $\Rightarrow$ (ii) above), precisely for the functions $\ustep_h$ to vary between 0 and 1, see \eqref{eq:pointwise_convergence_to_step_fun}. All the forms of variational problems conveniently played a different role in our last proof, although $\Pro_{m-1} = \Pro_{m-1,+}$ are equivalent due to (iv) $\Leftrightarrow$ (vi). The proof that $\Pro_{\Sigma} = \Pro_{\Sigma,+}$ as well (or, equivalently, $\Pro^*_{\Sigma} = \Pro^*_{\Sigma,+}$) seems not so straightforward. Although we have not come up with a counter-example, the idea of the proposed weight $\wcrit$ certainly gives room to try. Since this matter was not crucial for this work we have left it open.
	\end{remark}
	
	\newpage
	
	\section{The trace operators in the weighted Sobolev space}
	
	\label{sec:trace_extension}
	
	In this section we look at the behaviour of functions $u \in \Hnukp{m}$ at the boundary points $\am$ and $\ap$; we agree that for an interval $I = \Iaa$ and an exponent $p \in \prange$ we are given a weight $w\in L^1(I)$ that is stable in accordance to Definition \ref{def:stab_crit_point}. For an arbitrary order $m \geq 1$ we wish to verify whether the $\bk$-th tangential derivative $\tgradnuk{\bk} u$ is well defined at $\am$ and independently at $\ap$, we consider any $\bk$ in $\{0,\ldots,m-1\}$.
	
	For an interior point $x_0 \in I$ the local version of Theorem 4.1, or rather Corollary \ref{cor:higer_order_continuity}, states that all up to order $\bk$ derivatives of $u \in \Hnukp{m}$ (along with the function itself) are continuous at $x_0$ in the sense of the precise a.e. equal representative. Those results, however, were adjusted for $x_0$ being also the one of the boundary points $\am, \ap$ and in this particular scenario we are essentially enabled to attribute a value to the function $u \in \Hnukp{m}$ and its respective derivatives at the boundary point, for example at $x_0 = \am$; the condition is that this point cannot be critical with a suitable order.
	
	Below we show this result rigorously and in order to neatly utilize the embedding $\Hnukp{m} \hookrightarrow W^{\bk+1,1}(\U)$ from Corollary \ref{cor:higer_order_continuity} it is convenient to introduce trace operators on the space of smooth functions that may or may not be continuously extended to $\Hnukp{m}$. For a function $u \in \D(\R)$ we define, at each of the end-points $\am$ and $\ap$ independently, trace operators $\bigl( \Tr{k}{\,\cdot\,} \bigr) (\am), \ \bigl( \Tr{k}{\,\cdot\,} \bigr) (\ap): \D(\R) \rightarrow \R$ of order $k$:
	
	\begin{equation*}
	\vspace{2mm}
	\Tr{k}{u} (\am) := D^k \!u \,(\am), \qquad \Tr{k}{u} (\ap) := D^k\! u\, (\ap).
	\end{equation*}
	
	The next result is a direct consequence of Corollary \ref{cor:higer_order_continuity} and gives a sufficient condition for the trace $\Tr{k}{u} (a_{\, -\diagup+})$ to be well defined for functions $u$ in the weighted Sobolev space \nolinebreak $\Hnukp{m}$:
	\begin{proposition}
		\label{prop:trace_extension}
		For an interval $I = \Iaa \subset \R$, exponent $p \in \prange$ and weight $w \in L^1(I)$ we consider a weighted Sobolev space $\Hnukp{m}$, where $m \in \mathbb{N}_+$. Let us choose $\bk \in \nolinebreak \{0,\ldots,m-\nolinebreak1\}$.
		
		If $\am \! \notin \cIcr{\Dk-1,p}(w)$ with $\Dk= m-\bk$, then for each $k \in \{0,\ldots,\bk\}$ the trace operator
		\begin{equation*}
		\bigl(\Tr{k}{\,\cdot\,} \bigr) (\am): \D(\R) \rightarrow \R
		\end{equation*}
		has a continuous extension to $\Hnukp{m}$. The same independently applies to the operator $\bigl(\Tr{k}{\,\cdot\,} \bigr) (\ap)$ provided $\ap \notin \cIcr{\Dk-1,p}(w)$.
		\begin{proof}
			We assume that $\am \notin \cIcr{\Dk-1,p}(w)$. Then, by putting $x_0 := \am$, we obtain an \linebreak$\bar{I}$\nolinebreak-open set $\V$ in accordance with Corollary \ref{cor:higer_order_continuity}. There must exist $\delta>0$ such that the open interval $I_1:=(\am,\am\!+\delta)$ is contained in $\V \cap I$. Then by \eqref{eq:embedding_higher} we obtain an embedding $\Hnukp{m} \hookrightarrow W^{\bk+1,1}(I_1)$. It is well-established that there exists a continuous extension of $\bigl(\Tr{k}{\,\cdot\,} \bigr) (\am)$ to $W^{\bk+1,1}(I_1)$ for each $k \leq \bk$ and so there is one to $\Hnukp{m}$. The proof for $\ap$ is analogical.
		\end{proof}
	\end{proposition}
	
	The condition $\am \! \notin \cIcr{\Dk-1,p}(w)$ appears above as sufficient to extend the trace operator $\bigl(\Tr{\bk}{\,\cdot\,} \bigr) (\am)$ of order $\bk$. We will find that for stable weights it is also the necessary one. The intuition may be readily found in Corollary \ref{cor:step_function} where for, in  fact not a boundary, but an internal point $x_0 \in \cIcr{\Dk-1,p}(w)$ we obtain functions $\ustep \in \Hnukp{m}$ admitting jump-type discontinuity of $\tgradnuk{\bk}\ustep$ at $x_0$. The base Theorem \ref{thm:dirac_delta_approx_crit_point}, however, was formulated for arbitrary $x_0 \in \bar{I}$, including boundary points and thus can be adapted to prove the mentioned necessity. We will incorporate this result below in a theorem that enables to approximate any function $u \in \Hnukp{m}$ by smooth functions $u_h \in \D(\R)$ that have pre-fixed (independent of $h$) boundary values of functions/derivatives $\der{k}{u_h}$ for any natural $k\geq 0$. The proposed version of the statement is tailored for the other work of the same author \cite{bolbotowski2020a} where the weighted Sobolev functions are considered on graphs in $\R^d$: the proposition will serve as lemma while constructing an approximating sequence of $\bar{u}_h \in \D(\Rd)$ which requires gluing edge-wise defined functions $u^i_h \in \D(E_i)$ at junctions, $E_i$ denoting the $i$-th edge.
	
	\begin{theorem}
		\label{thm:trace_approximation}
		For an interval $I= \Iaa$, an exponent $p \in \prange$, a stable weight $w \in \nolinebreak L^1(I)$ and order $m \in \mathbb{N}_+$ we consider a function in a weighted Sobolev space $u \in \nolinebreak \Hnukp{m}$ and a boundary point $\am$. By $\Dk \geq 1$ we denote the smallest positive integer such that
		\begin{equation*}
		\am \notin \cIcr{\Dk -1,p}(w)
		\end{equation*}
		or we put $\Dk = \infty$ whenever $\am \in \cIcr{\al,p}(w)$ for any $\al\geq 0$.
		
		We set $\bk = \max\{m-\Dk,-1\} \in \{-1,0,\ldots,m-1\} $ and choose any $m$ numbers $v^-_0,\ldots,v^-_{m-1} \in \R$ that satisfy:
		\begin{enumerate}[(i)]
			\item for indices $k$ such that $0 \leq k \leq \bk$ (such $k$ may not exist)
			\begin{equation}
			\label{eq:Tr_vk}
			\Tr{k}{u}(\am) = v^-_k;
			\end{equation}
			\item for indices $k$ such that $\bk+1 \leq k \leq m-1$ (such $k$ may not exist)
			\begin{equation}
			v^-_k \text{ is an arbitrary real}.
			\end{equation}
		\end{enumerate}
		Then for any $\eps >0$ there exists a smooth function $u_\eps \in \D(\R)$ and $\delta = \delta(\eps)>0$ such that
		\begin{equation}
		\label{eq:u_ueps}
		\norm{u-u_\eps}_\Hnukp{m} < \eps
		\end{equation}
		while
		\begin{alignat}{3}
		\label{eq:low_Dk}
		\der{k}{u_\eps}(\am) &= v^-_k \qquad &&\text{for } k\in\{0,\ldots,m-1\},\\
		\label{eq:high_Dk}
		\der{m}{u_\eps}(x) &= 0 \qquad &&\text{for every } x\in B(\am,\delta)\cap I.
		\end{alignat}
		The analogical fact can be independently put forward for the other boundary point $\ap$ and boundary values $v^+_0, \ldots , v^+_{m-1}$ (the index $\bk$ may differ). Moreover $u_\eps$ satisfying \eqref{eq:u_ueps} may be chosen such that \eqref{eq:low_Dk}, \eqref{eq:high_Dk} hold for $\am$ and $\ap$ altogether.
		
		\begin{remark}
			Prior to giving a proof of the theorem we will explain its content. The integer $\bk$ should be seen as the index of the highest tangential derivative $\der{\bk}{u}$ that is well-defined at $\am$ or, equivalently, the highest order for which the trace operator $\Tr{\bk}{(\, \cdot \,)}(\am)$ has a continuous extension to $\Hnukp{m}$; the case when $\bk = -1$ simply means that there is no such order.  For $k \leq \bk$, due to continuity of $\Tr{k}{(\, \cdot \,)}(\am)$ the boundary value $v_k = \der{k}{u_\eps}(\am)$ cannot be different from a suitable trace if we wish that $u_\eps$ approximate $u$ arbitrarily close. To the contrary, if $k \geq \bk+1$ we are able to produce any boundary value $\der{k}{u_\eps}(\am)$ with the $\Hnukp{m}$-norm-cost as small as we would like. Our choice in the theorem is governed by the application in \cite{bolbotowski2020a} where we require the function $u_\eps$ to be a particular polynomial of degree $m-1$ in some neighbourhood of the boundary point $\am$. 
		\end{remark}
		\begin{proof}
			We will first display the proof only for the boundary point $\am$. For a given $\delta>0$ by definition of $\Hnukp{m}$ we may choose a smooth function $\ut \in \D(\R)$ such that
			\begin{equation}
			\label{eq:u-ut}
			\norm{u-\ut}_\Hnukp{m} < \delta.
			\end{equation}
			We put for every $k \in \{0,\ldots,m-1\}$
			\begin{equation}
			\label{eq:Dvk}
			\Dv_k := v^-_k -\der{k} \ut (\am) \in \R.
			\end{equation}
			We shall construct the target function $u_\eps$ by modifying the function $\ut$ in three steps below; C \nolinebreak will denote a constant that may change from line to line.
			
			\vspace{2mm}
			\noindent \underline{Step I (only if $ \bk \geq 0 $):} 
			\vspace{1mm}
			
			\noindent For every $k \in \{0,\ldots,\bk\}$ due to \eqref{eq:Tr_vk} we have
			\begin{equation*}
			\Dv_k = \Tr{k}{u}(\am) - \Tr{k}{\ut}(\am).
			\end{equation*}
			and by continuity of $\Tr{k}{(\,\cdot\,)}(\am)$ in $\Hnukp{m}$ for $k\leq \bk$ (guaranteed by Proposition \ref{prop:trace_extension}) we obtain for a finite constant $C>0$ independent of $\delta$ and $k \leq \bk$
			\begin{equation}
			\label{eq:dvk_Cdelta}
			\abs{\Dv_k} \leq C \norm{u - \ut}_\Hnukp{m} < C \, \delta.
			\end{equation}
			We introduce the first modification $\ut_\mathrm{I} := \ut + \Delta \ut_\mathrm{I} \in C^\infty(\R)$ where for any $x \in \R$
			\begin{equation*}
			\Delta \ut_\mathrm{I} (x) := \sum_{n=0}^{\bk} \frac{\Dv_n}{n!} (x-\am)^n
			\end{equation*}
			which gives
			\begin{alignat}{3}
			\nonumber
			&\der{k}{\ut_\mathrm{I}}(\am) = v^-_k \qquad &&\text{for each } k \in \{0,\ldots,\bk \},\\
			\label{eq:vk_after_I}
			&\der{k}{\ut_\mathrm{I}}(x) = \der{k}{\ut}(x) \qquad &&\text{for all } x\in \R \text{ and each } k \geq \bk+1.
			\end{alignat}
			It is easy to verify that the following estimate holds for each $k \leq \bk$ based on \eqref{eq:dvk_Cdelta}:
			\begin{equation}
			\label{eq:DuI_Lpnu_delta}
			\norm{\der{k}{\Delta \ut_\mathrm{I}}}_\Lpnu \leq  \sum_{n=k}^{\bk} \frac{\abs{I}^{n-k}}{(n-k)!} \norm{w}^{1/p}_{L^1(I)} \abs{\Dv_n} < C(k) \, \delta,
			\end{equation}
			where $\abs{I} = \Leb^1(I) = \abs{\ap-\am}$ and the finite positive constant $C(k)$ is independent of $\delta$. For $k \geq \bk+1$ we obviously have $\der{k}{\Delta \ut_\mathrm{I}} \equiv 0$. By combining \eqref{eq:u-ut} and \eqref{eq:DuI_Lpnu_delta} above we infer that $\delta$ may be chosen so that
			\begin{equation}
			\label{eq:}
			\norm{u-\ut_\mathrm{I}}_\Hnukp{m} \leq \norm{u-\ut}_\Hnukp{m} + \norm{\Delta \ut_\mathrm{I}}_\Hnukp{m} < \delta + C \delta < \eps/3
			\end{equation} 
			In the steps to follow we agree that such $\ut_\mathrm{I}$ is fixed, while $\delta$ remains live.
			
			\vspace{2mm}
			\noindent \underline{Step II (only if $\bk \leq m-2$):} 
			\vspace{1mm}
			
			\noindent By the definition of $\bk$ we have $\am \in \cIcr{m-\bk-2}$ and hence, for any $\delta>0$ Theorem \ref{thm:dirac_delta_approx_crit_point} furnishes a function $\hphi$ satisfying (we assume that $\delta\leq \abs{I}$)
			\begin{equation}
			\label{eq:hphi_fixed}
			\hphi \in \D \bigl( \Br(\am,\delta) \bigr), \qquad \int_{\Br(\am,\delta)} \hphi\, dx = 1, \qquad \norm{\der{m-\bk-2}{\hphi}}_\Lpnu < \delta.
			\end{equation}
			Since the weight $w$ is assumed to be stable we may choose the function $\hphi$ above such that also $\norm{\der{k}{\hphi}}_\Lpnu < \delta$ for any other $k \in \{0,\ldots,m-\bk-2\}$, see \eqref{eq:poincare_deriv}, \eqref{eq:poincare} in the proof of Corollary \ref{cor:step_function}.
			
			We introduce a smooth approximation $\tilde{s} \in C^\infty(\R)$ of step function $\mathbbm{1}_{(-\infty,\am)}$ that for every $x\in\R$ reads
			\begin{equation*}
			\tilde{s}(x) = 1 - \int_{-\infty}^{x} \hphi(y) \, dy;
			\end{equation*}
			by \eqref{eq:hphi_fixed} we notice that
			\begin{equation}
			\label{eq:prop_tildes}
			0 \leq \tilde{s} \leq 1, \quad \mathrm{supp} (\tilde{s}) \cap I \subset \Br(\am,\delta), \quad \tilde{s} \equiv 1 \text{ in } \Br(\am,\delta_1) \text{ for some } 0<\delta_1< \delta. 
			\end{equation}
			We define the next modification $\ut_\mathrm{II} := \ut_\mathrm{I} +  \Delta \ut_\mathrm{II} \in C^\infty(\R)$ where
			\begin{equation*}
			\Delta \ut_\mathrm{II}(x) :=  \sum_{n=\bk+1}^{m-1} \Dv_n \cdot
			\left\{
			\begin{array}{cl}
			\int_{\am}^{x} \tilde{s}(y) \frac{(x-y)^{n-1}}{(n-1)!} dy &\ \ \mathrm{if}\ \ n \geq 1, \\
			\tilde{s}(x) &\ \ \mathrm{if}\ \ n=0.
			\end{array}
			\right.
			\end{equation*}
			According to \eqref{eq:prop_tildes} we have $\Delta \ut_\mathrm{II}(x) = \sum_{n=\bk+1}^{m-1} \frac{\Dv_n}{n!} (x-\am)^n$ for any $x \in \Br(\am,\delta_1)$. After acknowledging \eqref{eq:Dvk} and \eqref{eq:vk_after_I} we thus obtain
			\begin{alignat}{3}
			\nonumber
			&\der{k}{\ut_\mathrm{II}}(\am) = v^-_k \qquad &&\text{for each } k \in \{0,\ldots,m-1 \},\\
			\label{eq:vk_after_II}
			&\der{m}{\ut_\mathrm{II}}(x) = \der{m}{\ut}(x) \qquad &&\text{for all } x\in \Br(\am,\delta_1).
			\end{alignat}
			We may compute the $m$-th derivative $\der{m}{\bigl(\Delta \ut_\mathrm{II}\bigr)} = - \sum_{n=\bk+1}^{m-1} \Dv_n \,\der{m-n-1}{\hphi} $. In order to estimate the $\Hnukp{m}$-norm of $\Delta \ut_\mathrm{II}$ we recognize it together with its all $m$ derivatives as linear combinations with coefficients $\Dv_k$ of functions whose $\Lpnu$-norms we estimate below:
			\begin{alignat}{3}
			\label{eq:est_hphi}
			&\norm{\der{k}{\hphi}}_\Lpnu < \delta \qquad &&\text{ for each } k\in \{0,\ldots, m-\bk-2 \},\\
			\label{eq:est_s}
			&\norm{\tilde{s}}_\Lpnu = \biggl( \int_I w \, \abs{\tilde{s}}^p dx \biggr)^{1/p} &&\leq \biggl(\int_{\Br(\am,\delta)} w \,dx \biggr)^{1/p},\\
			\label{eq:est_int_s}
			&\left\lVert\int_{\am}^{\,\cdot\,} \tilde{s}(y) \frac{(\,\cdot\, - y)^{n-1}}{(n-1)!} dy \right \rVert_\Lpnu &&\leq C \cdot \norm{w}^{1/p}_{L^1(I)} \cdot \delta  \quad \text{for any interger } n\geq 1.
			\end{alignat} 
			where $C$ depends only on $n$ and the length $\abs{I}$. The first inequalities were forced by the choice of $\hphi$, while the second one is clear upon acknowledging \eqref{eq:prop_tildes}. The last estimate needs explaining: due to \eqref{eq:prop_tildes} for all $x \in I$  we observe that  $\abs{\int_{\am}^{x} \tilde{s}(y) \frac{(x-y)^{n-1}}{(n-1)!} dy} \leq  \abs{\int_{\am}^{\ap} \tilde{s}(y) \frac{(\ap-y)^{n-1}}{(n-1)!} dy} \leq \abs{\int_{\am}^{\am+\delta} \frac{(\ap-y)^{n-1}}{(n-1)!} dy} = \frac{1}{n!} \abs{(\ap-\am)^n -(\ap-\am-\delta)^n}$, which, upon acknowledging that $\ap-\am =\abs{I}$ and $\delta < \abs{I}$, yields inequality \eqref{eq:est_int_s}.
			
			The estimates \eqref{eq:est_hphi}, \eqref{eq:est_int_s} and \eqref{eq:est_s} together with absolute continuity of Lebesgue integral allow to choose $\delta>0$ such that $\norm{\Delta \ut_\mathrm{II}}_\Hnukp{m} < \eps/3$ and thus
			\begin{equation*}
			\norm{u - \ut_\mathrm{II}}_\Hnukp{m} \leq \norm{u - \ut_\mathrm{I}}_\Hnukp{m}+ \norm{\Delta \ut_\mathrm{II}}_\Hnukp{m} < \frac{\eps}{3}+\frac{\eps}{3} = \frac{2\eps}{3}.
			\end{equation*}
			We assume that $\ut_\mathrm{II}$ satisfying above is fixed and again $\delta>0$ will be arbitrary below. Notwithstanding this we bear in mind the $\delta_1$ that appears in \eqref{eq:vk_after_II}.
			
			\vspace{2mm}
			\noindent \underline{Step III:} 
			\vspace{1mm}
			
			\noindent Let $\delta$ be an  arbitrary positive number smaller than $\abs{I}/2$. We choose a smooth cut-off function  $\varphi_\delta$ satisfying
			\begin{equation*}
			\varphi_\delta \in \D\bigl( B(\am,2\delta) \bigr), \qquad 0\leq \varphi_\delta \leq 1, \qquad \varphi_\delta \equiv 1 \quad \text{in} \quad B(\am,\delta)
			\end{equation*}
			and then we define the third modification $\ut_\mathrm{III} := \ut_\mathrm{II} +  \Delta \ut_\mathrm{III} \in C^\infty(\R)$ where
			\begin{equation*}
			\Delta \ut_\mathrm{III}(x) := -\int_{\am}^{x} \bigl(\varphi_\delta \cdot\der{m}{\ut} \bigr)(y) \ \frac{(x-y)^{m-1}}{(m-1)!} dy.
			\end{equation*}
			We note that for $k \leq m-1$ we obtain $\der{k}{\Delta \ut_\mathrm{III}}(\am) = 0$; moreover, we have $\der{m}{\Delta \ut_\mathrm{III}} = - \der{m}{\ut}$ in  $\Br(\am,\delta)$, hence, whenever $\delta \leq \delta_1$ (see \eqref{eq:vk_after_II}) we ultimately arrive at
			\begin{alignat}{3}
			\nonumber
			&\der{k}{\ut_\mathrm{III}}(\am) = v^-_k \qquad &&\text{for each } k \in \{0,\ldots,m-1 \},\\
			\label{eq:vk_after_III}
			&\der{m}{\ut_\mathrm{III}}(x) = 0 \qquad &&\text{for all } x\in \Br(\am,\delta).
			\end{alignat}
			Similarly as in \eqref{eq:est_int_s} we estimate for $k \in \{0,\ldots,m-1\}$
			\begin{equation*}
			\norm{\der{k}{\Delta \ut_\mathrm{III}}}_\infty \leq \norm{\der{m}{\ut}}_\infty \, \abs{\int_{\am}^{2\delta} \ \frac{(\ap-y)^{m-1-k}}{(m-1-k)!} dy} \leq \norm{\der{m}{\ut}}_\infty \, C(k) \, \delta 
			\end{equation*}
			where, aside from $k$, the constant $C(k)$ depends only on $m$ and $\abs{I}$. For the highest derivative we observe that $\der{m}{\Delta \ut_\mathrm{III}} = - \varphi_\delta \der{m}{\ut}$. We can readily write down  the inequality:
			\begin{alignat*}{2}
			\norm{\Delta \ut_\mathrm{III}}_\Hnukp{m} &\leq \norm{\der{m}{\Delta \ut_\mathrm{III}}}_\Lpnu +  \sum_{k = 0}^{m-1} \norm{\der{k}{\Delta \ut_\mathrm{III}}}_\Lpnu\\
			& \leq \norm{\der{m}{\ut}}_\infty \, \biggl(\int_{\Br(\am,2\delta)} w \,dx \biggr)^{1/p} + \sum_{k = 0}^{m-1} \norm{\der{m}{\ut}}_\infty\, \norm{w}^{1/p}_{L^1(I)}\, C(k)\, \delta,
			\end{alignat*}
			and so, recalling absolute continuity of Lebesgue integral once more, there exists $\delta>0$ such that $\norm{\Delta \ut_\mathrm{III}}_\Hnukp{m} < \eps/3$ and eventually
			\begin{equation*}
			\norm{u - \ut_\mathrm{III}}_\Hnukp{m} \leq \norm{u - \ut_\mathrm{II}}_\Hnukp{m}+ \norm{\Delta \ut_\mathrm{III}}_\Hnukp{m} < \frac{2\eps}{3}+\frac{\eps}{3} = \eps
			\end{equation*}
			which ends Step III.
			
			The proof for the boundary point $\am$ is thus almost finished. At this point we choose another smooth cut-off function $\varphi$, for instance $\varphi \in \D\bigl( B(\am,\abs{I}/2) \bigr)$ with $0 \leq \varphi \leq 1$ and $\varphi \equiv 1$ in $B(\am,\abs{I}/4))$. We define
			\begin{equation*}
			u_\eps := \ut + \varphi \cdot \left( \Delta \ut_\mathrm{I}+\Delta \ut_\mathrm{II}+\Delta \ut_\mathrm{III} \right).
			\end{equation*}
			If we assume that $\delta$ chosen as above is smaller than $\abs{I}$ then $u_\eps$ clearly satisfies \eqref{eq:low_Dk},\eqref{eq:high_Dk} as in the thesis, since $\ut_\mathrm{III}$ did so. The estimates on $\norm{\Delta \ut_\mathrm{I}}_\Hnukp{m}$, $\norm{\Delta \ut_\mathrm{II}}_\Hnukp{m}$ and $\norm{\Delta \ut_\mathrm{III}}_\Hnukp{m}$ will also hold for $\norm{\varphi\cdot \Delta \ut_\mathrm{I}}_\Hnukp{m}$, $\norm{\varphi\cdot\Delta \ut_\mathrm{II}}_\Hnukp{m}$ and $\norm{\varphi\cdot\Delta \ut_\mathrm{III}}_\Hnukp{m}$ since $\varphi$ is fixed (independent of $\delta$) but possibly with higher constants. Therefore, in order to guarantee $\norm{u-u_\eps}_\Hnukp{m} < \eps$ we might have to pick a smaller $\delta$, the proof for $\am$ is nevertheless complete.
			
			The important property of the function $u_\eps$ constructed above is that it did not affect the boundary values of the initial approximation $\ut$. Hence, by redefining $\ut$ as $u_\eps$ we may go through the same steps I,II,III together with cutting-off by $\varphi\in \D\bigl( B(\ap,\abs{I}/2) \bigr)$ for the other boundary point $\ap$. The outcome will be the new function $u_\eps$ that satisfies $\norm{u- \nolinebreak u_\eps}_\Hnukp{m} < 2 \eps$ and the boundary conditions \eqref{eq:low_Dk},\eqref{eq:high_Dk} at both boundary points $\am$ and $\am$ with values $v^-_k$ and $v^+_k$ respectively. The proof is at an end. 
		\end{proof}
	\end{theorem}

	\section{Application to variational problems on the example of linear elasticity in beams}
	\label{sec:conclusions}
	
	We conclude this work by revisiting the topic that was essentially the motivation for the whole contribution -- the second-order variational problem of linear elasticity in beam with a degenerate distribution of width $w$. In this section we repeatedly make use of results herein derived, we show that the developed theory of weighted Sobolev spaces $\Hnukp{m}$ allows to successfully relax the original variational problem formulated for smooth displacement functions $u \in \D(I)$ and, thanks to one-dimensional setting, also to construct the solutions. We will be additionally required to examine coercivity of the underlying energy functional; on top of that we shall characterize the space dual to $\Hnukp{m}$. For clarity we dare not to mind physical units in the process. 
	
	Our beam will occupy an interval $I =\Iaa$ with $\am =0$ and $\ap = 4$; we also distinct three internal points $x_1 =1, x_2=2, x_3=3$. For a given stable weight $w \in L^1(I)$ we consider an elliptic Dirichlet boundary problem of second order in its variational form:
	\begin{equation}
	\label{eq:initial_dirichlet_var}
	\inf \biggl\{ J_I\bigl(D^2 u\bigr) - \bigl \langle u,f \bigr \rangle \ : \ u \in \D(I) \biggr\}, \quad \text{where} \quad J_I(v) = \frac{1}{2} \int_I w \, \abs{v}^2\, dx
	\end{equation}
	and $f \in \D'(I)$ is a given distribution. With $w$ treated as a width function and $f$ as a loading, we readily recognized the above as the elasticity problem for a clamped (note the homogeneous Dirichlet boundary conditions) beam, while function $u$ that potentially solves the problem will be non other that the deflection function of the beam. For the variational problem to admit a solution we naturally must relax the condition on smoothness of function $u$. In analogy with the works from \cite{bouchitte1997} or \cite{bouchitte2003} our proposition is to use the space $\Hnuktwo$ or, more precisely, to consider a relaxed version of the problem \eqref{eq:initial_dirichlet_var}:
	\begin{equation}
	\label{eq:relaxed_dirichlet_var}
	\inf \biggl\{ J_I\bigl(\tgradnuk{2} u\bigr) - \bigl \langle u,\bar{f} \bigr \rangle \ : \ u \in \U^{2,2}_\muw \biggr\}, \quad \text{where} \quad J_I(v) = \frac{1}{2} \int_I w \, \abs{v}^2\, dx.
	\end{equation}
	By $\U^{2,2}_\muw$ we understand the closure of the space $\D(I)$ in topology of $\Hnuktwo$. The functional $\bar{f}$ stands for the continuous extension of the linear functional $f \in \D'(I)$ to the dual space $\bigl(\Hnuktwo\bigr)^*$; obviously such an extension may not exists, however, it would result in the infimum from \eqref{eq:initial_dirichlet_var} being equal to $-\infty$ and seeking relaxation would be hopeless in the first place. The functional $J_I\bigl( \tgradnuk{2} \, \cdot \, \bigr)$ is convex, lower semi-continuous (and even continuous) in $\Hnuktwo$  -- it is in fact the lower semi-continuous regularization of functional $J_I\bigl(D^2 \, \cdot \, \bigr)$ extended to $\Hnuktwo$ by $+\infty$, the details are moved to \cite{bolbotowski2020a}. For existence of solution in \eqref{eq:relaxed_dirichlet_var} we are left to show coercivity and this matter is more delicate. Naturally, the most we can obtain is coercivity of $J_I\bigl( \tgradnuk{2} \, \cdot \, \bigr)$ in the quotient space $\Hnuktwo / \ker \tgradnuk{2}$. This will be the case if we impose an additional condition on the weight that can be seen as a generalized global Poincar\'{e} inequality:
	\vspace{2mm}
	\begin{equation}
	\label{eq:general_Poincare}
	\text{there exists } C>0 \text{ such that} \quad \norm{u - P_{\,\ker \! \tgradnuk{2}}\! (u)}_{H^{1,2}_\muw} \leq C \, \norm{ \tgradnuk{2} u }_{L^2_\muw} \quad \forall \, u \in \Hnuktwo,
	\end{equation}
	where $P_{\,\ker\! \tgradnuk{2}}$ denotes the orthogonal projection in the Hilbert space $\Hnuktwo$. For $p \in \prange$ the LHS of the inequality can be replaced by the quotient norm in $\Hnukp{2} / \ker \tgradnuk{2}$. The non-triviality of the inequality lies in the structure of $\ker \tgradnuk{2}$ which may be larger than the two-dimensional space of affine functions being the case for $w \equiv c >0$. Corollary \ref{cor:step_function} states that for a.e. positive weights $w$ the subspace $\ker \tgradnuk{2}$ may contain e.g. step functions. Further we could easily check that the above Poincar\'{e} inequality is false for the non-stable weight $\wcrit$ proposed in the proof of Theorem \ref{thm:jump_summary}: by means of duality the theorem guarantees existence of a sequence $u_h \in \D(\R)$ with $u_h \rightarrow \mathbbm{1}_{(x_0,\ap)}$, $\tgradnuk{2} u_h \rightarrow 0$ in $L^2_\muw$, while $\mathbbm{1}_{(x_0,\ap)} \notin \Hnuktwo$ and thus also $\mathbbm{1}_{(x_0,\ap)} \notin \ker \tgradnuk{2}$ (we cannot force the sequence $\tgradnuk{} u_h$ to converge to zero in $L^2_\muw$). We expect the inequality \eqref{eq:general_Poincare} to hold for any stable weight yet the proof seems to be difficult in full generality. In this work we limit ourselves to show validity of the inequality in the case when interval $I$ can be partitioned into finite number of intervals on which $w$ is monotonic -- a simple argument will be demonstrated on the example of weight showed in Fig. \ref{fig:solution_for_beams}(b). 
	
	Since $\Hnuktwo$ is a Hilbert space (reflexivity for other $p \in (1,\infty)$ suffices), then once the Poincar\'{e} inequality \eqref{eq:general_Poincare} is established the relaxed elasticity problem \eqref{eq:relaxed_dirichlet_var} has a solution $\check{u}$ as soon as the continuous extension $\bar{f} \in \bigl( \Hnuktwo \bigr)^*$ of distribution $f \in \D(\R)$ exists and $\bar{f} \perp \U^{2,2}_{0,\muw}:= \ker \tgradnuk{2} \cap \, \U^{2,2}_\muw$. Should solution $\check{u}$ exist, it is unique up to a zero-energetic displacement function $u_0 \in \U^{2,2}_{0,\muw}$.
	
	A few words on the dual space to $\Hnukp{m}$ are in order. To provide its characterization we may simply repeat the argument from Chapter 3 in the book of \cite{adams2003} that is intended for standard Sobolev spaces. The result is that a functional $\Lambda$ is an element of $\bigl( \Hnukp{m} \bigr)^*$ for $p \in \prange$ if and only if it is expressed by a formula $\Lambda u = \sum_{k=0}^{m} \int_I w \, \bigl(\tgradnuk{k}u\bigr)\, v_k^* \,dx$ for some (not necessarily unique) family of functions $v_0^*,v_1^*,\ldots,v_m^* \in L^{p'}_\muw$ where $p'$ is the H\"{o}lder conjugate exponent of $p$. From H\"{o}lder inequality we may easily show that for any $v^* \in L^{p'}_\muw$ there holds $w \, v^* \in L^1(I)$. We immediately infer that a distribution $f \in \D(\R)$ extends to $\bar{f} \in \bigl( \Hnukp{m} \bigr)^*$ if and only if it is of the form $f = \sum_{k=0}^{m} (-1)^k D^k(w \, v_k^*)$ where $D^k$ stands for $k$-th distributional derivative. 
	
	We will construct solutions $\check{u} \in \Hnuktwo$ of the relaxed problem \eqref{eq:relaxed_dirichlet_var} for four cases (a)-(d) of the width/weight $w$. The widths are shown in Fig. \ref{fig:solution_for_beams}(a)-(d) respectively, the reader should consider the picture of $w$ as a view of the beam from the top. Except for the case (a), the width $w$ will vary in each of the four subintervals $(\am,x_1), (x_1,x_2), (x_2,x_3)$ and $(x_3,\ap)$ as the function $\abs{x-\am}^{\gamma_1}$, $\abs{x-x_2}^{\gamma_2}$,$\abs{x-x_2}^{\gamma_3}$ and $\abs{x-\ap}^{\gamma_4}$ for different $\gamma_i \geq 0$, i.e. $\am,x_2,\ap$ are the points of possible degeneration of the weight at different rates. The form of the distribution $f\in \D'(I)$ will be common for all the cases and it shall read:
	\begin{equation*}
	f= \sum_{i=1}^{4} f_{q_i} + \sum_{j=1}^{3} f_{F_j} + f_m, \quad f_{q_i} = \int_{i-1}^i q_i (\,\cdot\,)dx, \quad f_{F_j} = F_j \delta_{x_j}, \quad f_m = - m D(\delta_{x_2}) 
	\end{equation*} 
	where, for the time being, $q_i, F_j, m$ are arbitrary reals. The whole load/distribution $f$ is illustrated in Fig. \ref{fig:solution_for_beams}(a) in a schematic view typical for structural mechanics: distributions $f_{q_i}$ appear as a piece-wise uniformly distributed downward load, each $f_{F_j}$ plays a role of a downward point force, whereas $f_m$ represents a point moment load that rotates the centre of the beam. We note that in the beam theory it is typical to assign a positive sign to loads that are pointed downwards instead of upwards; the same convention applies to the displacement function $u$, i.e. a point $x$ with $u(x)>0$ translates downwards. We shall see that in the first two cases (a), (b) the distribution $f$ will receive a continuous extension $\bar{f} \in \bigl(\Hnuktwo \bigr)^*$ for any parameters $q_i, F_j, m$, whereas for stronger degeneration in case (c) and then (d) subsequent components of $f$ will have to vanish as otherwise they would be unbounded on $\Hnuktwo$.
	
	We briefly describe the method of solving the relaxed variational problem \eqref{eq:relaxed_dirichlet_var}. The variational problem factually solved will be the dual to $\eqref{eq:relaxed_dirichlet_var}$: we will seek a bending moment function $\check{M}$ that solves
	\begin{equation}
	\label{eq:dual_var_beam}
	\inf\biggl\{ \Jc{I}(M) : M \in L^1(I),\ D^2 M + f = 0 \biggr\}, \quad \text{where} \quad \Jc{I}(M) = \frac{1}{2}\int_I \frac{\abs{M}^2}{w} \, dx.
	\end{equation}
	This problem is easy due to one-dimensional setting where $D^2$ has a finite-dimensional kernel that consists of affine functions -- the solution exists as far as the problem is not trivially equal to $+\infty$. The link to the primal problem \eqref{eq:relaxed_dirichlet_var} leads through the optimality condition: if for a function $\check{u} \in \U^{2,2}_\muw$ a constitutive law $\check{M} = w \, \bigl(-\tgradnuk{2} \check{u} \bigr)$ holds for the solution $\check{M}$ of the dual problem, then $\check{u}$ is a solution of the primal problem \eqref{eq:relaxed_dirichlet_var}. Details of this duality-theoretic part of constructing the solution may be found in work \cite{bolbotowski2020a} that is specifically dedicated to the beam/grillage problem. Therefore, upon obtaining $\check{M}$ we actually arrive at the tangential derivative $\tgradnuk{2} \check{u}$ that must be carefully twice integrated so that factually $\check{u} \in \U^{2,2}_\muw$.
	\begin{figure}[!]
		\centering
		\subfloat[]{\includegraphics*[width=0.47\textwidth]{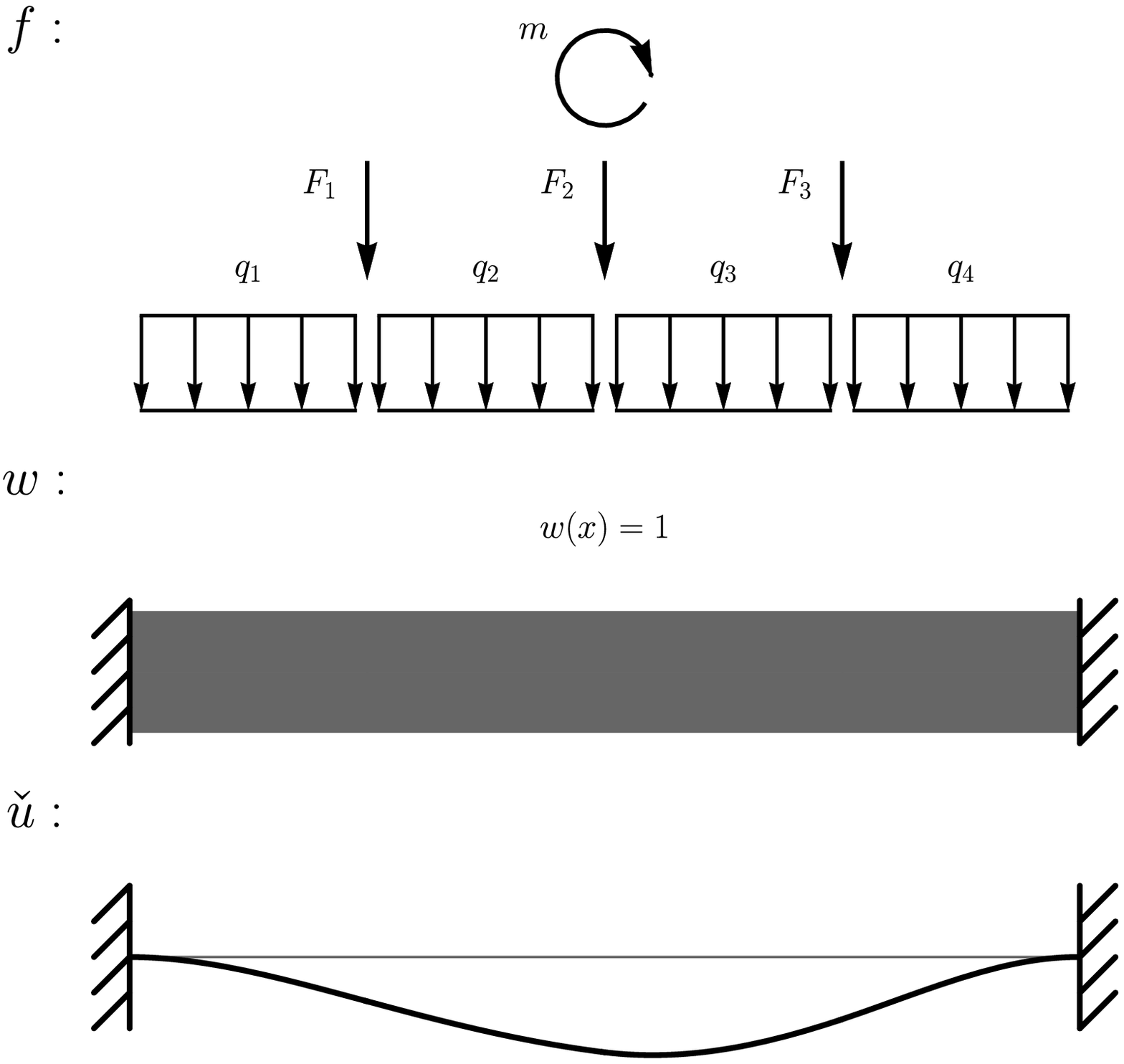}}\hspace{0.5cm}
		\subfloat[]{\includegraphics*[width=0.47\textwidth]{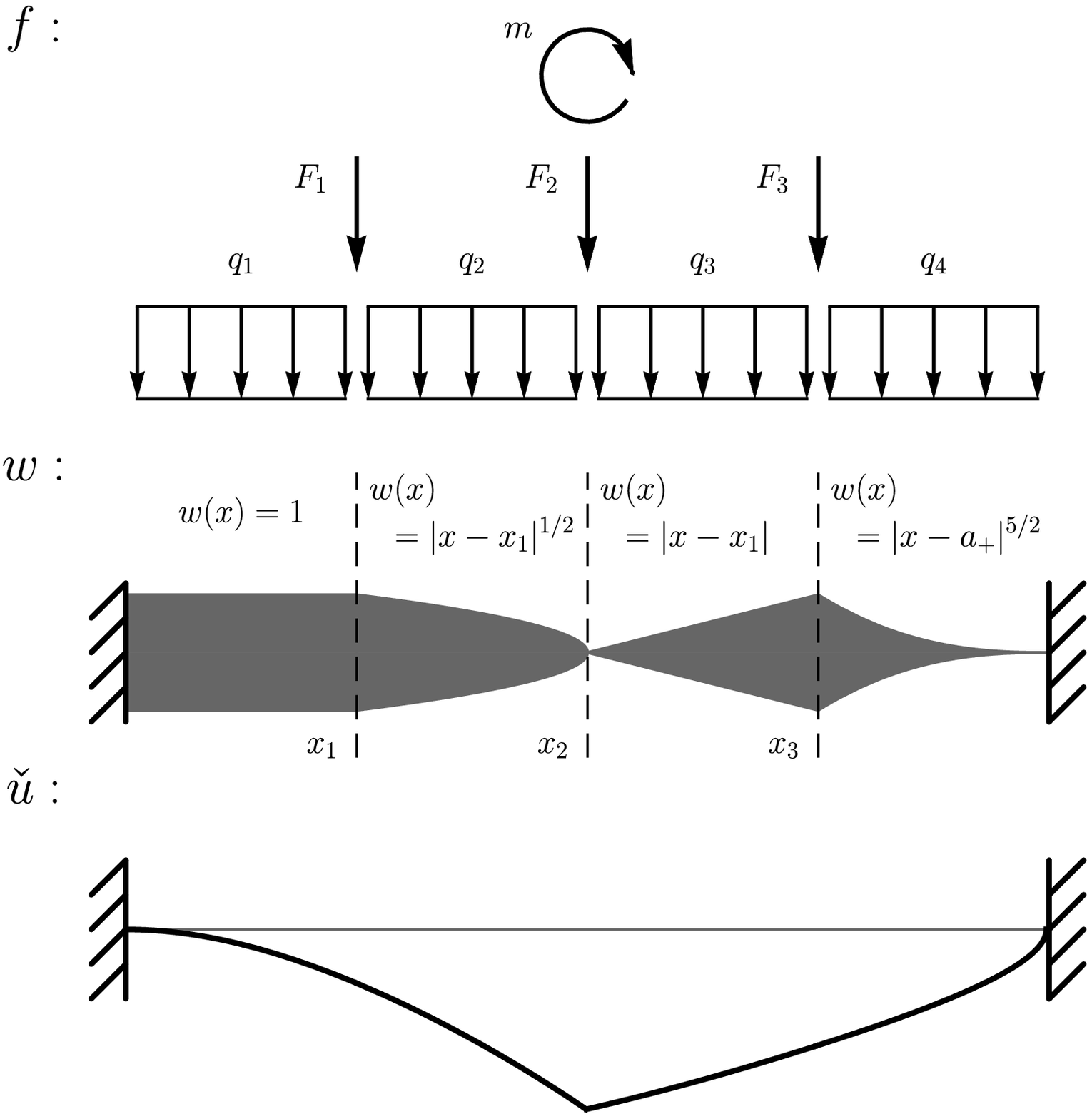}}\\
		\subfloat[]{\includegraphics*[width=0.47\textwidth]{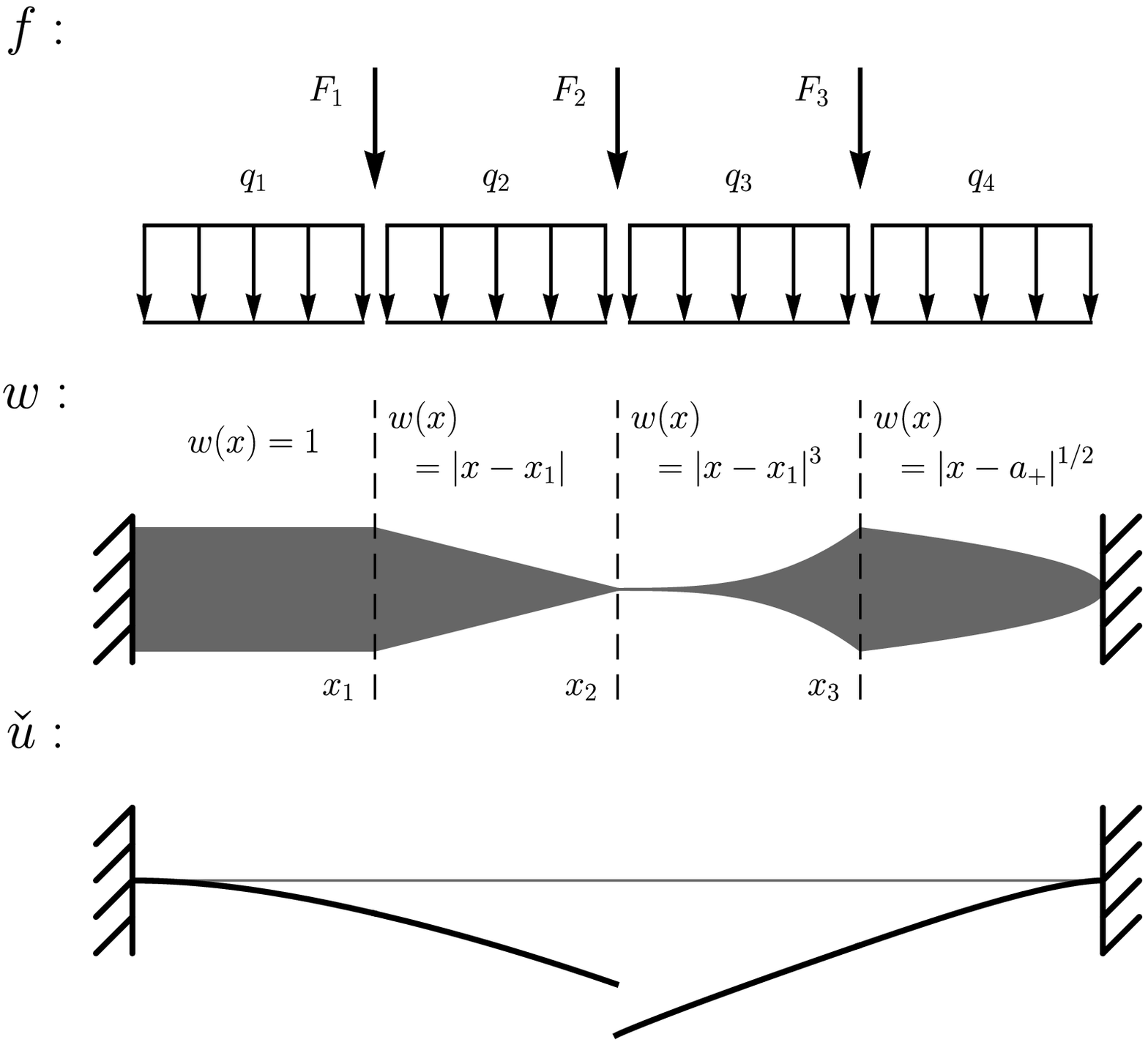}}\hspace{0.5cm}
		\subfloat[]{\includegraphics*[width=0.47\textwidth]{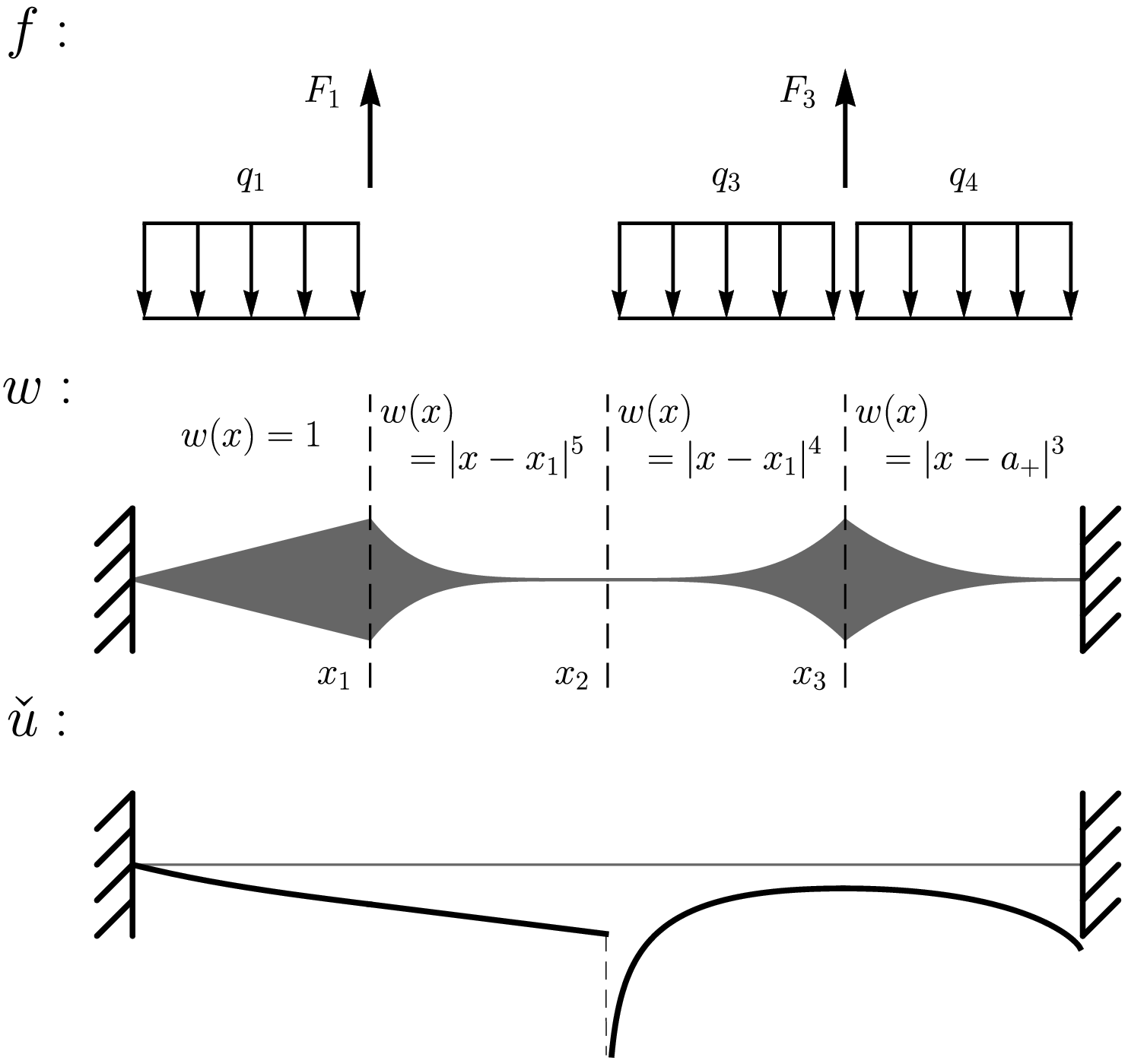}}\\
		
		\caption{Linear elasticity problem in a clamped beam subject to load $f$ and of different width functions $w$. In each case (a)-(d) a solution $\check{u} \in \Hnuktwo$ of the variational problem \eqref{eq:relaxed_dirichlet_var} is shown at the bottom.}
		\label{fig:solution_for_beams}
	\end{figure}
	
	We are ready to construct the solutions $\check{u}$ of the elasticity problem \eqref{eq:relaxed_dirichlet_var} for different widths/weights $w$, starting from the non-degenerated case and ending at the weight suffering most degeneration:
	\vspace{3mm}

	\noindent \underline{Case (a):} The uniform weight $w \equiv 1$ easily yields an equality $\Hnuktwo = H^{2,2}(I)$ which is isomorphic to the classical Sobolev space defined by weak derivatives $W^{2,2}(I)$. Accordingly, the space $\U^{2,2}_\muw$, being the closure of $\D(I)$ in $\Hnuktwo$, is isomorphic to the space $W^{2,2}_0(I)$. In turn every element $u \in \U^{2,2}_\muw$ is a $C^1$ function with zero boundary values $u(\am)=u(\ap)=0$ and $D u(\am)=D u(\ap)=0$. Since the distribution $f$ considered above is of first order it receives its natural extension $\bar{f}$ to $\bigl(\Hnuktwo\bigr)^*$ for any parameters $q_i, F_j, m$; we decide to choose $q_i = 1$ for all $i$, $F_1 = 1$, $F_2 = 2$, $F_3=3$ and $m=2$.
	
	The dual problem is a two-dimensional convex programming problem that is easily solved: having the piece-wise parabolic solution $\check{M}$ we put $\check{v}_2 := -\check{M} / w$ and we ultimately find the unique solution $\check{u}(x) = \int_\am^x \check{v}_2(y)\,(x-y) \,dy$. The minimality of $\check{M}$ guarantees that at the right end the values $\check{u}(\ap), D \check{u}(\ap)$ are indeed zero, hence $\check{u} \in \U^{2,2}_\muw$. The function $\check{u}$ is shown at the bottom of Fig. \ref{fig:solution_for_beams}(a), we stress that positive values of $\check{u}$ are drawn below the $x$-axis.
	\vspace{3mm}
	
	\noindent \underline{Case (b):} Here $x_2$ and $\ap$ are candidates for the critical points. We verify that $x_2$ is a right-sided critical point of zero order, i.e. $x_2 \in \cIcrr{0,2}(w)$; it is not, however, a left-sided critical point. Indeed, we compute for positive $\eps\leq 1$ that $\int_{\Br(x_2,\eps)} 1/w \,dx = \int_{\Br(x_2,\eps)} 1/\abs{x-x_2} \,dx = \infty$, whilst $\int_{\Bl(x_2,\eps)} 1/w \,dx = \int_{\Bl(x_2,\eps)} 1/\abs{x-x_2}^{1/2} \,dx < \nolinebreak \infty$. Similarly $\ap \in \cIcrl{0,2}(w)$, yet non of the two points is a critical point of first order, for instance $\int_{\Bl(\ap,\eps)} \abs{x-\ap}^2/w(x) \, dx = \int_{\Bl(\ap,\eps)} \abs{x-\ap}^2/\abs{x-\ap}^{5/2} \, dx < \infty$.
	
	In summary we have a stable weight $w$ with $\cIcr{0,2}(w) = \{x_2,\ap\}$ and $\cIcr{1,2}(w) = \nolinebreak \varnothing$; we utilize those informations through results from Sections \ref{sec:sufficient_conditions}, \ref{sec:necessary_conditions} and \ref{sec:trace_extension} in order to characterize, and essentially to find, a solution $\check{u} \in \Hnuktwo$ of the variational problem \eqref{eq:relaxed_dirichlet_var}. Firstly, Corollary \ref{cor:higer_order_continuity} from Section \ref{sec:sufficient_conditions} states that $\cIcr{1,2}(w) = \varnothing$ implies the embedding $\Hnuktwo \hookrightarrow W^{1,1}(I)$, that is, in particular, every function $u \in \Hnuktwo$ is absolutely continuous in $I$. Further, due to $\cIcr{0,2}(w) = \{x_2,\ap\}$ the same corollary furnishes that restrictions of each $u \in \Hnuktwo$ to intervals $(\am,x_2)$ and $(x_2,\ap)$ are in Sobolev spaces $W^{2,1}_\mathrm{loc}\bigl( (\am,x_2) \bigr)$ and $W^{2,1}_\mathrm{loc}\bigl( (x_2,\ap) \bigr)$ respectively. In fact, since $x_2 \notin \cIcrl{0,2}(w)$, for the first interval $I_1 := (\am,x_2)$ we can say more by skipping locality: by restricting $\muw$ to $\mu_{w,1}:= \muw \mres I_1$ we see that $\Hnuktwo \hookrightarrow H^{2,2}_{\mu_{w,1}} \hookrightarrow W^{2,2}(I_1)$ and it follows from the fact that $\bar{I}_{1,cr}^{\,0,2}\bigl(w\vert_{I_1} \bigr) = \varnothing$.
	
	On the contrary, due to $x_2 \in \cIcr{0,2}(w)$, Corollary \ref{cor:step_function} in Section \ref{sec:necessary_conditions} states that there exists a function $\ustep \in \Hnuktwo$ such that $\tgradnuk{} \ustep = \mathbbm{1}_{(x_2,\ap)}$ and $\tgradnuk{2} \ustep = 0$ in $\Lpnu$; the explicit formula reads $\ustep(x)=\int_{\am}^{x} \mathbbm{1}_{(x_2,\ap)}(y)\, dy$. In the context of our variational problem this renders $\ustep$ as an element of $\ker \tgradnuk{2}$. We can in fact write 
	\begin{equation}
	\label{eq:ker_Dmuw_b}
	\ker \tgradnuk{2} = \left\{ u_P + \beta \, \ustep : u_P \in \mathrm{P}^1,\ \beta \in \R \right\}
	\end{equation}
	where by $\mathrm{P}^k$ we see the space of polynomials of $k$-th degree. Indeed, $\ker \tgradnuk{2}$ does not contain other functions since the embeddings $\Hnuktwo \hookrightarrow W^{2,1}_\mathrm{loc}\bigl( (\am,x_2) \bigr)$ and $\Hnuktwo \hookrightarrow W^{2,1}_\mathrm{loc}\bigl( (x_2,\ap) \bigr)$ imply that a function $u \in \Hnuktwo$ with $\tgradnuk{2}u \equiv 0$ is affine separately on both intervals $(\am,x_2)$ and $(x_2,\ap)$ and thus the space given in \eqref{eq:ker_Dmuw_b} could only be missing a function $\mathbbm{1}_{(x_2,\ap)}$ which is not in $\Hnuktwo$ due to $x_2\notin \cIcr{1,2}(w)$. Eventually we see that $\dim \bigl( \ker \tgradnuk{2} \bigr) = 3$ and so, here degeneracy of $w$ at $x_2$ increases the dimension by one with respect to Case (a).
	
	Now that the space $\ker \tgradnuk{2}$ was established we may pass to check validity of the Poincar\'{e} inequality \eqref{eq:general_Poincare}. We shall base our argument on the fact that interval $I$ splits to finite number of subintervals where $w$ is monotonic, which also applies to the other Cases (c), (d). Let $u_h \in \D(\R)$ be any sequence of smooth functions such that $\norm{u_h}_{H^{1,2}_\muw} = 1$ for any $h$ whereas $\norm{\tgradnuk{2} u_h}_{L^2_\muw} \rightarrow 0$. The Poincar\'{e} inequality will be true if we manage to show that such a sequence $u_h$ must converge to an element $u_0 \in \ker \tgradnuk{2}$ in norm topology of $H^{1,2}_\muw$ or equivalently of $\Hnuktwo$. Since $u_h$ is bounded in a reflexive space $\Hnuktwo$ we may choose (without relabelling) a subsequence that converges weakly to some function $u_0$ in $\Hnuktwo$. Then $u_0$ must lie in $\ker \tgradnuk{2}$ due to $\norm{\tgradnuk{2} u_h}_{L^2_\muw} \rightarrow 0$. According to \eqref{eq:ker_Dmuw_b} the function $u_0$ is affine on each of the aforementioned subintervals. We shall first show the strong convergence $v_h := \tgradnuk{} u_h \rightarrow \tgradnuk{} u_0$ in $L^2_\muw$; we concentrate on the one of the subintervals, let us choose e.g. $(x_3,\ap)$ where $w$ is decreasing. We obtain that $v_h \rightharpoonup c$ in $L^2_{\muw \mres (x_3,\ap)}$ where $c$ is a constant function and in particular $\int_{(x_3,\ap)} w(v_h - c)\, dx \rightarrow \nolinebreak 0$. Since for any $\tilde{x} \in (x_3,\ap)$ the weight $w$ is separated from zero in $ (x_3,\tilde{x})$, we have the embedding $H^{1,2}_\muw \hookrightarrow W^{1,2}\bigl( (x_3,\tilde{x}) \bigr)$ and, from standard Poincar\'{e} inequality, we easily infer that $\norm{\tgradnuk{} v_h }_{L^2_\muw} = \norm{\tgradnuk{2} u_h}_{L^2_\muw} \rightarrow 0$ implies $v_h(x_3) \rightarrow c$. Next we utilize the monotonicity of $w$ by employing the same trick as in \eqref{eq:poincare_deriv} and we obtain that for every $x \in (x_3,\ap)$ there holds $w(x)\abs{v_h(x)-v_h(x_3)}^2 \leq C \int_{(x_3,\ap)} w(y) \abs{\tgradnuk{}v_h(y)}^2 dy$ with $C = \abs{x_3-\ap}$. We deduce that
	\begin{alignat*}{1}
	&\norm{(v_h - \tgradnuk{}u_0)\lvert_{(x_3,\ap)}}_{L^2_\muw} \\
	\leq\ &\biggl(\int_{(x_3,\ap)} w(x)\abs{v_h(x) - v_h(x_3)}^2 dx\biggr)^{1/2}
	+ \biggl(\int_{(x_3,\ap)} w(x)\abs{v_h(x_3) -c}^2 dx\biggr)^{1/2}\\
	\leq\ & C \, \biggl(\int_{(x_3,\ap)} w(y) \abs{\tgradnuk{}v_h(y)}^2 dy \biggr)^{1/2}+\abs{v_h(x_3) -c} \biggl( \int_{(x_3,\ap)} w(x) dx  \biggr)^{1/2}\\ 
	\leq\  &C\  \norm{\tgradnuk{2}u_h}_{L^2_\muw} + \abs{v_h(x_3) -c}\ \norm{w}_{L^1(I)}^{1/2}
	\end{alignat*}
	and we see that the RHS goes to zero. We may repeat the same argument for the rest of the four subintervals and, since the number of those subintervals is finite, we arrive at $\norm{v_h - \tgradnuk{}u_0}_{L^2_\mu} \rightarrow 0$, i.e. $\tgradnuk{} u_h \rightarrow \tgradnuk{} u_0$ in $L^2_\muw$. The proof that also $u_h \rightarrow u_0$ in $L^2_\muw$ follows similarly: assuming that $u_0(x)=c (x-x_3)+d$ for $x\in (x_3,\ap)$ we deduce that $u_h(x_3) \rightarrow d$ and then $w(x)\abs{u_h(x)-u_h(x_3)-c(x-x_3)}^2 \leq C \int_{(x_3,\ap)} w(y) \abs{\tgradnuk{}u_h(y)-c}^2 dy$; the estimate on $\norm{(u_h - u_0)\lvert_{(x_3,\ap)}}_{L^2_\muw}$ is carried out analogically to the one above. The Poincar\'{e} inequality is now validated. 
	
	Section \ref{sec:trace_extension} allows us to describe the boundary conditions, i.e. to characterize the space $\U^{2,2}_\muw$. Since $\am \notin \cIcr{0,2}$ and $\ap\in \cIcr{0,2}(w)$ but $\ap \notin \cIcr{1,2}(w)$, Proposition \ref{prop:trace_extension} immediately yields that
	\begin{equation}
	\label{eq:boundary_conditions}
	\U^{2,2}_\muw = \left\{ u \in \Hnuktwo \ : \Tr{}{\,u}(\am) =0, \ \Tr{1}{u}(\am) =0,\ \Tr{}{\,u}(\ap) =0 \right\},
	\end{equation}
	namely we have lost one Dirichlet boundary condition that imposed a zero derivative at the right end $\ap$. The three boundary conditions that are still live suffice to eliminate non-zero functions $u_0 \in \ker \tgradnuk{2}$ from the displacement space $\U^{2,2}_\muw$. With the Poincar\'{e} inequality established we have coercivity of $J_I\bigl(\tgradnuk{2} \, \cdot\, \bigr)$ and for solvability of the relaxed elasticity problem \eqref{eq:relaxed_dirichlet_var} we need to only make sure that the load $f \in \D(\R)$ admits its continuous extension $\bar{f} \in \left(\Hnuktwo\right)^*$.
	
	Due to the aforementioned embedding $\Hnuktwo \hookrightarrow W^{1,1}(I) \hookrightarrow C(\bar{I})$ we infer that the zero-order distributions $f_{q_i}$ and $f_{F_j}$ extends naturally, while the first-order distribution $f_m$ needs a closer look. It is supported in $\{x_2\}$ and for a smooth $u \in \D(I)$ gives $\pairing{u,f_m} = m \, D u(x_2)$, meanwhile $x_2$ is the point where the tangential derivative $\tgradnuk{} u$ may jump for $u \in \Hnuktwo$ hence the classical derivative at $x_2$ of such function cannot be well-defined. Notwithstanding this $f_m$ has a continuous extension $\bar{f}_m \in \bigl(\Hnuktwo \bigr)^*$ after all: we observe that $f_m$ can be written as $D^2(w\,v_2^*) = m \, D^2(\mathbbm{1}_{(\am,x_2)})$ where we have put $v^*_2:= m\, \mathbbm{1}_{(\am,x_2)} / w$ (one may easily check that $v^*_2 \in L^2_\muw$) and hence, according to the characterization of dual space $\bigl(\Hnuktwo \bigr)^*$, $f_m$ attains its continuous extension $\bar{f}_m$ given by $\pairing{u,\bar{f}_m} = \int_I w\, \bigl(\tgradnuk{2} u\bigr) \, v^*_2 \, dx = m \int_{(\am,x_2)} \tgradnuk{2} u\, dx$. Another, more elementary explanation for the existence of $\bar{f}_m$ stems from the already established embedding $\Hnuktwo \hookrightarrow W^{2,1}\bigl( (\am,x_2) \bigr)$ that followed from the fact that $x_2 \notin \cIcrl{0,2}(w)$ (in spite of $x_2 \in \cIcr{0,2}(w)$). This way functional $(D \, \cdot \,)(x_2)$ extends continuously to $\Hnuktwo$ in the same way as the functional $(D \, \cdot \,)(\am) =: \Tr{1}{(\, \cdot\,)}(\am)$ does. If for a function $u \in \Hnuktwo$ by $\breve{u}$ we denote its precise representative, the extension $\bar{f}$ may be written as
	\begin{equation*}
	\pairing{u,\bar{f}} = \sum_{i=1}^{4} \int_{i-1}^{i} q_i \,u \, dx + \sum_{j=1}^{3} F_j \, \breve{u}(x_j) + m\   \bigl(D \breve{u}\bigr)(x_2^-)
	\end{equation*}
	where by $\bigl(D \breve{u}\bigr)(x_2^-)$ we understand the left-sided derivative of $\breve{u}$ at $x_2$, which exists since $u \in W^{2,1}\bigl( (\am,x_2) \bigr)$. We choose the parameters $q_i, F_j, m$ identically as in Case (a). 
	
	We have finally proved that a solution of the elasticity problem \eqref{eq:relaxed_dirichlet_var} exists and now we shall give its construction again via solving the dual problem \eqref{eq:dual_var_beam}. Since the primal problem has a solution, the infimum in \eqref{eq:dual_var_beam} is finite and hence there is a solution $\check{M}$. In particular $\check{M} \in L^1(I)$ satisfies the equilibrium equation $D^2 \check{M} + f = 0$ and furnishes finite complementary energy $\Jc{I}(\check{M}) <\infty$ and we may show that there can only be one such function $\check{M}$: the two dimensional kernel of $D^2$ is precisely the space of affine functions $M_0$ on $I$ and each such non-zero function will give $\Jc{I}(M_0) = \infty$ due to presence of two critical points $x_2,\ap \in \cIcr{0,2}(w)$. The dual problem is thus in a way trivial and the unique piece-wise parabolic bending moment function $\check{M}$ is easy to find. We set a function $\check{v}_2 := -\check{M}/w$ and now our objective is to construct $\check{u}\in \U^{2,2}_\muw$ that yields $\tgradnuk{2}\check{u} = \check{v}_2$, while optimality of $\check{M}$ will render such $\check{u}$ a solution of the primal problem \eqref{eq:relaxed_dirichlet_var}. Since $\abs{\check{M}}^2/w = w \bigl(\check{M}/w \bigr)^2$ there holds $\check{M}/w \in L^2_\muw$. Further we again define  $\mu_{w,1}:= \muw \mres I_1$ with $I_1= (\am,x_2)$ and symmetrically $\mu_{w,2}:= \muw \mres I_2$ with $I_2= (x_2,\ap)$. Since $\am \notin \cIcrr{0,2}(w)$ and $x_2 \notin \cIcrl{0,2}(w)$, we have the embedding $L^2_\muw \hookrightarrow L^1\bigl( I_1 \bigr)$ and, on the other hand, $x_2 \in \cIcrr{0,2}(w)$, $ \ap \in \cIcrl{0,2}(w)$ and only the local embedding $L^2_\muw \hookrightarrow \Lloc{I_2}$ is available. We shall construct the function $\check{u}$ separately on $I_1$ and $I_2$. Firstly, for any $x \in I_1$ we may define $\check{u}_1(x) = \int_{\am}^{x} \check{v}_2(y)\, (x-y) \, dy$ that gives a function $\check{u}_1 \in W^{2,1}(I_1)$ satisfying $\Tr{\,}{\check{u}_1} = \Tr{1\,}{\check{u}_1} =0$. The second interval must be handled with more care and the definition must be local, for instance $\check{u}_2(x):=\int_{x_3}^x \check{v}_2(y)\, (x-y) \, dy$ for each $x \in I_2$ (any internal point of $I_2$ could have been chosen instead of $x_3$): we arrive at a function $\check{u}_2 \in W^{2,1}_\mathrm{loc}(I_2)$. Now we will show that $\check{u}_1 \in H^{2,2}_{\mu_{w,1}}$ and $\check{u}_2 \in H^{2,2}_{\mu_{w,2}}$. First we observe that for any $\eps>0$, thanks to approximation of $L^2_{\muw}$ functions by continuous functions and then the latter uniformly by smooth functions via mollification, we may choose $\check{v}_2^\eps \in \D(\R)$ such that $\norm{\check{v}_2-\check{v}_2^\eps}_{L^2_\muw} <\eps$. Then, starting from $I_2$, the smooth function $\check{u}_2^\eps(x):=\int_{x_3}^x \check{v}_2^\eps(y)\, (x-y) \, dy$ approximates $\check{u}_2$ in $H^{2,2}_{\mu_{w,2}}$, to show this we make use of monotonicity of $w$ on $(x_2,x_3)$ and $(x_3,\ap)$ (although stability of $w$ would have sufficed). Hence we deduce that indeed $\check{u}_2 \in H^{2,2}_{\mu_{w,2}}$ and the proof for $\check{u}_1 \in H^{2,2}_{\mu_{w,1}}$ is identical. In order to arrive at the target function $\check{u}$ we have to glue the two functions $\check{u}_1, \check{u}_2$ and take care of the boundary conditions at the right end $\ap$. Since the points $x_2,\ap$ are not in the set $\cIcr{1,2}(w)$ Proposition \ref{prop:trace_extension} guarantees that $\check{u}_1(x_2)$, $\check{u}_2(x_2)$ and $\check{u}_2(\ap)$ are meaningful -- we may thus modify the function $\check{u}_2$ by an affine function (function from $H^{2,2}_{\mu_{w,2}}$ with zero second tangential derivative) so that, without relabelling, $\check{u}_2(x_2) = \check{u}_1(x_2)$ and $\check{u}_2(\ap)=0$. After this alteration we may ultimately define our function as $\check{u}(x) :=  \check{u}_1(x)$ for $x \in I_1$ and $\check{u}(x) :=  \check{u}_2(x)$ for $x \in I_2$. To be successful it suffices to prove that indeed such function $\check{u}$ is in $\Hnuktwo$: if so then necessarily $\tgradnuk{2}\check{u} = \check{v}_2 = -\check{M}/w$ and the boundary conditions in \eqref{eq:boundary_conditions} are already guaranteed to hold. For arbitrary $\eps>0$ we will point to a function $\check{u}^\eps\in \D(\R)$ with $\norm{\check{u}^\eps-\check{u}_1^\eps}_{H^{2,2}_{\mu_{w,1}}}+\norm{\check{u}^\eps-\check{u}_2^\eps}_{H^{2,2}_{\mu_{w,2}}} <\eps$ which approves of the thesis. Since $\check{u}_1 \in H^{2,2}_{\mu_{w,1}}$ and $x_2 \notin \cIcrl{0,2}(w)$, Theorem \ref{thm:trace_approximation} furnishes function $\check{u}^\eps_1 \in \D(\R)$ with $\check{u}^\eps_1(x_2)= \check{u}_1(x_2)$, $D \check{u}^\eps_1(x_2)= D\check{u}_1(x_2)$, $D^k \check{u}^\eps_1(x_2) = 0$ for $k\geq 2$ and yielding $\norm{\check{u}^\eps-\check{u}_1^\eps}_{H^{2,2}_{\mu_{w,1}}} \leq \eps/2$. Acknowledging $x_2 \notin \cIcrr{1,2}(w)$ and $x_2 \in \cIcrr{0,2}(w)$ the same theorem gives for the function $\check{u}_2 \in H^{2,2}_{\mu_{w,2}}$ an approximation $\check{u}^\eps_2 \in \D(\R)$ with $\check{u}^\eps_2(x_2)= \check{u}_2(x_2)$, $D \check{u}^\eps_2(x_2)= D \check{u}_1(x_2)$, $D^k \check{u}^\eps_2(x_2) = 0$ for $k\geq 2$ and $\norm{\check{u}^\eps-\check{u}_2^\eps}_{H^{2,2}_{\mu_{w,2}}} \leq \eps/2$. Since $\check{u}_2(x_2)$ was already ensured to equal $\check{u}_1(x_2)$, the two functions $\check{u}^\eps_1$ and $\check{u}^\eps_2$ glue smoothly at $x_2$ and hence the function $\check{u}_\eps \in \D(\R)$ is found.
	
	The unique solution $\check{u} \in \U^{2,2}_\muw$ is illustrated at the bottom of Fig. \ref{fig:solution_for_beams}(b). It is visible that the first derivative $\tgradnuk{}\check{u}$ blows up to infinity in vicinity of the right end $\ap$. This stems from the fact that $\tgradnuk{2}\check{u} \notin L^1\bigl(\Bl(\ap,\eps)\bigr)$ for any $\eps>0$, which was possible due to $\ap \in \cIcrl{0,2}(w)$.
	
	\vspace{3mm}
	
	\noindent \underline{Case (c):} The beam elasticity problem for the Case (b) was examined in a fair amount of detail. In this and the next case we shall omit or shorten the arguments that should run analogically to those already made beforehand. In particular the proof of Poincar\'{e} inequality \eqref{eq:general_Poincare} stays unaltered and thus existence of a solution $\check{u}$ is assured provided that the extension $\bar{f}$ exists and is orthogonal to $\U^{2,2}_{\muw,0}=\ker \tgradnuk{2} \cap \U^{2,2}_\muw$. 
	
	For the weight/width $w$ displayed in Fig. \ref{fig:solution_for_beams}(c) standard computations furnish that: $\cIcr{0,2}(w) = \{x_2\}$ and $\cIcr{1,2}(w) = \{x_2\}$, while $x_2$ is double-sided zero-order critical point, i.e. $x_2 \in \cIcrl{0,2}(w) \cap \cIcrr{0,2}(w)$; at the same time $x_2$ is right-sided, but not left sided first-order critical point, namely $x_2 \notin \cIcrl{1,2}(w)$. The embeddings at our disposal are as follows: $\Hnuktwo \hookrightarrow W^{2,1}\bigl( (\am, x_2 - \delta) \bigr),\ W^{2,1}\bigl( (x_2 + \delta,\ap) \bigr)$ for every $\delta \in (0,2)$. Moreover we have $\Hnuktwo \hookrightarrow W^{2,1}\bigl( (\am,x_2) \bigr)$ and $\Hnuktwo \hookrightarrow L^1(I)$, whereas the latter is due to $\cIcr{2,2}(w) = \varnothing$. Both the end-points $\am,\ap$ are not critical points of any order and characterization of the space of admissible displacements follows:
	\begin{equation}
	\label{eq:boundary_conditions_c}
	\U^{2,2}_\muw = \left\{ u \in \Hnuktwo \ : \Tr{}{\,u}(\am) =0, \ \Tr{1}{u}(\am) =0,\ \Tr{}{\,u}(\ap) =0, \ \Tr{1}{u}(\ap) =0 \right\},
	\end{equation}
	namely all the possible Dirichlet boundary conditions are maintained. Criticality $x_2 \in \cIcr{1,2}(w)$ decides through Corollary \ref{cor:step_function} that functions $\ustep_0 := \mathbbm{1}_{(x_2,\ap)}$ and $\ustep_1 = \ustep_1(x) := \int_{\am}^{x} \mathbbm{1}_{(x_2,\ap)} (y)\, dy$ belong to the space $\Hnuktwo$. In terms of structural mechanics the two halves of the beam are thus entirely disconnected and work independently.  We can write down the kernel of $\tgradnuk{2}$ as the four-dimensional space:
	\begin{equation}
	\label{eq:ker_Dmuw_c}
	\ker \tgradnuk{2} = \left\{ u_{P,1}\cdot \mathbbm{1}_{(\am,x_2)} + u_{P,2} \cdot \mathbbm{1}_{(x_2,\ap)} : u_{P,1}, \, u_{P,2} \in \mathrm{P}^1\right\}.
	\end{equation}
	By comparing \eqref{eq:boundary_conditions_c} and \eqref{eq:ker_Dmuw_c} once more we easily check that $\U^{2,2}_{\muw,0} = \{0\}$, which means that the solution of the elasticity problem \eqref{eq:relaxed_dirichlet_var} exists and is unique if and only if $f \in \D'(I)$ receives its continuous extension $\bar{f} \in \bigl( \Hnuktwo \bigr)^*$. This time around we find that the part $f_m = - m D(\delta_{x_2}) $ representing the point-moment is unbounded in $\Hnuktwo$ for non-zero $m\in\R$. We thus choose $m=0$ which finds its reflection in lack of the moment in Fig. \ref{fig:solution_for_beams}(c). This is a consequence of $x_2$ being a double-sided critical point: through Theorem \ref{thm:dirac_delta_approx_crit_point} we may find two sequences $\varphi^-_h \in \D\bigl( \Bl(x_2,1/h) \bigr)$ and $\varphi^+_h \in \D\bigl( \Br(x_2,1/h) \bigr)$ that can be twice integrated to $u^-_h$ and $u^+_h$, in a fashion from the proof of Corollary \ref{cor:step_function}, so that they converge in $\Hnuktwo$ to, respectively, $(\,\cdot\,-x_2) \, \mathbbm{1}_{(x_2,\ap)}$ and $ (x_2-\,\cdot\,) \, \mathbbm{1}_{(x_2,\ap)}$, whilst $D u^-_h(x_2) = 1$ and $D u^+_h(x_2) = 0$ for each $h$. As a result the sequence $u_h:=u^-_h+u^+_h \in \D(\R)$ converges to zero in $\Hnuktwo$ while $f_m(u_h) = -m D u_h (x_2) = -m$. Despite the fact that $x_2$ is also a first-order critical point the situation is different with the distribution $f_{F_2} = F_2 \,\delta_{x_2}$, since $x_2 \notin \cIcrl{1,2}(w)$ and, upon defining $v^*_2(x) = - F_2\cdot(x-x_2) / w(x)\, \mathbbm{1}_{(\am,x_2)}(x)$, we verify that $v^*_2 \in L^2_\muw$ and that $f_{F_2} = D^2(w\, v^*_2)$. We come to the same conclusion by recalling the embedding $\Hnuktwo \hookrightarrow W^{2,1}\bigl( (\am,x_2) \bigr)$. The rest of the components of $f$ extends continuously as well and the rest of the established embeddings may be utilized to prove it. Eventually, with $\breve{u}$ standing for the precise representative of $u \in \Hnuktwo$, we may characterize the extension of $f$ as
	\begin{equation*}
	\pairing{u,\bar{f}} = \sum_{i=1}^{4} \int_{i-1}^{i} q_i \,u \, dx + F_1 \, \breve{u}(x_1) + F_2 \, \breve{u}(x^-_2) + F_3 \, \breve{u}(x_3)
	\end{equation*}
	where $\breve{u}(x^-_2)$ denotes the left-sided limit of $\breve{u}$ at $x_2$. The values of parameters $q_i, F_j$ stays the same as before.
	
	Identically to Case (b) also here the dual problem \eqref{eq:dual_var_beam} is trivial as there is only one $\check{M} \in L^1(I)$ that satisfies the equilibrium constraint $D^2 \check{M} +f =0$ and produces finite energy $\Jc{I}(\check{M}) <\infty$. Indeed, all non-zero affine functions $M_0$ give $\Jc{I}(M_0) = \infty$ since $x_2 \in \cIcr{1,2}(w)$ and by definition $\int_I \abs{x-x_2}^2/w \, dx = \infty$. We easily find the solution $\check{M}$ and define $\check{v}_2 :=- \check{M}/w \in L^2_\muw$. On the separate intervals we define functions $\check{u}_1 = \check{u}_1(x) := \int_{\am}^{x} \check{v}_2(y)\, (x-y)\, dy$ for $x \in I_1$ and $\check{u}_2 = \check{u}_2(x) := \int_{\ap}^{x} \check{v}_2(y)\, (x-y)\, dy$ for $x \in I_2$ (note that in the second integral $x \leq \ap$ and thus we integrate backwards). The boundary conditions in \eqref{eq:boundary_conditions_c} are clearly met, while gluing is not necessary since the step function with the jump at $x_2$ is an element of $\Hnuktwo$. The final solution $\check{u} \in \U^{2,2}_\muw$, visible in Fig. \ref{fig:solution_for_beams}(c), is thus obtained as  $\check{u}(x) :=  \check{u}_1(x)$ for $x \in I_1$ and $\check{u}(x) :=  \check{u}_2(x)$ for $x \in I_2$. The proof that this function indeed lies in the space $\Hnuktwo$ runs analogically to the proof in Case \nolinebreak (b).
	
	\vspace{3mm}
	
	\noindent \underline{Case (d):} We move on to the last case where the width of the beam is the most degenerate: $\cIcr{0,2}(w)= \{ \am, x_2, \ap \}$, $\cIcr{1,2}(w) = \{x_2,\ap \}$, $\cIcr{2,2}(w) = \cIcrl{2,2}(w) = \{ x_2 \}$, while $x_2$ is a double-sided first order critical point, i.e. $x_2 \in \cIcrl{1,2}(w) \cap  \cIcrr{1,2}(w)$. We list the embeddings important for further considerations: $\Hnuktwo \hookrightarrow W^{2,1}_\mathrm{loc}(I_1), \  W^{2,1}_\mathrm{loc}(I_2), \ L^1\bigl( (\am,x_2- \nolinebreak \delta) \bigr), \ L^1\bigl(I_2\bigr)$, where $\delta$ is any number from $(0,2)$. The space $\Hnuktwo$ contains the same functions singular at $x_2$ as in Case (c): $\ustep_0 := \mathbbm{1}_{(x_2,\ap)}$ and $\ustep_1 = \ustep_1(x) := \int_{\am}^{x} \mathbbm{1}_{(x_2,\ap)}(y) \, dy$; in particular the kernel of $\tgradnuk{2}$ is identical to \eqref{eq:ker_Dmuw_c}. According to Proposition \ref{prop:trace_extension} the degeneracies at the end points $\am, \ap$ furnish the space of admissible displacements:
	\begin{equation}
	\label{eq:boundary_conditions_d}
	\U^{2,2}_\muw = \left\{ u \in \Hnuktwo \ : \Tr{}{\,u}(\am) =0 \right\}.
	\end{equation}
	The loss of almost all Dirichlet boundary conditions combined with the four dimensional kernel of $\tgradnuk{2}$ results in a non-trivial, three dimensional space of the so-called zero-energetic displacements:
	\begin{equation*}
	\U^{2,2}_{\muw,0} =  \left\{ u_{P,1}\cdot \mathbbm{1}_{(\am,x_2)} + u_{P,2} \cdot \mathbbm{1}_{(x_2,\ap)} : u_{P,1}, \, u_{P,2} \in \mathrm{P}^1, \ u_{P,1}(\am)=0 \right\}.
	\end{equation*} 
	The Poincar\'{e} inequality is still valid and the elasticity problem \eqref{eq:relaxed_dirichlet_var} will have a solution provided $f$ has an extension $\bar{f}$ and, moreover, $\pairing{u_0,\bar{f}} = 0$ for any $u_0 \in \U^{2,2}_{\muw,0}$. Once more the distribution $f_m$ cannot be extended to $\bigl(\Hnuktwo \bigr)^*$ and precisely for the same reason as in Case (b). Here a similar argument allows us to conclude that also $f_{F_2} = F_2 \, \delta_{x_2}$ is unbounded with respect to $\Hnuktwo$-norm: since $x_2$ is a double sided first-order critical point, Theorem \ref{thm:dirac_delta_approx_crit_point} guarantees a sequence $u_h \in \D(\R)$ converging to zero in $\Hnuktwo$ and attaining $u_h(x_2) = 1$ for each $h$. Further, the fact that $x_2 \in \cIcrl{2,2}(x_2)$ allows us to produce another sequence $u_h \rightarrow 0$ in $\Hnuktwo$ with $\int_{\Bl(x_2,1/h)} u_h \,dx = 1$: this way we must also eliminate the component $f_{q_2}$. The remaining part of the distribution $f$ extends continuously for any $q_1,q_3,q_4 \in \R$ and $F_1,F_3 \in \R$ to
	\begin{equation*}
	\pairing{u,\bar{f}} = \int_{\am}^{x_1} q_1 \,u \, dx+\int_{x_2}^{x_3} q_3 \,u \, dx+\int_{x_3}^{\ap} q_4 \,u \, dx + F_1 \, \breve{u}(x_1) + F_3 \, \breve{u}(x_3),
	\end{equation*}
	where $\breve{u}$ is the precise representative of $u \in \Hnuktwo$. For a solution $\check{u}$ to exist we must now choose the load parameters so that $\bar{f}$ is orthogonal to $\U^{2,2}_{\muw,0}$. Easy computation implies that this is true if and only if the following relations hold: $F_1 = - q_1 \cdot \abs{\am-x_1}/2 = -q_1/2$, $q_3 = q_4$ and $F_3 = - q_3\cdot \nolinebreak \abs{x_2-x_3} - \nolinebreak q_4 \abs{x_3-\ap} = - 2 q_3$. Those relations are in fact equilibrium equations written down for the two independent beams occupying intervals $I_1$ and $I_2$. Eventually we pick $q_1 = q_3 = q_4 =1$ and $F_1 = -1/2$, $F_3 = -2$; we observe that the point forces $F_1$ and $F_3$ are therefore pointed upwards which, together with the absence of the unbounded components of $f$, is noted in Fig. \ref{fig:solution_for_beams}(d).
	
	As in the previous two cases the dual problem \eqref{eq:dual_var_beam} is trivial and there is only one candidate $\check{M} \in L^1(I)$ for the solution. Upon setting $\check{v}_2 := -\check{M}/w$ we may define locally in $I_1$ and $I_2$ two functions $\check{u}_1 \in W^{2,1}_\mathrm{loc}(I_1)$ and $\check{u}_2 \in W^{2,1}_\mathrm{loc}(I_2)$ by: $\check{u}_1(x) := \int_{x_1}^x \check{v}_2(y) \, (x-y) \, dy$ for $x \in I_1$ and $\check{u}_2(x) := \int_{x_3}^x \check{v}_2(y) \, (x-y) \, dy$ for $x \in I_2$. As before we may prove that those functions are elements of $H^{2,2}_{\mu_{w,1}}$ and $H^{2,2}_{\mu_{w,2}}$ respectively. Since $\am \notin \cIcr{1,2}(w)$ Proposition \ref{prop:trace_extension} makes the boundary value $\check{u}_1(\am) \in \R$ meaningful and thus we may shift the function $\check{u}_1$ so that  $\check{u}_1(\am) = 0$. Then the same arguments as before prove that the function $\check{u}$, defined separately by $\check{u}_1$ and $\check{u}_2$ on the two intervals, is an element of $\U^{2,2}_\muw$ and constitutes a solution of the elasticity problem \eqref{eq:relaxed_dirichlet_var}. We stress that this solution is determined up to a zero-energetic displacement function $u_0$ being any element from the three-dimensional space $\U^{2,2}_{\muw,0}$, we capture this in Fig. \ref{fig:solution_for_beams}(d) where the right half of the beam seems to "float", while the first  arbitrarily rotates about the point $\am$. It is also worth noting that the displacement function $\check{u}$ blows up in the right-sided neighbourhood of the centre point $x_2$.

	\bigskip
	\footnotesize
	\noindent\textbf{Acknowledgments.}
	The paper was prepared within the Research Grant no 2015/19/N/ST8/00474 financed by the National Science Centre (Poland), entitled: Topology optimization of thin elastic shells - a method synthesizing shape and free material design.
	
	I would like to thank the Faculty of Mathematics, Informatics and Mechanics, University of Warsaw for its kind hospitality in the course of this research. I also wish to express my deepest appreciation for the invaluable guidance of my PhD supervisors Professor Piotr Rybka and Professor Tomasz Lewi\'{n}ski.

\end{document}